\documentclass[12pt]{amsart}

\usepackage{amssymb, a4wide, mathdots,url}
\usepackage{verbatim}
\usepackage{enumerate}
\usepackage[usenames]{color}
\usepackage[all]{xy}
\usepackage[colorlinks,pagebackref=true]{hyperref}
\hypersetup{colorlinks,linkcolor={blue},citecolor={blue},urlcolor={red}}

\providecommand{\germ}{\mathfrak}

\newcommand{\xib}{\xi_{{\rm bc}}}

\newcommand{\supp}{{\rm Supp}}
\newcommand{\ind}{{\rm ind}}

\newcommand{\Sp}{\mathrm{Sp}}

\newcommand{\cO}{{\mathcal O}}

\newcommand{\inv}{\theta}

\newcommand{\Cusp}{{\rm Cusp}}

\newcommand{\Supp}{\mathrm{Supp}\,}

\newcommand{\jm}{{\bf r}}
\newcommand{\ip}{{\bf i}}

\def\Irr{{\rm Irr}}

\def\Jord{{\rm Jord}}
\def\Res{{\rm Res}}
\def\GL{{\rm GL}}
\def\Alg{\Pi}

\def\mult{\mathfrak m}

\def\disc{{\rm disc}}

\newcommand{\C}{\mathbb{C}}
\newcommand{\Z}{\mathbb{Z}}

\newcommand{\R}{\mathbb{R}}
\newcommand{\N}{\mathbb{N}}
\newcommand{\NN}{\mathcal{N}}

\newcommand{\mbf}[1]{{\mathbf{#1}}}
\newcommand{\Stab}{\operatorname{Stab}}
\newcommand{\e}{\mbf{e}}
\newcommand{\f}{\mbf{f}}

\newcommand{\J}{\mathcal{J}}                     
\newcommand{\OO}{\mathcal{O}}              
\newcommand{\Ind}{{\rm Ind}}

\renewcommand{\subset}{\subseteq}

\newcommand{\bs}{\backslash}

\newcommand{\WR}{\mathcal{W}}

\newcommand{\diag}{\operatorname{diag}}

\newcommand{\LQ}{\operatorname{LQ}}
\newcommand{\U}{{\operatorname{U}}}
\newcommand{\rec}{{\operatorname{rec}}}

\newcommand{\cusp}{{\operatorname{cusp}}}
\newcommand{\temp}{{\operatorname{temp}}}

\newcommand{\abs}[1]{\left|{#1}\right|}

\newcommand{\triv}{{\bf 1}}

\newcommand{\inj}{\iota}

\newcommand{\sprod}[2]{\left\langle#1,#2\right\rangle}

\newcommand{\set}{{\mathfrak{c}}}

\newcommand{\grph}{{\mathfrak G}}
\newcommand{\weyl}{{\mathfrak W}}
\newcommand{\proj}{\operatorname{p}}
\renewcommand{\SS}{\mathcal{S}}
\newcommand{\HH}{\mathcal{H}}
\newcommand{\Index}{\mathcal{I}}
\newcommand{\Hom}{\operatorname{Hom}}
\newcommand{\E}{\mathcal{E}}

\renewcommand{\mod}{{\operatorname{mod}}}
\newcommand{\ds}{\operatorname{ds}}

\newcommand{\go}[1]{\overset{#1}\longrightarrow}
\renewcommand{\path}[1]{\overset{#1}\curvearrowright}

\newcommand{\sm}[4]{{\bigl(\begin{smallmatrix}{#1}&{#2}\\{#3}&{#4}
\end{smallmatrix}\bigr)}}

\newtheorem{theorem}{Theorem}
\newtheorem*{theorem*}{Theorem}

\newtheorem{claim}{Claim}
\newtheorem{lemma}{Lemma}
\newtheorem{conjecture}{Conjecture}

\newtheorem{proposition}{Proposition}
\newtheorem{corollary}{Corollary}

\theoremstyle{definition}
\newtheorem{definition}{Definition}
\newtheorem{example}{Example}
\theoremstyle{remark}
\newtheorem{remark}{Remark}

\title{On $\Sp$-distinguished representations of the quasi-split unitary groups}

\author{A Mitra}
\email{00.arnab.mitra@gmail.com}

\author{Omer Offen}
\email{offen@brandeis.edu}

\thanks{Arnab Mitra, partially supported by postdoctoral fellowships funded by the Department of Mathematics, Technion.}
\date{\today}

\begin{document}

\setcounter{tocdepth}{1}
\date{\today}

\begin{abstract}
We study $\Sp_{2n}(F)$-distinction for representations of the quasi-split unitary group $U_{2n}(E/F)$ in $2n$ variables with respect to a quadratic extension $E/F$ of $p$-adic fields.
A conjecture of Dijols and Prasad predicts that no tempered representation is distinguished. We verify this for a large family of representations in terms of the M{\oe}glin-Tadi\'c classification of the discrete series.
We further study distinction for some families of non-tempered representations. In particular, we exhibit $L$-packets with no distinguished members that transfer under stable base change to $\Sp_{2n}(E)$-distinguished representations of $\GL_{2n}(E)$.
\end{abstract}

\maketitle
\tableofcontents

\section{Introduction}\label{s_intro}
Let $G$ be a $p$-adic reductive group and $H$ a closed subgroup. A smooth, complex valued representation $(\pi,V)$ of $G$ (henceforth simply a representation $\pi$ of $G$) is called $H$-distinguished if there exists a non-zero linear functional $\ell$ on $V$ such that 
\[
\ell(\pi(h)v)=\ell(v),\ \ \ h\in H,\ \ \ v\in V.
\]
Distinguished representations play a central role in harmonic analysis of homogeneous spaces (see  for instance \cite{MR1075727}) and in the study of period integrals of automorphic forms, special values of $L$-functions and characterization of the image of a functorial transfer. 
Inspired by the work of Sakellaridis and Venkatesh \cite{MR3764130}, much attention is given to the case where $G/H$ is spherical and in particular, if $G/H$ is a $p$-adic symmetric space.

In this work we focus on the symmetric space $G/H$ where $G=U_{2n}=U_{2n}(E/F)$ is the quasi-split unitary group in $2n$ variables with respect to a quadratic extension $E/F$ of $p$-adic fields and $H=\Sp_{2n}(F)$.
The following is conjectured in \cite[Conjecture 2]{DP}.
\begin{conjecture}[Dijol-Prasad]\label{con dpi}
 An $L$-packet of irreducible representations of $U_{2n}$ that is associated to an Arthur parameter contains an $\Sp_{2n}(F)$-distinguished member if and only if its stable base change is an irreducible representation of $\GL_{2n}(E)$ that is $\Sp_{2n}(E)$-distinguished. 
\end{conjecture}
We remark that stable base change is a functorial transfer (in particular, a finite to one map) from irreducible representations of $U_{2n}$ to irreducible representations of $\GL_{2n}(E)$ that takes tempered representations to tempered representations (see \cite{MR3338302}).
The fibers of the stable base change map form the $L$-packets of irreducible representations of $U_{2n}$.
 
In particular, it is expected that no tempered representation of $U_{2n}$ is $\Sp_{2n}(F)$-distinguished. For cuspidal representations this is proved independently in \cite{DP} and in \cite{MR3590280}. In fact, more generally, if $\pi$ is an irreducible representation of $U_{2n}$ with non-trivial partial cuspidal support then $\pi$ is not $\Sp_{2n}(F)$-distinguished. (To say that $\pi$ has trivial partial cuspidal support means that for a realization of $\pi$ as a subrepresentation of a representation induced from cuspidal representations on a standard parabolic subgroup, this parabolic subgroup is contained in the Siegel parabolic subgroup.)

Since a similar vanishing result holds for period integrals of cuspidal automorphic representations, a natural approach to local non-distinction, at least for the discrete series, is a globalization argument of invariant linear forms as in \cite[\S 16]{MR3764130}. It would be interesting to see if invariant linear forms on discrete series representations can be globalized in our context (i.e., realized as local components of the period integral of some cuspidal automorphic representation). However, we do not know how to generalize the argument of loc. cit. to our setting. Instead, we apply the M{\oe}glin-Tadi\'c classification of discrete series \cite{MR1896238} and show that many discrete series (and tempered non-discrete) representations of $U_{2n}$ are not $\Sp_{2n}(F)$-distinguished. Exhausting the entire tempered spectrum, however, will require different methods. 

We further study distinction for an interesting family of non-tempered representations, namely the stable base change fiber of the class of Speh representations of $\GL_{2n}(E)$. Recall that Speh representations are the building blocks of the unitary dual of general linear groups \cite[Theorem D]{MR870688}. In particular, we exhibit $L$-packets of irreducible representations of $U_{2n}$ with no $\Sp_{2n}(F)$-distinguished member that transfer under stable base change to an irreducible representation of $\GL_{2n}(E)$ that is $\Sp_{2n}(E)$-distinguished. This does not contradict Conjecture \ref{con dpi} since those $L$-packets admit no Arthur parameters.

We now formulate the main results of this work. 
M{\oe}glin and Tadi\'c classified in \cite{MR1896238} the discrete series representations of classical groups in terms of certain combinatorial data. In particular, to an irreducible discrete series representation $\pi$ of $U_{2n}$ they attach a triple $(\pi_\cusp, \Jord(\pi),\epsilon_\pi)$. If the partial cuspidal support $\pi_\cusp$ of $\pi$ is non-trivial, as remarked above, $\pi$ is not $\Sp_{2n}(F)$-distinguished. We therefore focus our attention on representations $\pi$ with trivial partial cuspidal support. For the definition of the other components of the triple in this case see \S \ref{ss_dsc}. We only remark here that $\Jord(\pi)$ stands for a certain finite set of pairs $(\rho,a)$ attached to $\pi$ where $\rho$ is a conjugate (with respect to $E/F$) self-dual cuspidal representation of $\GL_m(E)$ for some $m\in \N$ and $a\in \N$ and with our assumption that $\pi$ has trivial partial cuspidal support we may think of $\epsilon_\pi$ as a function from $\Jord(\pi)$ to $\{\pm 1\}$.
\begin{theorem}\label{main intro}
Let $\pi$ be an irreducible discrete series representation of $U_{2n}$ with trivial partial cuspidal support. Assume that there exists $\rho$ such that the set 

\[
\Jord_\rho(\pi)=\{a\in \N: (\rho,a)\in \Jord(\pi)\}
\] 
is not empty and at least one of the following three properties holds: 
\begin{enumerate}
\item $t$ is odd;
\item $x_{2i-1}>x_{2i}+2$ for some $i\le t/2$;
\item $\epsilon_\pi(\rho,x_{t+2})=\epsilon_\pi(\rho,x_{t+1})$.
\end{enumerate}
Here, we write $\Jord_\rho(\pi)=(x_1,\dots,x_k)$ where $x_1>\cdots>x_k$, and let $t=1$ if $k=1$, and $t<k$ be such that $\epsilon_\pi(\rho,x_i)\ne \epsilon_\pi(\rho,x_{i+1})$, $i=1,\dots,t-1$ and $\epsilon_\pi(\rho,x_t)=\epsilon_\pi(\rho,x_{t+1})$, otherwise. Then $\pi$ is not $\Sp_{2n}(F)$-distinguished.
\end{theorem}
As explained in \S \ref{sss tpl}, the $L$-packet of an irreducible discrete series representation $\pi$ of $U_{2n}$ is determined by $\Jord(\pi)$. We emphasize that in the above theorem condition (2) gives non-distinction for entire $L$-packets.

We further remark that in light of \cite[Proposition 3.1]{MR2581039} we obtain, in particular, that no irreducible generic discrete series representation of $\U_{2n}$ is $\Sp_{2n}(F)$-distinguished and the same holds for the strongly positive discrete series.
(In fact, a strongly positive discrete series with generic partial cuspidal support is generic.) However, it was proved in \cite{MR1447416} more generally that no generic irreducible representation of $U_{2n}$ is $\Sp_{2n}(F)$-distinguished applying standard techniques of invariant distributions.

To formulate the next result we recall the definition of a Speh representation. Let $d,\,m\in \N$, $n=dm$ and $\delta$ an irreducible essentially square integrable representation of $\GL_d(E)$. The Speh representation $U(\delta,m)$ is the Langlands quotient of the representation of $\GL_n(E)$ parabolically induced from
\[
\abs{\det}^{\frac{m-1}2}\delta \otimes \abs{\det}^{\frac{m-3}2}\delta\otimes\cdots\otimes\abs{\det}^{\frac{1-m}2}\delta.
\]
We recall that when $n$ is even, the Speh representation $U(\delta,m)$ is $\Sp_n(F)$-distinguished if and only if $m$ is even (see for instance \cite[Theorem 10.3]{MR3628792}). 

For a representation $\pi$ of $\GL_n(E)$ we view $\pi$ as a representation of the Siegel Levi subgroup of $U_{2n}$ and denote by $\pi\rtimes \triv_0$ the representation of $U_{2n}$ parabolically induced from $\pi$.

We remark that if an irreducible representation of $\GL_{2n}(E)$ is in the image of stable base change then it is, in particular, conjugate self-dual.
\begin{theorem}\label{speh intro}
Let $\pi=U(\delta,m)$ be a Speh representation of $\GL_{2n}(E)$ that is conjugate self-dual. 
\begin{enumerate}
\item If $m$ is odd (i.e., if $\pi$ is not $\Sp_{2n}(E)$-distinguished) then $\pi$ is in the image of stable base change if and only if $\delta$ is so (if and only if the Langlands parameter of $\delta$ is conjugate-symplectic in the sense of \cite[\S3]{MR3202556}). In this case, there is a unique irreducible representation of $U_{2n}$ that transfers to $\pi$ under stable base change and it is not $\Sp_{2n}(F)$-distinguished.
\item If $m$ is even (i.e., if $\pi$ is $\Sp_{2n}(E)$-distinguished) then there is a unique representation $\pi_0$ of $U_{2n}$ that transfers to $\pi$ under stable base change. Suppose further that $\pi\rtimes \triv_0$ is irreducible. Then, $\pi_0$ is $\Sp_{2n}(F)$-distinguished if and only if $m$ is divisible by $4$.
 
\end{enumerate}
\end{theorem}
The condition that $\pi\rtimes \triv_0$ is irreducible is explicated in Corollary \ref{corr:irrd}. A Speh representation is a special case of a {\it ladder} representation (see \S \ref{ss:ladder}). In fact, Proposition \ref{prop:bc_dist} is a generalization of Theorem \ref{speh intro} (2) to ladder representations.

Since any irreducible representation of $U_{2n}$ occurs as a unique irreducible quotient of its standard module (see Theorem \ref{thm_lqt}), a step towards classification of the distinguished smooth dual, is the classification of distinguished standard modules. We treat the special case where the standard module is induced from a Levi subgroup of the Siegel Levi subgroup of $U_{2n}$ and the $\GL$-part is irreducible. 
\begin{theorem}\label{thm:gen_lq_int}
Let $\delta_i$ be an irreducible essentially square integrable representation of $\GL_{n_i}(E)$, $i=1,\dots,t$ such that $\exp(\delta_1)\ge \cdots\ge \exp(\delta_t)>0$ (see \S \ref{sss exp}) and the representation $\pi$ of $\GL_n(E)$ (with $n=n_1+\cdots+n_t$) parabolically induced from $\delta_1\otimes\cdots\otimes\delta_t$ is irreducible. Then the standard module $\pi\rtimes\triv_0$ of $U_{2n}$ is $\Sp_{2n}(F)$-distinguished if and only if the irreducible generic representation $\nu^{-1/2}\pi$ of $\GL_n(E)$ is $\GL_n(F)$-distinguished. When this is the case, if in addition $\exp(\delta_i)\in\frac{1}{2}\mathbb Z$, $i=1,\dots,t$, then $\nu^{-1/2}\pi$ is tempered.
\end{theorem}

Next we outline the structure of the paper and the methods we use to prove the above theorems.
An explicit version of the geometric lemma \cite[Theorem 5.2]{MR0579172} (adapted in \cite{MR3541705} to restriction to $H$ of parabolically induced representations of $G$ when $G/H$ is a $p$-adic symmetric space) plays a key role in the proof of all the above mentioned results. 

After setting up the notation in \S \ref{s_notn} and recalling some preliminary results in \S \ref{s pre} we explicate in \S \ref{s_geo_lem} certain necessary conditions for a parabolically induced representation of $U_{2n}$ to be $\Sp_{2n}(F)$-distinguished. 

In \S \ref{s_van_ds} we prove Theorem \ref{main intro}. More precisely, we prove in Theorem \ref{thm:main_res_ds} that if $\pi$ is a representation of $U_{2n}$ that satisfies the hypothesis of Theorem \ref{main intro} then its contragredient $\pi^\vee$ is not $\Sp_{2n}(F)$-distinguished. However, based on the realization of the contragredient via the principal involution of $U_{2n}$ \cite[\S 4, II.1]{MR1041060}, we observe in Corollary \ref{cor cont} that an irreducible representation $\pi$ of $U_{2n}$ is $\Sp_{2n}(F)$-distinguished if and only if $\pi^\vee$ is so. Theorem \ref{main intro} immediately follows.

Our proof of Theorem \ref{thm:main_res_ds} is based on the results of \cite{MR1896238} that we recall in \S \ref{ss_dsc}. For an irreducible discrete series representation of $U_{2n}$, M{\oe}glin and Tadi\'c provide a recipe to read off the triple $(\pi_\cusp,\Jord(\pi),\epsilon_\pi)$ in different ways to realize $\pi$ as a subrepresentation of a certain representation parabolically induced from an essentially square integrable representation of a Levi subgroup. 

For distinction problems, it is more convenient to realize $\pi$ as a quotient. Hence, we dualize. It is also important to consider the possible different realizations (of $\pi^\vee$ as a quotient). We prove that $\pi^\vee$ is not $\Sp_{2n}(F)$-distinguished by proving that it is the quotient of some induced representation that is not $\Sp_{2n}(F)$-distinguished. For this sake, we study in \S \ref{app_s:signs} a graph that underlies the combinatorics behind $\epsilon_\pi$ and the different realizations of $\pi$ in induced representations.
This section is purely combinatorial and independent of the rest of this work. We hope that it will be useful for other problems where explicit realizations of discrete series representations of classical groups play a role.

In \S \ref{s_van_temp} we treat tempered representations that are not in the discrete series. In Proposition \ref{prop_van_temp} we prove that many tempered irreducible representations of $U_{2n}$ are not $\Sp_{2n}(F)$-distinguished. 
Our treatment is based on the description of the tempered spectrum in \cite[\S 13]{MR1896238}. We remark, that in any case treated in Proposition \ref{prop_van_temp} non-distinction is proves for an entire $L$-packet of tempered representations. This observation is based on the description of tempered $L$-packets that we recall in \S \ref{sss tpl} and  \S \ref{sss tlp}.

In order to address the other results of the paper we recall some further preliminaries. In \S \ref{s_rep_gln} we recall the representation theory of $\GL_n(E)$, Zelevinsky's segment notation, Langlands parameters and the reciprocity map. In \S \ref{s_geo_app} we deduce further applications of the geometric lemma that are applied in the sequel. In particular, we provide in Lemma \ref{hered} a sufficient condition for a representation of $U_{2n}$ to be $\Sp_{2n}(F)$-distinguished. 
In \S \ref{s_rep_u_n} we recall Mok's reciprocity map and the properties of the stable base change transfer. 

We prove Theorem \ref{speh intro} in \S \ref{s_stab_dist} (see Theorem \ref{thm:speh_bc}). It is an immediate consequence of Propositions \ref{prop:bc_dist} and \ref{prop:odd2} that are formulated more generally in the language of ladder representations. Finally, in \S \ref{s_std_mod} we prove Theorem \ref{thm:gen_lq_int} (see Theorem \ref{thm:gen_lq} and Remark \ref{rem:gen_lq}).

\begin{remark}
Throughout the paper we assume that $F$ has characteristic zero. This assumption is only used in some of the results we rely on, namely the M{\oe}glin-Tadi\'c classification (and in particular, their basic assumption) \cite{MR1896238}, the structure of the discrete $L$-packet \cite{MR2366373} and Mok's reciprocity map \cite{MR3338302}.
\end{remark}

\subsection*{Acknowledgments}
The authors wish to thank K. Morimoto for sharing his preprint \cite{Mor} with them and Dipendra Prasad for answering their many questions. The first author also wishes to thank Sandeep Varma for several helpful conversations and the Tata Institute of Fundamental Research, Mumbai, for providing an effective environment in which a part of this work was done. 

\section{Notation}\label{s_notn}

We set the general notation in this section. More particular notation is defined in the section where it first occurs.

\subsection{The relevant groups}
\subsubsection{}
Let $F$ be a non-archimedean local field of characteristic zero with normalized absolute value $\abs{\cdot}_F$ and $E/F$ a quadratic extension.  Thus, 
\[
\abs{a}_E=\abs{a}_F^2,\, a\in F. 
\]
We denote by $a\mapsto \bar a$ the action of non-trivial element of ${\rm Gal}(E/F)$. 

\subsubsection{}
We use bold letters such as $\mbf{X}$ to denote varieties defined over $F$ and the corresponding usual font to denote the topological space of its rational points. That is, $X=\mbf{X}(F)$. 

\subsubsection{}
For a variety $\mbf{Y}$ defined over $E$ let $\mbf{Res_{E/F} (Y)}$ denote its restriction of scalars from $E$ to $F$.

\subsubsection{}
Let $\mathbf{U_{2n}}=\mathbf{U_{2n}(E/F)}$ be the quasi-split unitary group in $2n$ variables associated to the anti-hermitian form defined by 
\[
J_n=\begin{pmatrix}
            & w_n \\
-w_n   &
   \end{pmatrix}
\]
where $w_n=(\delta_{i,n+1-j})\in \GL_n$.
Explicitly,
\[
\mathbf U_{2n}= \{g\in \mbf{Res_{E/F}((GL_{2n})_E)}\mid {}^{t}\bar{g} \,J_n\, g=J_n\}.
\]

\subsubsection{}
Let $\mathbf{Sp_{2n}}$ denote the symplectic group of rank $n$ defined by the skew-symmetric matrix $J_{n}$, i.e.,
\[
\mathbf{Sp_{2n}}=\{g\in \mathbf{GL_{2n}} \mid {}^t g J_{2n} g=J_{2n}\}.
\]   
Note that 
\[
\mathbf{Sp_{2n}}=\{g\in \mathbf{U_{2n}}\mid \bar g=g\}.
\]

\subsubsection{}\label{sub par} Standard parabolic subgroups of $\mbf{U_{2n}}$ are in bijection with tuples $\alpha=(n_1,\dots,n_r;m)$ such that $n_1+\cdots+n_r+m=n$, $r,\,m\in \Z_{\ge 0}$ and $n_1,\dots,n_r\in \N$. The standard parabolic subgroup $\mbf{Q_\alpha}=\mbf{L_\alpha}\ltimes \mbf{V_\alpha}$ with standard Levi subgroup $\mbf{L_\alpha}$ and unipotent radical $\mbf{V_\alpha}$ is the subgroup consisting of block upper triangular matrices in $\mbf{U_{2n}}$ so that
\[
\mbf{L_\alpha}=\{\diag(g_1,\dots,g_r,h,g_r^*,\dots,g_1^*)\mid g_i\in \mbf{Res_{E/F}((GL_{n_i})_E)},\,i=1,\dots,r,\ h\in \mbf{U_{2m}}\}.
\]
Here $g\mapsto g^*$ is the involution on $\mbf{Res_{E/F}((GL_k)_E)}$ defined for any $k\in \N$ by $g^*=w_k {}^t\bar g^{-1} w_k$. 
In particular, $\mbf{U_{2n}}=\mbf{Q_{(\ ;n)}}=\mbf{L_{(\ ;n)}}$ corresponds to $r=0$ and $m=n$ while $\mbf{Q_{(n;0)}}$ is the Siegel parabolic subgroup.

Let $\inj_\alpha: \mbf{\Res_{E/F}((\GL_{n_1})_E)} \times\cdots\times \mbf{\Res_{E/F}((\GL_{n_r})_E)} \times \mbf{U_{2m}}\rightarrow L_\alpha$ be the isomorphism defined by 
\[
\inj_\alpha(g_1,\dots,g_r,h)=\diag(g_1,\dots,g_r,h,g_r^*,\dots,g_1^*).
\]
If $m=0$ (i.e., if $\mbf{Q_\alpha}$ is contained in the Siegel parabolic subgroup) we simply write $\inj_\alpha(g_1,\dots,g_r)$ and when clear from the context we often omit the subscript $\alpha$.

\subsubsection{} 
If $\mbf{P}=\mbf{M}\ltimes \mbf{U}$ is a standard parabolic subgroup of $\mbf{U_{2n}}$ with its standard Levi decomposition, by taking rational points we refer to $P$ (resp. $M$) as a standard parabolic (reps. Levi) subgroup of $U_{2n}$.

\subsubsection{} Let 
\[
T=\{\diag(a_1,\dots,a_n,a_n^{-1},\dots,a_1^{-1}) \mid a_1,\dots,a_n\in F^*\}
\]
be the diagonal maximal split torus in $U_{2n}$. 

For a standard Levi subgroup $M$ of $U_{2n}$ let $R(M,T)$ be the corresponding root system and let $R^+(M,T)$ (resp. $\Delta^M$) be the set of positive (resp. simple) roots with respect to the standard Borel subgroup $M\cap Q_{(1^{(n)};0)}$ of $M$ consisting of upper triangular matrices in $M$. Here $1^{(n)}$ is the $n$-tuple of ones.

\subsubsection{} More generally, let $P=M\ltimes U\subseteq Q=L\ltimes V$ be standard parabolic subgroups of $U_{2n}$ with their standard Levi decomposition and $T_M$ the connected component of the center of $M$. We denote by $R(L,T_M)$ the roots of $L$ on $T_M$ and by $R^+(L,T_M)$ (resp. $\Delta_M^L$) the subset of positive (resp. simple) roots with respect to $L\cap P$. The set $\Delta_M^L$ consists of non-zero restrictions to $T_M$ of elements of $\Delta^L$. For $\alpha\in R(L,T_M)$ we write $\alpha>0$ if $\alpha\in R^+(L,T_M)$ and $\alpha<0$ otherwise. When $L=U_{2n}$ we also set $\Delta_M=\Delta_M^L$.

\subsubsection{} 
For a standard Levi subgroup $M$ of $U_{2n}$ let
\[
W^{M}=N_{M}(T)/C_{M}(T) 
\]
be the Weyl group of $M$. In particular, the Weyl group $W^{U_{2n}}$ of $U_{2n}$ is isomorphic to the signed permutation group in $n$ letters.

Note that $C_{M}(T)=L_{(1^{(n)};0)}$ is the standard Levi subgroup of the standard Borel subgroup of $U_{2n}$ corresponding to the decomposition $(1^{(n)};0)=(1,\dots,1;0)$ of $n$. Therefore, $W^M$ is a subgroup of $W^L$ whenever $M\subseteq L$ are standard Levi subgroups of $U_{2n}$.

There is a unique element of maximal length in the set of $w\in W^{L}$ satisfying
\begin{itemize}
\item $w$ is of minimal length in $wW^M$;
\item $wMw^{-1}$ is a standard Levi subgroup of $L$.
\end{itemize}
We denote this element by $w_M^L$. It also satisfies $w_M^L\alpha<0$ for $\alpha\in R^+(L,T_M)$.

\subsubsection{} 
Let $W=W^{U_{2n}}$. For standard Levi subgroups $M$ and $L$ of $U_{2n}$ any double coset in $W^M\bs W/W^L$ admits a unique element of minimal length. Denote by ${}_MW_L$ the set of $w\in W$ such that $w$ is of minimal length in $W^M wW^L$. The set ${}_MW_L$ consists of all left $M$-reduced and right $L$-reduced elements in $W$. That is, $w\alpha\in R^+(U_{2n},T)$ for all $\alpha\in \Delta^L$ and $w^{-1}\alpha\in R^+(U_{2n},T)$ for all $\alpha\in \Delta^M$. The inclusion ${}_MW_L\subseteq W$ gives a natural bijection
\[
{}_MW_L\simeq W^M\bs W/W^L.
\]

\subsubsection{} 

Let $P=M\ltimes U$ and $Q=L \ltimes V$ be standard parabolic subgroups of $U_{2n}$ with their standard Levi decompositions. The Bruhat decomposition is the bijection 
\[
w\mapsto PwQ : {}_MW_L  \overset{\sim}\longrightarrow  P\bs U_{2n}/Q.
\]

For every $w\in {}_MW_L$ the group 
\[
P(w)=M\cap wQw^{-1}
\]
is a standard parabolic subgroup of $M$ (that is, $P(w)\supseteq M\cap Q_{(1^{(n)};0)}$). Its standard  Levi decomposition is given by $P(w)=M(w)\ltimes U(w)$, where
\[
M(w)=M\cap wLw^{-1}\ \ \ \text{and}\ \ \ U(w)=M\cap wVw^{-1}.
\]

\subsection{Representations}
Let $\mbf G$ be a linear algebraic group defined over $F$. 
By a representation of $G$ we always mean a smooth complex valued representation.

Let $\Alg(G)$ be the subcategory of representations of $G$ of finite length and $\Irr(G)$ the class of irreducible representations in $\Alg(G)$. 

If in addition $\mbf{G}$ is a connected reductive group let $\Cusp(G)$ be the set of all cuspidal representations in $\Irr(G)$.  

Let $\pi^\vee$ denote the contragredient of a representation $\pi\in \Alg(G)$. Then $(\pi^\vee)^\vee\cong\pi$ and $\pi\in \Irr(G)$ if and only if $\pi^\vee\in \Irr(G)$. 

\subsubsection{}
Let $\delta_G$ be the modulus function of $G$ with the convention that $\delta_G(g)dg$ is a right-invariant Haar measure if $dg$ is a left-invariant Haar measure on $G$.

\subsubsection{}
Let $\mbf{H}$ be a subgroup of $\mbf{G}$ and $\sigma$ a smooth, complex-valued representation of $H$.
We denote by ${\rm Ind}_H^G(\sigma)$ the normalized induced representation. It is the representation of $G$ by right translations on the space of functions $f$ from $G$ to the space of $\sigma$ satisfying 
\[
f(hg)=(\delta_H\delta_G^{-1})^{1/2}(h)\sigma(h)f(g), \ \ \ h\in H,\,g\in G
\] 
and $f$ is right invariant by some open subgroup of $G$.

The representation of $G$ on the subspace of functions with compact support modulo $H$ is denoted by ${\rm ind}_H^G(\sigma)$.
\subsubsection{}\label{sec pi}

Let $\mbf{P}=\mbf{M}\ltimes \mbf{U}$ be a parabolic subgroup of $\mbf{G}$ with Levi part $\mbf{M}$ and unipotent radical $\mbf{U}$.
The functor $\ip_{G,M}:\Alg(M)\rightarrow\Alg(G)$ of normalized parabolic induction is defined as follows. For $\rho\in \Alg(M)$ we consider $\rho$ as a representation of $P$ by composing with the projection $P/U \rightarrow M$ and set
\[
\ip_{G,M}(\rho)=\Ind_P^G(\rho).
\]
The functor $\ip_{G,M}$ is exact and we have
\begin{equation}\label{eq cont ind}
\ip_{G,M}(\rho)^\vee\cong \ip_{G,M}(\rho^\vee).
\end{equation}

\subsubsection{}
In the notation of \S \ref{sec pi}, the functor $\ip_{G,M}$ admits a left adjoint, namely, the normalized Jacquet functor $\jm_{M,G}:\Alg(G)\rightarrow\Alg(M)$. For $\sigma\in \Alg(G)$, $\jm_{M,G}(\sigma)$ is the representation of $M$ on the space of $U$-coinvariants of $\sigma$ induced by the action $\delta_{P}^{-1/2}\sigma$. It is also an exact functor and for $\sigma\in \Alg(G)$ and $\rho\in \Alg(M)$ we have the natural linear isomorphism (Frobenius reciprocity):
\begin{equation}\label{eq: 1st adj}
\Hom_G(\sigma,\ip_{G,M}(\rho))\cong \Hom_M(\jm_{M,G}(\sigma),\rho).
\end{equation}

When $G$ is either $\GL_n(E)$ or $U_{2n}$ we will further use the following standard notation for parabolic induction.

\subsubsection{} Let $G=\GL_n(E)$, $n=n_1+\cdots+n_k$ and 
\[
M=\{\diag(g_1,\dots,g_k)\mid g_i\in \GL_{n_i}(E),\ i=1,\dots,k\}.
\]
For representations $\rho_i\in \Alg(\GL_{n_i}(E))$, $i=1,\dots,k$ let $\rho=\rho_1\otimes\cdots\otimes\rho_k\in \Alg(M)$ and
\[
\rho_1\times\cdots\times\rho_k=\ip_{G,M}(\rho).
\]

\subsubsection{} Let $G=U_{2n}$, $\alpha=(n_1,\dots,n_k;m)$ a composition of $n$ as in \S \ref{sub par} and $M=L_\alpha$. For representations $\rho_i\in \Alg(\GL_{n_i}(E))$, $i=1,\dots,k$  and $\sigma\in \Alg(U_{2m})$ let $\rho=\rho_1\otimes \cdots \otimes \rho_k\otimes\sigma\in \Alg(M)$ and
\[
\rho_1\times\cdots \times \rho_k\rtimes \sigma=\ip_{G,M}(\rho).
\]

\subsubsection{}
We state basic properties of parabolic induction for the group $U_{2n}$ which we will use several times in the sequel, often without further reference.
For their proof see for instance \cite[Propositions 4.1 and 4.2]{MR1266251}\footnote{Strictly speaking, this reference is only for symplectic groups but the proof works for every classical group. See also \cite{MR1896238}.}. Let $[\pi]$ denote the semi-simplification of a representation $\pi\in\Alg(U_{2n})$. 
\begin{proposition}\label{prop_basic_ind} 
Let $\pi_i\in \Alg(\GL_{n_i}(E))$, $i=1,2$ and $\sigma\in \Alg(U_{2n})$ for some $n_1,\,n_2,\,n\in \N$. We have 
\begin{enumerate}
\item $(\pi_1\times \pi_2)\rtimes \sigma\cong \pi_1\rtimes (\pi_2\rtimes \sigma)$,
\item $(\pi_1\rtimes \sigma)^{\vee}\cong \pi_1^{\vee}\rtimes \sigma^{\vee}$,
\item $[\overline{\pi_1}^{\vee}\rtimes \sigma]=[\pi_1\rtimes \sigma]$.
\end{enumerate} \qed
\end{proposition}

\subsubsection{}
This paper is concerned with distinction of representations in the following sense.
\begin{definition}
Let $\pi$ be a representation of $G$ and $\mbf{H}$ a subgroup of $\mbf{G}$. 
\begin{itemize}
\item We say that $\pi$ is \emph{$H$-distinguished} if there exists a non-zero $H$-invariant linear form $\ell$ on the space of $\pi$, i.e., $\ell(\pi(h)v)=\ell(v)$ for all $h\in H$ and $v$ in the space of $\pi$. We denote by $\Hom_H(\pi,1)$ the space of $H$-invariant linear forms on $\pi$. 
\item More generally, for a character $\chi$ of $H$ we say that $\pi$ is $(H,\chi)$-distinguished if the space $\Hom_H(\pi,\chi)$ of $H$-equivariant linear forms on $\pi$ is non-zero.
\end{itemize}
\end{definition}
By Frobenius reciprocity (\cite[Theorem 2.28]{MR0425030}) we have a natural linear isomorphism
\begin{equation}\label{eq: frob rec}
\Hom_H(\pi,\chi \delta_H^{1/2}\delta_G^{-1/2})\cong\Hom_G(\pi,\Ind_H^G(\chi)).
\end{equation}

\subsubsection{} We set some conventions for particular cases of distinction relevant to us.
We primarily consider $\Sp_{2n}(F)$-distinguished representations of $U_{2n}$ and $\Sp_{2n}(E)$-distinguished representations of $\GL_{2n}(E)$. In both cases, for simplicity, we say that the representation is $\Sp$-distinguished.   

Again for the sake of notational simplification, we say that a representation of $\GL_n(E)$ is $\GL(F)$-distinguished if it is $\GL_{n}(F)$-distinguished.

\section{Preliminaries}\label{s pre}

\subsection{Involutions on Weyl groups}\label{sec weyl}
Let $W$ be a Weyl group of a root system with a basis $\Delta$ of simple roots. Let $s_\alpha\in W$ be the simple reflection associated to $\alpha\in \Delta$.
Denote by $W[2]=\{w\in W:w^2=e\}$ the set of involutions in $W$. 

Based on Springer's treatment of involutions in \cite{MR803346} we define a directed $\Delta$-labeled graph $\grph_W$ with vertices $W[2]$ and labeled edges $w\xrightarrow{\alpha} w'$ whenever $w'=s_\alpha w s_\alpha$ and $w\alpha\ne \pm\alpha$.
Note that if $w\xrightarrow{\alpha} w'$ then also $w'\xrightarrow{\alpha} w$. If $w_0,\dots,w_k\in W[2]$ and $\alpha_1,\dots,\alpha_{k}\in\Delta$ are such that $w_{i-1} \xrightarrow{\alpha_i} w_i$, $i=1,\dots,k$ we also write $w_0  \overset{\sigma}\curvearrowright w_k$ with $\sigma=s_{\alpha_k}\cdots s_{\alpha_1}$ for a path on the graph (the notation suppresses the dependence on the chosen word for $\sigma$).
Note that if $w  \overset{\sigma}\curvearrowright w'$ then in particular $\sigma w\sigma^{-1}=w'$.

\begin{definition}
An involution $w\in W[2]$ is called minimal if it is the longest element of the Weyl group generated by $\{s_\alpha:\alpha\in\Pi\}$ for some subset $\Pi\subseteq\Delta$ and $w\alpha=-\alpha$ for $\alpha\in\Pi$.
\end{definition}
\begin{example}\label{exm s}
The minimal involutions in the symmetric group $S_n$ (with respect to the standard basis $\Delta=\{e_i-e_{i+1}: i=1,\dots,n-1\}$) are the products of disjoint simple reflections. That is, $w=s_{i_1}\cdots s_{i_k}$ where $s_i=(i,i+1)$ and $\{i_v,i_{v+1}\}$ is disjoint from $\{i_u,i_u+1\}$ for any $1\le u\ne v \le k$.
\end{example}
For the following result of Springer see \cite[Proposition 3.3]{MR803346}.
\begin{lemma}\label{lem springer}
For any $w\in W[2]$ there exists a minimal involution $w'\in W[2]$ and $\sigma\in W$ such that $w \overset{\sigma}\curvearrowright w'$.
\qed
\end{lemma}

Let $\weyl_n$ be the signed permutation group in $n$ variables. We realize it as
\[
\weyl=\weyl_n=S_n \ltimes \Xi_n
\]
where $\Xi_n$ is the group of subsets of $\{1,\dots,n\}$ with symmetric difference as multiplication. 
Clearly, $\Xi_n \cong (\Z/2\Z)^n$ and we may take $\sigma_i=\{i\}$, $i=1,\dots,n$ as generations of $\Xi_n$.

The action of $S_n$ on $\Xi_n$ is given by
\[
\tau \set \tau^{-1}=\tau(\set), \ \ \ \tau\in S_n,\ \set\in\Xi_n.
\]

It is easy to see that the set of involutions in $\weyl$ is
\[
\weyl[2]=\{\tau\set:\tau\in S_n[2],\,\tau(\set)=\set\}.
\]
We think of $\weyl$ as a Weyl group of a root system of type $C_n$ with the standard basis of simple roots $\Delta_n=\{\alpha_1,\dots,\alpha_n\}$ where $\alpha_i=e_i-e_{i+1}$, $i=1,\dots,n-1$ and $\alpha_n=2e_n$. 
Here $\weyl$ acts by 
\[
\tau(e_i)=e_{\tau(i)}, \ \ \ \tau\in S_n\ \ \ \text{and} \ \ \ \set(e_i)=\begin{cases} e_i & i\not\in \set \\ -e_i & i\in \set \end{cases}, \ \ \ \set\in\Xi_n,\ \ \ i=1,\dots,n.
\]      
Here we examine certain properties of the graph $\grph_\weyl$.
The set of minimal involutions in $\weyl$ is
\begin{equation}\label{eq min inv}
\{\rho\set_k: 0\le k\le n, \rho\text{ is a minimal involution of }S_k,\ \set_k=\set_{k,n}= \{k+1,\dots,n\}\}.
\end{equation}
Here $S_k$ is embedded in $S_n$ as the subgroup of permutations fixing $k+1,\dots,n$ point-wise. Also, by convention, $\set_{n,n}=\emptyset$.
Thus, the minimal involutions in $S_n$ are the minimal involutions of $\weyl$ associated to $k=n$.

Let $w=\tau\set\in \weyl$ with $\tau\in S_n$ and $\set\in \Xi_n$. Define the following subsets of $\{1,\dots,n\}$:
\begin{itemize}
\item $\set_+(w)=\{i\in \set\mid \tau(i)=i\}$;
\item $\set_-(w)=\{i\not\in \set\mid \tau(i)=i\}$;
\item $\set_{\ne}(w)=\{i\mid \tau(i)\ne i\}$;
\item $\set_{<}(w)=\{i\mid i<\tau(i)\}$.
\end{itemize}
Note that $\set_{\ne}(w)$ and $\set_{<}(w)$ depend only on $\tau$. 
\begin{example}\label{ex min} If $w=\rho\set_{k,n}\in \weyl$ is a minimal involution as in \eqref{eq min inv} then  according to Example \ref{exm s} we have $\set_+(w)=\set_{k,n}$ while $\{1,\dots,k\}$ is the disjoint union
\[
\{1,\dots,k\}=\set_-(w) \sqcup \sqcup_{i\in \set_{<}(w)} \{i,i+1\}
\]
and
\[
\rho=\prod_{i\in \set_{<}(w)} s_i.
\]
\end{example}
We make the following simple observations. 
\begin{lemma}\label{lem w inv}
Let $w,\,\sigma\in \weyl$ and $w'=\sigma w\sigma^{-1}$. Then 
\begin{enumerate}
\item\label{part +} $\sigma\set_+(w)\sigma^{-1}=\set_+(w')$ (in particular, $\abs{\set_+(w')}=\abs{\set_+(w)}$), 
\item\label{part -} $\sigma\set_-(w)\sigma^{-1}=\set_-(w')$ and 
\item\label{part ne} $\sigma\set_{\ne}(w)\sigma^{-1}=\set_{\ne}(w')$.
\end{enumerate}
On the left hand side of each equation multiplication is in $\weyl$.
\end{lemma}

\begin{proof}
Let $w=\tau\set$ and $w'=\tau'\set'$ with $\tau,\,\tau'\in S_n$ and $\set,\,\set'\in \Xi_n$. It is enough to prove the lemma for $\sigma$ in a set of generators for $\weyl$, hence in the two cases: $\sigma\in S_n$ and $\sigma\in \Xi_n$.

Assume first that $\sigma\in S_n$. Then $\tau'=\sigma\tau\sigma^{-1}$ and therefore $\set_{\ne}(w')=\sigma(\set_{\ne}(w))=\sigma\set_{\ne}(w)\sigma^{-1}$ (hence the equality \eqref{part ne}) and $\set_+(w')\cup \set_-(w')=\sigma(\mathfrak{d})=\sigma \mathfrak{d}\sigma^{-1}$ where $\mathfrak{d}=\set_+(w)\cup \set_-(w)=\{i\mid\tau(i)=i\}$.
Since furthermore, $\set'=\sigma(\set)$ and $\set_+(w)=\set\cap \mathfrak{d}$ the equalities \eqref{part +} and \eqref{part -} also follow.

Assume now that $\sigma\in \Xi_n$. Then $\tau'=\tau$ and it follows immediately that $\set_{\ne}(w')=\set_{\ne}(w)=\sigma\set_{\ne}(w)\sigma^{-1}$ (hence the equality \eqref{part ne}). Furthermore, $\set'=\tau^{-1}(\sigma)\sigma \set$. Since the product in $\Xi_n$ is the symmetric difference it is clear that the set $\tau^{-1}(\sigma)\sigma$ contains no fixed points of $\tau$. Therefore, $\set_+(w')=\set_+(w)=\sigma\set_+(w)\sigma^{-1}$ which gives the equality \eqref{part +}. The equality \eqref{part -} follows from \eqref{part ne} and \eqref{part +}.
\end{proof}

\begin{lemma}\label{lem <}
Let $\sigma\in \weyl$ and $w,\,w'\in \weyl[2]$ be such that $w\overset{\sigma}\curvearrowright w'$. Then
\[
\sigma \set_<(w)\sigma^{-1}=\set_<(w').
\]
\end{lemma}
\begin{proof}
It follows from Lemma \ref{lem w inv} \eqref{part ne} that the two sets have the same cardinality and it is therefore enough to show that $\sigma \set_<(w)\sigma^{-1}\subseteq\set_<(w')$. By induction on the length of the path defined by $\sigma$ it is enough to prove the lemma in the case that $\sigma$ is a simple reflection.
Assume $\sigma=s_i=s_{\alpha_i}$ for $i<n$ and assume by contradiction that $s_i \set_<(w)s_i\not\subseteq\set_<(w')$. Writing $w=\tau\set$ with $\tau\in S_n$ and $\set\in \Xi_n$ we have $w'=\tau'\set'$ where $\tau'=s_i\tau s_i$ and $\set'=s_i(\set)$. Our assumption by contradiction (together with  Lemma \ref{lem w inv} \eqref{part ne}) means that there exists $j<\tau(j)$ such that $s_i(j)>\tau'(s_i(j))$. Since $\tau'(s_i(j))=s_i(\tau(j))$, this means that $j<\tau(j)$ but $s_i(j)>s_i(\tau(j))$. Hence $j=i$ and $\tau(j)=i+1$. Thus, $i\in \set$ if and only if $i+1\in \set$ and therefore
\[
w(e_i-e_{i+1})=\tau\set(e_i-e_{i+1})=\begin{cases} e_i-e_{i+1} & i\in\set \\ e_{i+1}-e_i & i\not\in\set\end{cases}
\]
which contradicts $w\xrightarrow{\alpha_i}w'$.

Assume now that $\sigma=\sigma_n=s_{\alpha_n}$ and recall that $\set_<(w)=\sigma_n \set_<(w)\sigma_n$. Writing $w=\tau\set$ with $\tau\in S_n$ and $\set\in \Xi_n$ we have $w'=\tau\set'$ where $\set'=\tau^{-1}(\sigma_n)\sigma_n\set$ and in particular $\set_<(w)=\set_<(w')$. The lemma follows.
\end{proof}

\subsection{The symmetric space}

To a connected reductive group $\mathbf{G}$ and an involution $\theta$ on $\mathbf{G}$ both defined over $F$ we associate the symmetric space
\[
\mathbf{X}=\mathbf{X}(\mathbf{G},\theta)=\{g\in \mathbf{G}:g\theta(g)=e\}
\]
with the $\mathbf{G}$-action by $\theta$-twisted conjugation
\[
(g,x)\mapsto g\cdot x=gx\theta(g)^{-1}, \ \ \ g\in \mathbf{G},\,x\in \mathbf{X}.
\]
For every $x\in \mathbf{X}$, let $\mathbf{G_x}=\Stab_{\mathbf{G}}(x)$ be its stabilizer, an algebraic group defined over $F$.
More generally, for any subgroup $\mbf{Q}$ of $\mbf{G}$ let $\mbf{Q_x}=\{g\in \mbf{Q} \mid g\cdot x=x\}$. 
\begin{example}\label{exm transitive}
We provide two well known examples where $X$ consists of a unique $G$-orbit.
\begin{enumerate}
\item\label{dxm alt}
Let $\mathbf{G}=\mathbf{GL_{2n}}$ and $\theta(g)=J_n {}^t g^{-1} J_n^{-1}$ (so that $\mathbf{G_e}=\mathbf{Sp_{2n}}$). Then 
\[\mathbf{X}=\{g\in \mathbf{G}: {}^t(gJ_n)=-gJ_n\}
\] 
and it is well known that $X=G\cdot e$.

\item\label{dxm gal} 
Let $\mathbf{G}=\mathbf{Res_{E/F}((GL_n)_E)}$ and $\theta(g)=\bar g$ (so that $G=\GL_n(E)$ and $G_e=\GL_n(F)$). Then by Hilbert 90 we have
$X=G\cdot e$.
\end{enumerate}
\end{example}

In this paper, our focus is mainly on the symmetric space $\mathbf{X}=\mathbf{X}(\mathbf{U_{2n}},\theta)$
where
\[
\theta(g)=\bar g.
\]
Recall that the stabilizer of the identity is $(\mathbf{U_{2n})_e}=\mbf{Sp_{2n}}$.

\begin{lemma}\label{lemma: cohom} For $\mathbf{X}=\mathbf{X}(\mathbf{U_{2n}},\theta)$  we have
\[
X=U_{2n}\cdot e.
\] 
\end{lemma}
\begin{proof}
Note that $\mathbf{G}(E)\simeq \GL_{2n}(E)$ and with this identification, over $E$ (and in particular over an algebraic closure of $F$), $\mathbf{X}(E)$ is isomorphic to the symmetric space of Example \ref{exm transitive}\eqref{dxm alt} with $F$ replaced by $E$. Thus it follows that $\mathbf{X}=\mathbf{G}\cdot e$. Since $H^1(F,\mbf{\Sp_{2n}})$ is trivial we have $(\mathbf{G}/\mathbf{G_e})(F)=\mathbf{G}(F)/\mathbf{G_e}(F)$. The lemma follows.
\end{proof}
\begin{remark}
This subsection is valid for any field $F$.
\end{remark}

\subsection{Distinction and contragredients}\label{s_p_cont}

Next, we show that an irreducible representation of $U_{2n}$ is $\Sp_{2n}$-distinguished if and only if its contragredient is.
First we recall the following result of M{\oe}glin, Vign\'eras and Waldspurger.

Let $V=E^{2n}$ (column vectors) and $\GL_F(V)$ the group of $F$-linear automorphism of $V$. We identify $\GL_{2n}(E)$ as a subgroup of $\GL_F(V)$ via matrix multiplication. In particular,  $U_{2n}$ is identified with the subgroup of $\GL_F(V)$ preserving the anti-hermitian form 
\[
\sprod{v}{v'}={}^t\bar v J_n v',\ v,\,v'\in V.
\]
Let
\[
\delta\begin{pmatrix} v_1\\ v_2\end{pmatrix}=\begin{pmatrix} \bar v_1\\ -\bar v_2\end{pmatrix},\ v_1,\,v_2\in E^n.
\]
Then $\delta\in \GL_F(V)$ satisfies $\delta^2=e$ and
\[
\sprod{v'}{v}=\sprod{\delta(v)}{\delta(v')},\ v,\,v'\in V. 
\]
In particular, $\delta(a v)=\bar a \delta(v)$ for $a\in  E$ and $v\in V$.
Note that we have
\begin{equation}\label{eq eps}
\delta g\delta^{-1}=\epsilon \bar g\epsilon\in U_{2n}, \ \ \ g\in U_{2n}
\end{equation}
where $\epsilon=\diag(I_n,-I_n)$. In particular, for $\pi\in \Alg(U_{2n})$ the representation $\pi^\delta\in \Alg(U_{2n})$ is defined on the space of $\pi$ by $\pi^\delta(g)=\pi(\delta g\delta^{-1})$.
The following is the content of  \cite[\S 4, II.1]{MR1041060}.
\begin{theorem}[M{\oe}glin-Vign\'eras-Waldspurger]\label{thm mvw}
For $\pi\in \Irr(U_{2n})$ we have 
\[
\pi^\vee\simeq\pi^\delta.
\]
\qed
\end{theorem}
\begin{corollary}\label{cor cont}
For $\pi\in \Irr(U_{2n})$ we have that $\pi$ is $\Sp_{2n}$-distinguished if and only if $\pi^\vee$ is $\Sp_{2n}$-distinguished.
\end{corollary}
\begin{proof}
Clearly $\pi$ is $\Sp_{2n}$-distinguished if and only if $\pi^\delta$ is $\delta^{-1} \Sp_{2n}\delta$-distinguished. However, by \eqref{eq eps} we have $\delta^{-1} h\delta=\epsilon h\epsilon$ for $h\in \Sp_{2n}$ and since $\epsilon\in \operatorname{GSp_{2n}}$ it normalizes $\Sp_{2n}$. Therefore $\delta^{-1} \Sp_{2n}\delta=\Sp_{2n}$ and the corollary follows from Theorem \ref{thm mvw}.
\end{proof}

\subsection{A necessary condition for distinction}

For $\pi\in\Irr(U_{2n})$ there exists a unique $0\le m\le n$ and a unique $\sigma\in\Cusp(U_{2m})$ up to isomorphism such that $\pi$ is a quotient of $\tau\rtimes \sigma$ for some representation $\tau\in \Pi(\GL_{n-m}(E))$. The representation $\sigma$ is called the partial cuspidal support  of $\pi$ and denoted by $\pi_\cusp$.

We say that $\pi$ has trivial partial cuspidal support if $m=0$ (and then $\sigma=\triv_0$ is the trivial representation of the trivial group so that $\tau\rtimes \triv_0$ is induced from the Siegel parabolic subgroup of $U_{2n}$).

The following is one of the the main results of \cite{DP}. It is also immediate from the main result of \cite{MR3590280} and the double coset analysis in \S \ref{s_geo_lem}.

\begin{proposition}\label{prop dp}
Let $\pi\in \Irr(U_{2n})$ be $\Sp_{2n}$-distinguished. Then $\pi$ has trivial partial cuspidal support. \qed
\end{proposition}

\section{The geometrical lemma for the symmetric space}\label{s_geo_lem}
Throughout this section let $G_n=U_{2n}$ and $H_n=\Sp_{2n}$ for $n\in \N$. Fix $n$ and let $G=G_n$ and $H=H_n$.  
Let
\[
X=\{g\in G \mid g\bar g=e\}
\]
be considered as a $G$-space with action given by twisted conjugation $g\cdot x=gx\bar g^{-1}$, $g\in G$ and $x\in X$. Recall that by Lemma \ref{lemma: cohom} we have that $X=G\cdot e$. The map $g\mapsto g\bar g^{-1}=g\cdot e$ defines a homeomorphism $G/H\simeq X$.

Let $P=M\ltimes U$ be a standard parabolic subgroup of $G$ with its standard Levi decomposition. 
We write $P=Q_{(n_1,\dots,n_k;m)}$ (so that $M=L_{(n_1,\dots,n_k;m)}$ and $U=V_{(n_1,\dots,n_k;m)}$).

We have a bijection $P\bs G/H\simeq P\bs X$. 
This set of orbits gives rise to a filtration of the restriction to $H$ of a representation of $G$ parabolically induced from $P$. 

\subsection{The filtration}

Since $P$ is stable under Galois conjugation we have $P\cdot x\subseteq PxP$, $x\in X$. Let 
\[
\iota_P:P\bs X\rightarrow W[2]\cap {}_MW_M
\] 
be the map defined by the Bruhat decomposition. That is, $\iota_P(P\cdot x)=w$ where $PxP=PwP$.
Then $L=M\cap wMw^{-1}$ is a standard Levi subgroup of $M$. Let $Q$ be the standard parabolic subgroup of $G$ with Levi subgroup $L$. Then $Q$ (resp. $L$) is the parabolic (resp. Levi) subgroup of $G$ associated to the $P$-orbit $P\cdot x$. 

Let $\sigma\in \Pi(M)$.
The restriction $\pi|_H$ to $H$ of the induced representation $\pi=\ip_{G,M}(\sigma)$ admits a natural filtration parameterized by $P\bs G/H$.

The decomposition factor associated to a double coset $P\eta H$ is isomorphic to
\[
I_\eta=\ind_{H_\eta^P}^H(\delta_{H_\eta^P}^{-1/2}\cdot(\delta_P^{1/2}\sigma)^{\eta})
\]
where $H_\eta^P=H\cap \eta^{-1}P\eta$. Let $x=\eta\bar\eta^{-1}\in X$. 
\begin{definition}\label{def co}
We say that $P\cdot x$ (or $w$) is \emph{relevant} to $\sigma$ if $I_\eta$ is $H$-distinguished.
\end{definition}
We say that $x$ is a \emph{good} representative for its $P$-orbit if $x\in N_G(L)$ (in this case $xL=wL$ and $L=M\cap xMx^{-1}$).
A good representative exists for every $P$-orbit (see \cite[Lemma 3.2]{MR3541705}).

Assume that $x$ is a good representative of its $P$-orbit.
By \cite[Proposition 4.1]{MR3541705} we have that $I_\eta$ is $H$-distinguished if and only if $r_{L,M}(\sigma)$ is $(L_x,\delta_{Q_x}\delta_Q^{-1/2})$-distinguished.

By \cite[Theorem 4.2]{MR3541705} we have
\begin{theorem}\label{thm geoml}
With the above notation, if $\pi$ is $H$-distinguished then there exists a $P$-orbit $\mathcal{O}$ in $X$ that is relevant to $\sigma$. 
In particular, $r_{L,M}(\sigma)$ is $(L_x,\delta_{Q_x}\delta_Q^{-1/2})$-distinguished for a good representative $x\in \mathcal{O}$.
\qed
\end{theorem}

In what follows, we make the condition that a $P$-orbit is relevant more explicit.
With the above notation, for $x\in X$ a good representative of its $P$-orbit set 
\[
\delta_x=\delta_{Q_x}\delta_Q^{-1/2}|_{L_x}.
\]

Let 
\[
[M]=\{L_{(n_{\sigma(1)},\dots,n_{\sigma(k)};m)}: \sigma\in S_k\}
\]
be the set of standard Levi subgroups of $G$ that are conjugate to $M$.
For $M'\in[M]$ let $W(M,M')$ be the set of $w\in W$ such that $w$ is of minimal length in $wW^M$ and $wMw^{-1}=M'$.
Let 
\[
W(M)=\sqcup_{M'\in[M]} W(M,M')
\]
and note that for $w\in W(M)$ and $w'\in W(wMw^{-1})$ we have that $w'w\in W(M)$.

Although $W(M)$ is not a group it is in a natural bijection with the signed permutation group $\weyl_k$. 
The set $\Delta_M$ of simple roots with respect to $M$ is in natural bijection with the set $\Delta_k$ of simple roots associated with $\weyl_k$ in \S \ref{sec weyl}. This identifies the set of elementary symmetries in $W(M)$ (in the sense of \cite[\S I.1.7]{MR1361168}) and the set of simple reflections in $\weyl_k$ and defines a unique bijection $\jmath_M:W(M)\rightarrow \weyl_k$ that satisfies
\[
\jmath_M(w'w)=\jmath_{wMw^{-1}}(w')\jmath_M(w),\ \ \ w\in W(M),\ w'\in W(wMw^{-1}).
\]
Explicitly, if $w\in W(M)$ and $\jmath_M(w)=\tau\set$ then there exists an element $t_w\in N_G(T)$, unique up to multiplication by an element of the center of $M$, representing the Weyl element $w$ and satisfying 
\[
t_w \inj(g_1,\dots,g_k,h) t_w^{-1}=\inj(g_1',\dots,g_k';h),\ \ \ g_i\in \GL_{n_i}(E),\ i=1,\dots,k,\ h\in G_m
\]
where
\[
g_i'=\begin{cases} g_{\tau^{-1}(i)}^* & i\in \set \\ g_{\tau^{-1}(i)} & i\not\in\set. \end{cases}
\]

Let 
\[
\NN_M=\{g\in G: gMg^{-1}\in [M]\}=\sqcup_{w\in W(M)} t_w M. 
\]
Then $\kappa_M:\NN_M/M \rightarrow W(M)$ given by $\kappa_M(t_w M)=w$ is a bijection satisfying 
\[
\kappa_{wMw^{-1}}(g')\kappa_M(g)=\kappa_M(g'g),\,g\in \NN_M,\,g'\in \NN_{wMw^{-1}}. 
\]
Let $\proj_M:\NN_M\rightarrow \weyl_k$ be the composition $\proj_M=\jmath_M\circ \kappa_M$.

Similar maps can be defined with $G$ replaced by $H$ and in particular, each $t_w$ can be chosen in $H$, that is, we may further assume that $\bar t_w=t_w$. Since also $\overline{M}=M$ it follows that
\[
\bar g^{-1}g\in M,\ g\in \NN_M.
\]
Note that restricted to $N_G(M)$, $\kappa_M$ defines a group isomorphism $N_G(M)/M \cong W(M,M)$.
Furthermore, for $w\in W(M)$ such that $\jmath_M(w)=\tau\set\in \weyl_k$ we have $w\in W(M,M)$ if and only if $n_{\tau(i)}=n_i$, $i=1,\dots,k$. 

\begin{remark}\label{rmk grph}
Let $x\in N_G(M)\cap X$ and $\rho=\proj_M(x)\in \weyl_k[2]$. If $\sigma\in\weyl_k$, $\rho'\in \weyl_k[2]$ and $g\in \NN_M$ are such that $\rho \overset{\sigma}\curvearrowright \rho'$ and $\proj_M(g)=\sigma$ then $(M,x) \overset{g}\curvearrowright (M',x')$ is a path in the graph defined in \cite[\S 6]{MR3541705} where $M'=gMg^{-1}=gM\bar g^{-1}$ and $x'=gx\bar g^{-1}$.
\end{remark}

\subsection{Minimal involutions}

In light of Lemma \ref{lem springer} the key for the study of $M$-orbits in $N_G(M)\cap X$ is to first understand the orbits that project under $\proj_M$ to minimal involutions.

\begin{definition}
An element $w\in W(M,M)$ is called $M$-minimal if $w=w_M^L$ for some standard Levi subgroup $L\supseteq M$ and $w\alpha=-\alpha$ for $\alpha\in\Delta_M^L$.         
\end{definition}
Note that $w\in W(M,M)$ is $M$-minimal if and only if $\jmath_M(w)\in \weyl_k$ is a minimal involution.

Let $w\in W(M,M)$ be $M$-minimal.
Let $\ell\in \{0,1,\dots,k\}$ and $\rho\in S_\ell$ a minimal involution so that $\jmath_M(w)=\rho\set_{\ell,k}$ (see \eqref{eq min inv} and Example \ref{ex min}). Let 
\[
S=\set_-(\jmath_M(w))\ \ \  {\rm and} \ \ \ R=\set_<(\jmath_M(w)).
\]

\begin{lemma}\label{lem Mmin}
With the above notation, $wM\cap X$ is not empty if and only if $n_i$ is even for all $\ell+1\le i\le k$ and in this case $wM\cap X$ is a unique $M$-orbit.
\end{lemma}
\begin{proof}
Assume that $x\in wM\cap X$. We can choose
\[
t_w=\inj(t_1,\dots,t_{r+s};t)\in w
\]
where 
\begin{equation}\label{eq ti}
t_j=\begin{cases} I_{n_{i(j)}} & i(j)\in S \\  \sm{0}{I_{n_{i(j)}}}{I_{n_{i(j)}}}{0} & i(j)\in R \end{cases}, \ \ \ i(1)=1, \ \ \ i(j+1)=\begin{cases} i(j)+1 & i(j)\in S\\ i(j)+2 & i(j)\in R\end{cases}
\end{equation}
for $j=1,\dots,r+s$ (note that $i(j)\in R\sqcup S$ for all $j$) and
\[
t=
\begin{pmatrix}
 & & & & & & I_{n_{\ell+1}} \\ & & & & & \iddots & \\ & & & & I_{n_k} & &  \\ & & & I_{2m} & & & \\ & & -I_{n_k} & & & & \\ & \iddots & & & & &
 \\ -I_{n_{\ell+1}} & & & & & &
\end{pmatrix}.
\]
Then, by definition, $x\in t_w M\cap X$. Let $x=t_w \inj(g_1,\dots,g_k;h)$ with $g_i\in \GL_{n_i}(E)$ and $h\in G_m$.
The condition $x\in X$ is equivalent to
\begin{equation}\label{eq min orb cond}
\begin{cases} 
g_i \bar g_i=I_{n_i} & i\in S \\
g_i \bar g_{i+1}=I_{n_i} & i\in R \\
g_i \bar g_i^*=-I_{n_i} & \ell+1\le i\le k \\
h \bar h=I_{2m}.\end{cases}
\end{equation}
In particular, for $\ell+1\le i\le k$ the condition is that $g_i w_{n_i}$ is an alternating matrix and therefore $n_i$ is even.
For $d_i\in \GL_{n_i}(E)$, $i=1,\dots,k$ and $d'\in G_m$ let $d=\inj(d_1,\dots,d_k;d')\in M$. Then
\[
dx\bar d^{-1}=t_w \inj(g_1',\dots,g_k';h')
\]
where
\[
\begin{cases} 
g_i'=d_i g_i\bar d_i^{-1} & i\in S \\
g_i'= d_{i+1} g_i\bar d_i^{-1},\ g_{i+1}'= d_i g_{i+1}\bar d_{i+1}^{-1}=\bar{g_i'}^{-1}& i\in R \\
g_i'=d_i^* g_i \bar d_i^{-1} & \ell+1\le i\le k \\
h'=d' h\bar{d'}^{-1}.\end{cases}
\]
It follows from Example \ref{exm transitive} and Lemma \ref{lemma: cohom} that $t_wM\cap X$ is a unique $M$-orbit.

On the other hand, assuming that $n_i$ is even whenever $\ell+1\le i\le k$, with the above choice of $t_w$ we have
\begin{equation}\label{eq min rep}
x_w=t_w\inj(I_{n_1+\cdots+n_\ell},\epsilon_{n_{\ell+1}},\dots,\epsilon_{n_k};I_{2m})\in wM\cap X
\end{equation}
where $\epsilon_{2n}=\sm{-I_n}{0}{0}{I_n}$.
\end{proof}

Next we compute the stabilizer and relevant modulus function for a minimal orbit.
\subsubsection{}\label{sss min stab}

Let $x\in N_G(M)\cap X$ be $M$-minimal. Applying the notation of Lemma \ref{lem Mmin} and its proof we write $x=t_w  \inj(g_1,\dots,g_k;h)$ so that \eqref{eq min orb cond} holds. Following the proof of the lemma we see that $M_x$ consists of elements of the form $d=\inj(d_1,\dots,d_k;d')$ where
\[
\begin{array}{lll} 
d_i\in \GL_{n_i}(E) & g_i=d_i g_i\bar d_i^{-1} & i\in S \\
d_i\in \GL_{n_i}(E) & d_{i+1}= g_i\bar d_i g_i^{-1} & i\in R \\
d_i\in \GL_{n_i}(E) & g_i=d_i^* g_i \bar d_i^{-1} & \ell+1\le i\le k \\
d'\in G_m & h=d' h\bar{d'}^{-1} & 
\end{array}
\]
and in particular
\[
M_x \cong \left[ \prod_{i\in S}\GL_{n_i}(F)\right]\times \left[ \prod_{i\in R}\GL_{n_i}(E)\right]\times\left[ \prod_{i=\ell+1}^k \Sp_{n_i}(E)\right] \times H_m.
\]

With the above notation, for $x=x_w$ (see \eqref{eq min rep}) it is easy to see that $M_x$ consists of elements of the form $\inj(g_1,\dots,g_k;h)$ where 
\begin{itemize}
\item $g_i\in \GL_{n_i}(F)$, $i\in S$;
\item $g_{i+1}=\bar g_i\in \GL_{n_i}(E)$, $i\in R$;
\item $g_i\in \Sp_{n_i}(E)$, $\ell+1\le i\le k$; 
\item $h\in H_m$.  
\end{itemize}

\subsubsection{}\label{sss min mod}
We explicate the modulus functions $\delta_x$ associated to $M$-minimal elements when $M$ is contained in the Siegel Levi subgroup. We freely use the notation of \S \ref{sss min stab}.
\begin{lemma}\label{lem modulus comp}
Let $x\in N_G(M)\cap X$ be $M$-minimal and assume that $M$ is contained in the Siegel Levi subgroup of $G$ (i.e., that with the above notation $m=0$). Then,  
\[
\delta_x(\inj(d_1,\dots,d_k))=\prod_{i\in S} \abs{\det d_i}_F \cdot \prod_{i\in R} \abs{\det d_i}_E, \ \ \ \inj(d_1,\dots,d_k)\in M_x.
\]
\end{lemma}
\begin{proof}
For $g\in M$ it is easy to see that the above equation holds for $x$ if and only if it holds for $gx\bar g^{-1}$. By Lemma \ref{lem Mmin} we may therefore assume that $x=x_w$ as defined by \eqref{eq min rep}. For this case we carry the computation out explicitly. 

For an automorphism $\alpha$ of a group $Q$ we denote by $\mod_Q(\alpha)$ the associated modulus function.
For example, for $z\in \GL_a(F)$ and $y\in \GL_b(F)$ we have
\begin{equation}\label{eq mm}
\mod_{M_{a\times b}(F)}(z)=\abs{\det z}_F^b, \ \ \ \mod_{M_{a\times b}(F)}(y)=\abs{\det y}_F^a
\end{equation}
where the respective automorphism is given by left or right matrix multiplication. Let $\SS_n(F)=\{z\in M_n(F): {}^t z=z\}$
and $\HH_n(E/F)=\{z\in M_n(E): {}^t \bar z=z\}$. Then, with respect to the natural actions $g\cdot z=gz\,{}^t g$, $g\in \GL_n(F)$, $z\in \SS_n(F)$ and $g\cdot z=gz{}^t \bar g$, $g\in \GL_n(E)$, $z\in \HH_n(E/F)$ we have
\begin{equation}\label{eq mod sh}
\mod_{\SS_n(F)}(g)=\abs{\det g}_F^{n+1},\ g\in \GL_n(F), \ \ \ \ \ \mod_{\HH_n(E/F)}(g)=\abs{\det g}_E^n,\ g\in \GL_n(E).
\end{equation}
In our modulus function computations below we freely use, without further mention, the fact that $\det g=1$ for $g\in \Sp_{2m}(E)$, $m\in \N$.

Recall (see \eqref{eq ti}) that $x=\inj(t_1,\dots,t_{r+s};x')$ where
\[
x'=
-\begin{pmatrix}
 & & & & &  \epsilon_{n_{\ell+1}} \\ & & & &  \iddots & \\ & & &  \epsilon_{n_k} & &   \\ & & \epsilon_{n_k}  & & & \\ & \iddots  & & & &
 \\ \epsilon_{n_{\ell+1}}  & & & & &
\end{pmatrix}.
\]
and that $M_x$ consists of elements of the form $\inj(g_1,\dots,g_k)$ where
\begin{equation}\label{eq Mx min}
\begin{cases} g_i\in \GL_{n_i}(F) & i\in S \\ g_{i+1}=\bar{g_i}\in \GL_{n_i}(E) & i\in R\\ g_i\in \Sp_{n_i}(E) & i=\ell+1,\dots,k. \end{cases}
\end{equation}

Let $N_1=n_1+\cdots+n_\ell$, $N_2=n_{\ell+1}+\cdots+n_k$ and $Q=L\ltimes V$ the maximal parabolic subgroup of $G$ with $L=M_{(N_1;N_2)}$ and $V=U_{(N_1;N_2)}$. Let $Z=L\cap U$ so that $U=Z\ltimes V$ and consider the involution $\inv_x(g)=x\bar g x^{-1}$, $g\in G$. Note that $g\cdot x=x$ if and only if $\theta_x(g)=g$. Since $x\in L$, it normalizes $V$ and therefore $\theta_x(V)=V$. It is easy to see that $U\cap \theta_x(U)=(Z\cap \theta_x(Z))\ltimes V$ and since this decomposition is $\theta_x$-stable we get that $U_x=Z_x \ltimes V_x$. Furthermore, each of the groups $Z$, $V$, $Z_x$ and $V_x$ is normalized by $M_x$. It follows that 
 \begin{equation}\label{eq mod comp}
 \delta_P(h)=\mod_Z(h) \mod_{V}(h) \text{ and similarly }\delta_{P_x}(h)=\mod_{Z_x}(h) \mod_{V_x}(h),\ \ \ h\in M_x.
\end{equation}
We write a matrix $z\in Z$ in block form as $z=\inj(z_1;z_2)$ where $z_1=(z_{i,j})$, with $1\le i,\,j\le \ell$ and $z_{i,j}\in M_{n_i\times n_j}(E)$ satisfying
\[
\begin{array}{ll}
z_{i,j}=0  &  i>j  \\   z_{i,i}=I_{n_i} & 
\end{array}
\]
and $z_2\in U_{(n_{\ell+1},\dots,n_k)}$.

It is then easy to verify that $z\in Z_x$ if and only if $z_2=I_{N_2}$ and whenever $1\le i<j\le \ell$ we have
\[
\begin{array}{ll}
z_{i,i+1}=0 & i\in R \\
z_{i,j}\in M_{n_i\times n_j}(F) & i,\,j\in S \\
z_{i,j+1}=\bar z_{i,j} & i\in S,\,j\in R \\
z_{i+1,j}=\bar z_{i,j} & i\in R,\,j\in S \\
z_{i+1,j+1}=\bar z_{i,j} \text{ and } z_{i+1,j}=\bar z_{i,j+1} & i,\,j\in R.
\end{array}
\]
Furthermore, for $h=\inj(g_1,\dots,g_k)\in M_x$ as in \eqref{eq Mx min} and $z\in Z$ we have
\[
hzh^{-1}=\inj(z_1';z_2')
\]
where $z_1'=(g_i z_{i,j} g_j^{-1})$ (and similarly $z_2'\in U_{(n_{\ell+1},\dots,n_k)}$ is the conjugate of $z_2$ by $\iota(g_{\ell+1},\dots,g_k)$).

Taking \eqref{eq mm} into consideration, it is now easy to see that
\begin{equation}\label{eq Z}
\mod_Z^{1/2}|_{M_x}=\mod_{Z_x}|_{M_x}.
\end{equation}
Next, let
\[
V^{(1)}=\{v^{(1)}(s)=\begin{pmatrix} I_{N_1} & s & \\  & I_{2N_2} & s' \\ & & I_{N_1}\end{pmatrix}\mid s\in M_{N_1\times 2N_2}(E),\ s'=J_{N_2} {}^t\bar s \,w_{N_1}\}
\]
and 
\[
V^{(2)}=\{v^{(2)}(y)=\begin{pmatrix} I_{N_1} &   & y \\  & I_{2N_2} &  \\ & & I_{N_1}\end{pmatrix}\mid yw_{N_1}\in \HH_{N_1}(E/F)\}. 
\]
Note that $V^{(2)}$ is a normal subgroup of $V$ and $V^{(1)}$ is a set of representatives for $V/V^{(2)}$. Thus, 
\[
V=\{v(s,y)=v^{(1)}(s)v^{(2)}(y)\mid s\in M_{N_1\times 2N_2}(E),\ yw_{N_1}\in \HH_{N_1}(E/F)\}.
\]
Note that if $v=v_1v_2$ with $v_i\in V^{(i)}$ then also $\theta_x(v_i) \in V^{(i)}$, $i=1,2$. By the uniqueness of decomposition it follows that $v\in V_x$ if and only if $v_1,\,v_2\in V_x$. For $v=v(r,s)\in V$ write $s=(s_1,s_2)$ with $s_1,\,s_2\in M_{N_1\times N_2}(E)$ and $yw_{N_1}=(y_{i,j})$ with $y_{i,j}\in M_{n_i\times n_j}(E)$, $1\le i,\,j\le \ell$. Then explicitly,
\[
\begin{array}{ll}
y_{j,i}={}^t\bar y_{i,j} \in M_{n_j\times n_i}(E) & i\ne j \\
y_{i,i}\in \HH_{n_i}(E/F).& 
\end{array}
\]
We then have $v\in V_x$ if and only if
\[
s_2=-\diag(t_1,\dots,t_{r+s})\,\overline{s_1} \begin{pmatrix}  & & \epsilon_{n_{\ell+1}} \\ & \iddots & \\ \epsilon_{n_k} & &   \end{pmatrix}
\]
and $y$ satisfies
\[
\begin{array}{ll}
y_{i,i}\in\SS_{n_i}(F) & i\in S\\
y_{i+1,i+1}=\bar y_{i,i} \text{ and }y_{i+1,i}=\bar y_{i,i+1}\in \SS_{n_i}(E)&i \in R \\
\text{and for all }i<j & \\
y_{i,j+1}=\bar y_{i,j} & i\in S,\,j\in R \\
y_{i+1,j}=\bar y_{i,j} & i\in R,\,j\in S \\
y_{i+1,j+1}=\bar y_{i,j} \text{ and } y_{i+1,j}=\bar y_{i,j+1} & i,\,j\in R.
\end{array}
\]

Furthermore, for $h=\inj(g_1,\dots,g_k)\in M_x$ as in \eqref{eq Mx min} we have
\[
hv^{(1)}(s)h^{-1}=v^{(1)}(s'')
\]
where 
\[
s''=\diag(g_1,\dots,g_\ell) \,s \,  \iota(g_{\ell+1},\dots,g_k)^{-1}.
\]
and
\[
hv^{(2)}(y)h^{-1}=v^{(2)}(y''') \text{ where }(y''')_{i,j}=g_iy_{i,j}{}^t \bar g_j. 
\]
Thus,
\[
hv(s,y)h^{-1}=h(s'',y''').
\]

Taking \eqref{eq mm} and \eqref{eq mod sh} into consideration, it is now easy to see that
\begin{equation}\label{eq V}
\mod_V^{-1/2}(h)\mod_{V_x}(h)=\mod_{V^{(2)}}^{-1/2}(h)\mod_{V^{(2)}_x}(h)=\left[\prod_{i\in S}\abs{\det g_i}_F \right]\left[\prod_{i\in R} \abs{\det g_i}_E  \right].
\end{equation}
The lemma is now immediate from \eqref{eq Z}, \eqref{eq V} and \eqref{eq mod comp}.
\end{proof}

\subsubsection{}\label{sss min diet} As an immediate consequence we formulate the following.
\begin{corollary}
Assume that $M$ is contained in the Siegel Levi subgroup of $G$. Let $x\in N_G(M)\cap X$ be $M$-minimal and $\proj_M(x)=\rho\set_{\ell,k}$, $\rho\in S_\ell$ (see \eqref{eq min inv}). With the above notation let $\pi_i$ be a representation of $\GL_{n_i}(E)$, $i=1,\dots,k$. The representation $\pi=\pi_1\otimes\cdots\otimes\pi_k$ of $M$ is $(M_x,\delta_x)$-distinguished if and only if 
\[
\begin{cases} \nu^{-1/2}\pi_i \text{ is }\GL_{n_i}(F)\text{-distinguished} & i\in S \\ \pi_{i+1} \cong \nu\bar\pi_i^\vee & i\in R \\ \pi_i \text{ is }\Sp_{n_i}(E)-\text{distinguished} & \ell+1\le i\le k.   \end{cases}
\] \qed
\end{corollary}

\subsection{Admissible orbits}
Let 
\[
\WR_M=\proj_M(N_G(M)\cap X).
\]
\begin{lemma}
We have
\[
\WR_M=\{w\in\weyl_k[2]\mid n_i\text{ is even for all }i\in \set_+(w)\}.
\]
Furthermore, $p_M$ defines a bijection
\[
M\bs (N_G(M)\cap X)\simeq\WR_M.
\]
\end{lemma}
\begin{proof}
Let $w\in\weyl_k[2]$ and let $\ell$ be such that $k-\ell=\abs{\set_+(w)}$. It follows from Lemma \ref{lem springer}, \eqref{eq min inv} and Lemma \ref{lem w inv} \eqref{part +} that there exist $\sigma\in \weyl_k$ and $\rho\in S_\ell$ a minimal involution such that $w \overset{\sigma}\curvearrowright \rho\set_{\ell,k}$.

For $g\in \NN_M$ let $M'=gM\bar g^{-1}=gMg^{-1}\in[M]$. Then the map $x\mapsto gx\bar g^{-1}$ defines a bijection between $M$-orbits in $N_G(M)\cap X$ and $M'$-orbits in $N_G(M')\cap X$. 
By choosing $g$ so that $\proj_M(g)=\sigma$ we see that $w\in \WR_M$ if and only if $\rho\set_{\ell,k}\in \WR_{M'}$.
The lemma now follows from Lemma \ref{lem Mmin}.
\end{proof}

\subsubsection{}\label{sss stabx}
Let $x\in N_G(M)\cap X$ and $w=\proj_M(x)\in\WR_M$. Let $\tau\in S_k$ and $\set\in \Xi_k$ be such that $w=\tau\set$.
Based on the above analysis we can now describe $M_x$ as follows:
\[
M_x \cong \left[ \prod_{i\in \set_+(w)} \GL_{n_i}(F)   \right]\times\left[\prod_{i\in \set_<(w)} \GL_{n_i}(E)    \right]\times\left[\prod_{i\in \set_+(w)}\Sp_{n_i}(E)    \right]\times H_m.
\]
More explicitly, a representative $x$ as above can be chosen so that $M_x$ consists of elements $\inj(g_1,\dots,g_k;h)$ such that
\begin{equation}\label{eq adm stab}
\begin{array}{ll}
g_i\in\GL_{n_i}(F) & i\in \set_-(w) \\
g_{\tau(i)}=\bar g_i\in \GL_{n_i}(E) & i\in \set_<(w)\setminus \set \\
g_i\in \Sp_{n_i}(E) & i\in \set_+(w) \\
g_{\tau(i)}=\bar g_i^*\in \GL_{n_i}(E) & i\in \set_<(w)\cap \set \\
h\in H_m.& 
\end{array}
\end{equation}

\subsubsection{}\label{sss xmod}
The following proposition follows from \cite[Corollary 6.5]{MR3541705} in light of Remark \ref{rmk grph}. It is a consequence of Springer's theory of twisted involutions as further developed in \cite{MR2010737}.

\begin{proposition}\label{prop grph}
Let $x\in N_G(M)\cap X$ with $\proj_M(x)=\rho\in \weyl_k[2]$ and let $\sigma\in \weyl_k$ and $\rho'\in \weyl_k[2]$ be such that $\rho \overset{\sigma}\curvearrowright \rho'$. For $g\in \NN_M$ such that $\proj_M(g)=\sigma$ set $M'=gMg^{-1}$ and $x'=gx\bar g^{-1}$ (so that $\proj_{M'}(x')=\rho'$). Then  
\[
\delta_{x'}(gmg^{-1})=\delta_x(m), \ \ \ m\in M_x.
\]\qed
\end{proposition}
\begin{corollary}\label{cor adm mod}
Assume that $m=0$ (i.e. $M$ is contained in the Siegel Levi subgroup of $G$) and let $w=\tau\set\in \WR_M$ with $\tau\in S_k$ and $\set\in \Xi_k$. Then, for $x\in N_G(M)\cap X$ such that $\proj_M(x)=w$ and $\inj(g_1,\dots,g_k)\in M_x$ we have
\begin{equation}\label{eq adm mod}
\delta_x(\inj(g_1,\dots,g_k))=\prod_{i\in \set_-(w)}\abs{\det g_i}_F \cdot \prod_{i\in\set_< (w)} \abs{\det g_i}_E=\abs{\det g}_F
\end{equation}
where $g=\diag(g_1,\dots,g_k)$.
\end{corollary}
\begin{proof}
The second equality of \eqref{eq adm mod} is straightforward from \eqref{eq adm stab}. We prove the first equality. 
By Lemma \ref{lem springer} there exist $w'\in \weyl_k[2]$ a minimal involution and $\sigma\in \weyl$ such that $w\overset{\sigma}\curvearrowright w'$.
Let $g\in \NN_M$ be such that $\sigma=\proj_M(g)$ and let $x'=gx\bar g^{-1}$ and $M'=gMg^{-1}$. Then $x'$ is $M'$-minimal. 
It follows from Lemma \ref{lem modulus comp} that \eqref{eq adm mod} holds for $x'$. That it holds for $x$ now follows from Proposition \ref{prop grph} and Lemmas \ref{lem w inv} and \ref{lem <}.
\end{proof}

\subsubsection{}\label{sss xdist}
As a consequence of the explication of the stabilizer in \eqref{eq adm stab} and Corollary \ref{cor adm mod} we have the following.

\begin{corollary}\label{cor dist adm}
Assume that $M$ is contained in the Siegel Levi subgroup of $G$  and let $x\in N_G(M)\cap X$ and $\proj_M(x)=w$. Let $\tau\in S_k$ and $\set\in \Xi_k$ be such that $w=\tau\set$. For representations $\pi_i$ of $\GL_{n_i}(E)$, $i=1,\dots,k$, the representation $\pi=\pi_1\otimes\cdots\otimes\pi_k$ of $M$ is $(M_x,\delta_x)$-distinguished if and only if 
\[
\begin{cases} \nu^{-1/2}\pi_i \text{ is }\GL_{n_i}(F)\text{-distinguished} & i\in \set_-(w) \\\pi_i \text{ is }\Sp_{n_i}(E)\text{-distinguished} & i\in \set_+(w) \\ \pi_{\tau(i)}\cong \nu\bar\pi_i^\vee & i\in\set_{\ne}(w) \setminus\set \\  \pi_{\tau(i)}\cong \nu^{-1} (\pi_i^\iota)^\vee & i\in \set_<(w)\cap \set \end{cases}
\]
where $\iota$ is the involution $g\mapsto g^\iota=w_n {}^tg^{-1}w_n$ on $\GL_n(E)$. 
\qed
\end{corollary}
\begin{remark}\label{rmk gk}
By \cite{MR0404534} we have $\pi^\vee\cong \pi^\iota$ if $\pi$ is irreducible. Thus, if the $\pi_i$'s are assumed to be irreducible then the condition above for $i\in  \set_<(w)\cap \set$ becomes $\pi_{\tau(i)}\cong \nu^{-1} \pi_i$.
\end{remark}

\subsection{General orbits}\label{ss gen orb}
Assume that $M$ is contained in the Siegel Levi ($m=0$). We consider a general $P$-orbit $P\cdot x$ in $X$. Assume without loss of generality that $x$ is a good representative. Let $w\in{}_MW_M\cap W[2]$ be such that $PxP=PwP$ so that $x\in wL\cap X$, where $L=M\cap wMw^{-1}$ is the standard Levi subgroup of $M$ associated with the orbit $P\cdot x$. Let $Q=L\ltimes V$ be the associated standard parabolic subgroup.

Write $L=M_{(\gamma_1,\dots,\gamma_k;0)}$ where $\gamma_i=(a_{i,1},\dots,a_{i,j_i})$ is a partition of $n_i$ and let 
\[
\Index=\{(i,j):i=1,\dots,k,\ j=1,\dots, j_i\}.
\]
Via lexicographic order on $\Index$ (i.e. $(i,j)\prec (i',j')$ if either $i<i'$, or $i=i'$ and $j<j'$) we identify $\weyl_{\abs{\Index}}$ with $S_\Index\ltimes \Xi_\Index$ where $S_\Index$ is the permutation group on $\Index$ and $\Xi_\Index$ the group of subsets of $\Index$ with respect to symmetric difference.
Then $x\in N_G(L)\cap X$ and we can set $\proj_L(x)=\tau\set$ with $\tau\in S_\Index$ and $\set\subseteq \Index$.

Note that since $w$ is $L$-admissible it acts as an involution on the set of roots 
\[
\Sigma_L=R(G,T_L).
\] 
The fact that $L=M\cap wMw^{-1}$ is equivalent to saying that $w\alpha\not\in\Sigma_L^M:=R(M,T_L)$ for $\alpha\in \Sigma_L^M$. This can be explicated as follows. 
\begin{itemize}
\item If $(i,j),\,(i,j') \in \set$ with $j\ne j'$, $\tau(i,j)=(a,b)$ and $\tau(i,j')=(c,d)$ then $a\ne c$.
\item If $(i,j),\,(i,j')\not\in \set$ with $j\ne j'$, $\tau(i,j)=(a,b)$ and $\tau(i,j')=(c,d)$ then $a\ne c$.
\end{itemize}

The fact that $w$ is $M$-reduced means that $w\alpha>0$ for every $\alpha\in \Delta_L^M$. This can be explicated as
\begin{itemize}
\item For $j\ne j'$ if $(i,j)\not\in\set$ and $(i,j')\in \set$ then $j<j'$. 
\item if $(i,j),\,(i,j')\in \set$ with $j<j'$ then $\tau(i,j')\prec\tau(i,j)$.
\item if $(i,j),\,(i,j')\not\in \set$ with $j<j'$ then $\tau(i,j)\prec\tau(i,j')$.
\end{itemize}

Combined, we get that for every $i$ there is $s_i\in \{0,1,\dots,j_i\}$ such that
\begin{equation}\label{eq: c form}
\set=\{(i,j): s_i<j\le j_i \}
\end{equation}
and if $\tau(\imath)=(c_{\imath},d_{\imath})$, $\imath\in \Index$ then 
\begin{equation}\label{orbit_cond}
\begin{cases} c_{(i,j)}<c_{(i,j')} & \ {\rm for\ } j<j'\le s_i \\ c_{(i,j)}>c_{(i,j')} &  \ {\rm for\ } s_i<j<j'. \end{cases}
\end{equation}

\section{A graph of signs}\label{app_s:signs}
In this section we describe a graph that underlies the combinatorics behind the classification of M{\oe}glin and Tadi\'c of the discrete series of classical groups \cite{MR1896238}. The material of the section is purely combinatorial and can be read independently from the rest of this article. We hope that the combinatorial framework that we develop here would have applications more generally in questions where the M{\oe}glin-Tadi\'c classification plays a role.
 
Let $V=\sqcup_{k=0}^\infty \{\pm 1\}^k$ be the set of ordered tuples of signs. We consider the directed, labeled graph $\E$ with vertices $V$ and labeled edges $\e \go{i} \e'$ whenever $\e=(e_1,\dots,e_k)\in \{\pm1\}^k$, $i\in \{1,\dots,k-1\}$ is such that $e_i=e_{i+1}$ and $\e'=(e_1,\dots,e_{i-1},e_{i+2},\dots,e_k)\in \{\pm1\}^{k-2}$ is obtained from $\e$ by deleting the $i$th and the $(i+1)$th entries.

For $\e,\,\e'\in V$, a path from $\e$ to $\e'$ is labeled by the sequence $\imath=(i_1,\dots,i_t)$ and denoted by $\e \path{\imath}\e'$ if there exists $\e_1,\dots,\e_t\in V$ and edges
on the graph such that
\[
\mbf{e}=\mbf{e}_1 \go{i_1} \mbf{e}_2 \go{i_2}\cdots \go{i_t} \mbf{e}_t=\mbf{e'}.
\]
We write $\e\path{}\e'$ if there exists a path in $\E$ from $\e$ to $\e'$. The $t-tuple$ of natural numbers $\imath$ is the \emph{pattern} 
of the path 
 $\e\path{\imath}\e'$.
 
To make some of the basic arguments more formal it will also be convenient to introduce the coordinate \emph{history} of the path $\e \path{\imath}\e'$. For $\e\in \{\pm 1\}^k$ and pattern $\imath=(i_1,\dots,i_t)$ it is a sequence of $t$ pairs of indices in $\{1,\dots,k\}$ that keeps track of the coordinates of $\e$ that are deleted at each edge of the path. More explicitly, it is the sequence 
\[
((x_1,y_1),\dots,(x_t,y_t))
\] 
where $x_i<y_i$ are the coordinates of $\e$ that are deleted in the $j$th edge, $j=1,\dots,t$.
The coordinate history is defined recursively as follows. 
Set $(x_1,y_1)=(i_1,i_1+1)$. For $1< s\le t$ let $z_1<\dots <z_{k-2(s-1)}$ be such that 
\[
\{z_1,\dots,z_{k-2(s-1)}\}=\{1,\dots,k\}\setminus\{x_1,\,y_1,\dots,x_{s-1},\,y_{s-1}\}.
\]
Then we define
\[
x_s=z_{i_s} \ \ \ \text{and}\ \ \ y_s=z_{i_s+1}.
\]

For future reference we record the following basic properties of the history of a path. They are straightforward from the definitions.
\begin{equation}\label{eq history}
\text{For }1\le i,\,j\le t \text{ the inequalities } x_i<x_j<y_i<y_j \text{ cannot hold simultaniously.}
\end{equation}
\begin{equation}\label{eq between}
[x_i+1,y_i-1]\subseteq \{x_j,\,y_j: j=1,\dots,i-1\} \text{ for }i=2,\dots,t.
\end{equation}
Here and henceforth, for integers $a\le b$ we denote by $[a,b]=\{i\in \Z\mid a\le i\le b\}$ the associated interval of integers.

For $t\in \Z_{\ge 0}$ let $\f_t=(1,-1,\dots,(-1)^{t-1})\in \{\pm 1\}^t$ where $\f_0\in V$ is the empty tuple of signs. 
\begin{definition}\label{def:vi}
We define the subsets of vertices  $V_t$ and $V_{-t}$ by
\[
V_{\pm t}=\{\e\in V:\e\path{}\pm\f_t\}
\] 
and let $\E_{\pm t}$  be the full subgraph of $\E$ with vertices $V_{\pm t}$.
\end{definition}

Clearly, $\e\in V$ is not the origin of any edge in $\E$ if and only if $\e=\pm\f_t$ for some $t\in \Z_{\ge 0}$. Since the number of coordinates decreases by two with every edge it follows that 
\[
V=\cup_{t\in \Z} V_t.
\]
As we will soon see, this union is disjoint.

\begin{lemma}\label{lem conn}
Let $\e \go{i} \e'$ be an edge in $\E$ and $t\in \Z$. Then $\e\in V_t$ if and only if $\e'\in V_t$.
\end{lemma}
\begin{proof}
If $\e'\path{\jmath}\pm \f_t$ then clearly also $\e\path{(i,\jmath)}\pm f_t$. The if part follows.

Let $((x_1,y_1),\dots,(x_k,y_k))$ be the history of a path $\e=(e_1,\dots,e_t)\path{}\pm\f_t$. Note that 
$\{i,i+1\}\cap\{x_1,y_1,\dots,x_k,y_k\}$ cannot be empty. We separate the proof into several cases.

Assume first that $(i,i+1)=(x_j,y_j)$ for some $j$.
Then it is a simple consequence of \eqref{eq history} and \eqref{eq between} that
\[
((x_j,y_j), (x_1,y_1),\dots,(x_{j-1},y_{j-1}),(x_{j+1},y_{j+1}),\dots,(x_k,y_k))
\] 
is the history of a path $\e\path{}\pm\f_t$ that starts with the edge $\e\go{i}\e'$. Truncating this edge gives a path $\e'\path{}\pm\f_t$.

Next, assume that there exist $a\ne b$ such that $i\in \{x_a,y_a\}$ and $i+1\in \{x_b,y_b\}$. It follows from \eqref{eq history} and \eqref{eq between} that there are three cases to consider:
\begin{enumerate}
\item\label{case separate}
$y_a=i$ and $x_b=i+1$;
\item\label{case left}
 $y_a=i$, $y_b=i+1$, $x_b<x_a$ and $a<b$;
\item\label{case right}
$x_a=i$, $x_b=i+1$, $y_b<y_a$ and $b<a$.
\end{enumerate}
In case \eqref{case separate}, let 
\begin{itemize}
\item $A=\{j\in \{1,\dots,k\}: x_a<x_j,\,y_j<i\}$, 
\item $B=\{j\in \{1,\dots,k\}: i+1<x_j,\,y_j<y_b\}$ and 
\item $C=\{j\in \{1,\dots,k\}: x_j,\,y_j\not\in [x_a,y_b]\}$. 
\end{itemize}

In case \eqref{case left} let 
\begin{itemize}
\item $A=\{j\in \{1,\dots,k\}: x_b<x_j,\,y_j<x_a\}$, 
\item $B=\{j\in \{1,\dots,k\}: x_a<x_j,\,y_j<i\}$ and 
\item$C=\{j\in \{1,\dots,k\}: x_j,\,y_j\not\in [x_b,i+1]\}$. 
\end{itemize}

In case \eqref{case right} let 
\begin{itemize}
\item $A=\{j\in \{1,\dots,k\}: y_b<x_j,\,y_j<y_a\}$, 
\item $B=\{j\in \{1,\dots,k\}: i+1<x_j,\,y_j<y_b\}$ and 
\item $C=\{j\in \{1,\dots,k\}: x_j,\,y_j\not\in [i,y_a]\}$. 
\end{itemize}

It then follows from \eqref{eq history} and \eqref{eq between} that with the standard linear order on $A$, on $B$, and on $C$ the sequence
$((x_j,y_j)_{j\in A},(x_j,y_j)_{j\in B}, (i,i+1), (z,w), (x_j,y_j)_{j\in C})$ is the history of a path $\e\path{}\pm\f_t$ where
\[
(z,w)=\begin{cases} (x_a,y_b) & \text{in case } \eqref{case separate} 
\\ (x_b,x_a) & \text{in case } \eqref{case left}
\\ (y_b,y_a) & \text{in case } \eqref{case right}.
 \end{cases}
\] 
Since $(i,i+1)$ is in the history of this path, it follows from the case considered first that $\e'\path{}\pm\f_t$.

The only case left is where $\{i,i+1\}\cap\{x_1,y_1,\dots,x_k,y_k\}$ is a singleton set. It follows from \eqref{eq history} and \eqref{eq between} that if the intersection is $\{i\}$ then $i=y_a$ for some $a$, if the intersection is $\{i+1\}$ then $i=x_a$ for some $a$ and that in either case, for this $a$ the sequence
\[
((x_1,y_1),\dots,(x_{a-1},y_{a-1}),(i,i+1),(x_{a+1},y_{a+1}),\dots,(x_k,y_k))
\] 
is the history of another path $\e\path{}\pm\f_t$. We are reduced again to the first case and the lemma follows.

\end{proof}
As an immediate consequence of Lemma \ref{lem conn}, by induction on the length of the path $\e\path{}\e'$, we have the following
\begin{corollary}\label{cor conn}
For a path $\e\path{}\e'$ in $\E$ and $t\in \Z$ we have $\e\in V_t$ if and only if $\e'\in V_t$. 
In particular, $\E_t$, $t\in \Z$ are the (non-directed) connected components of $\E$ and the union
\[
V=\sqcup_{t\in \Z} V_t
\]
is disjoint.
\qed
\end{corollary}

Clearly, every $\e\in V\setminus\{\f_0\}$ is of the form $\e=\pm(\triv_{s_1},-\triv_{s_2},\dots,(-1)^{m-1}\triv_{s_m})$ for unique $m,\,s_1,\dots,s_m\in \N$ where $\triv_s=(1,\dots,1)\in\{\pm 1\}^s$, $s\in \N$. For our purposes it will be more convenient to represent elements of $V$ differently, in terms of the $\f_t$. 

For $m,\,t_1,\dots,t_m\in \N$ let
\[
\f_{t_1,\dots,t_m}=(\f_{t_1},(-1)^{\tau_2}\f_{t_2},\dots,(-1)^{\tau_{m}}\f_{t_m})
\]
where $\tau_{2i}=t_1+\cdots +t_{2i-1}+1$ and $\tau_{2i+1}=t_1+\cdots+t_{2i}$, $i=1,\dots,[m/2]$.
For example, 
\[
\f_{3,1,1,2}=(1,-1,1,1,1,1,-1). 
\]
It is a simple observation that every $\e\in V\setminus\{\f_0\}$ is of the form $\e=\pm\f_{t_1,\dots,t_m}$ for unique $m,\,t_1,\dots,t_m\in \N$.
Indeed, for $\e=(e_1,\dots,e_k)\in V\setminus\{\f_0\}$ we have $\e=e_1 \f_{t_1,\dots,t_m}$ where 
\[
\{i\in\{1,\dots,k\}: e_i=e_{i+1}\}=\{t_1+\cdots+t_i:i=1,\dots,m-1\}
\]
and $t_1+\cdots+t_m=k$.

For $\e=\pm \f_{t_1,\dots,t_m}\in V\setminus\{\f_0\}$ let
\[
\tau(\e)=\pm (-1)^m\sum_{i=1}^m(-1)^it_i.
\]
We also set $\tau(\f_0)=0$.
\begin{lemma}\label{lem nonsplt}
For $\e\in V$ we have $\e\in V_{\tau(\e)}$.
\end{lemma}
\begin{proof}
Let $m,\,t_1,\dots,t_m\in\N$ and $\e=\pm\f_{t_1,\dots,t_m}$.
Since $\tau(\pm\f_{t_1})=\pm t_1$ the case when $m=1$ is straightforward.  For $m>1$ we proceed by induction on $t_1+\cdots+t_m$. Note that $\e\go{t_1}\e'$ where
\begin{equation}\label{eq: edge}
\e'=\begin{cases}  \pm\f_{t_1-1,t_2-1,t_3,\dots,t_m} & t_1,\,t_2>1 \\ \mp \f_{t_2-1,t_3,\dots,t_m} & t_1=1<t_2 \\ \pm f_{t_1-1+t_3,t_4,\dots,t_m} & t_1>t_2=1 \\ \pm \f_{t_3,\dots,t_m} & t_1=t_2=1.\end{cases} 
\end{equation}
In all cases we have $\tau(\e)=\tau(\e')$. By the induction hypothesis we have $\e'\in V_{\tau(\e)}$ and by Lemma \ref{lem conn}, we get that $\e\in V_{\tau(\e)}$.
\end{proof}

The combinatorial description of discrete series representations that have trivial partial cuspidal support is based on the graphs $\E_0$ and $\E_1$.
We finish this section with two simple lemmas concerning these subgraphs.
\begin{lemma}\label{lem v0}
Let $\f_0\ne \e=\pm\f_{t_1,\dots,t_m}\in V_0$. Then there exists a path $\e\path{\imath}\f_0$ with pattern $\imath$ of the form $(i_1,\dots,i_T,t_1,\dots,2,1)$ where $i_j>t_1+1$, $j=1,\dots,T$. 

If in addition $\e=(e_1,\dots,e_k)$ with $e_{t_1+2}=e_{t_1+1}(=e_{t_1})$ then there exists a path $\e\path{\jmath}\f_0$ with pattern $\jmath$ of the form $(i_1,\dots,i_S,t_1+1,t_1,\dots,2,1)$ where $i_j>t_1+2$, $j=1,\dots,S$.
\end{lemma}
\begin{proof}
We prove the first part by induction on $t_1+\cdots+t_m$. Let $\e\go{t_1}\e'$. It follows from Lemma \ref{lem conn} that $\e'\in V_0$. If $t_1=1$ it clearly follows that a path $\imath$ as required exists. This case includes the basis of induction. Assume now that $t_1>1$. By \eqref{eq: edge} we have that $\e'=\pm\f_{s_1,\dots,s_k}$ with $s_1\ge t_1-1$ (in particular $\e'\ne \f_0$). By induction there exists a path $\e'\path{\imath'}\f_0$ with pattern $\imath'$ of the form $(j_1,\dots,j_S,s_1,\dots,2,1)$ where $j_k>s_1+1$ for $k=1,\dots,S$. Thus, $\imath'$ is also of the form $(j_1,\dots,j_{S+s_1-t_1+1},t_1-1,\dots,2,1)$ where $j_k>t_1-1$ for $j=1,\dots,S+s_1-t_1+1$.
It is then straightforward that $\e\path{\imath}\f_0$ with $\imath=(j_1+2,\dots,j_{S+s_1-t_1+1}+2,t_1,\dots,2,1)$.
The first part follows.

 Assume now that $e_{t_1}=e_{t_1+2}$. Equivalently, assume that $m>2$ and $t_2=1$. Applying Lemma \ref{lem conn} let $\e'\in V_0$ be such that $\e\go{t_1+1} \e'$. Note that then $\e'=\pm\f_{s_1,\dots,s_k}$ with $s_1\ge t_1$. It therefore follows from the first part of the lemma that there exists a path $\e'\path{\jmath'}\f_0$ with pattern of the form $\jmath'=(j_1,\dots,j_S,s_1,\dots,2,1)$ and $j_k>s_1+1$ for $k=1,\dots,S$. Writing  $\jmath'=(j_1,\dots,j_{S+s_1-t_1},t_1,\dots,2,1)$ we have $j_k>t_1$ for $k=1,\dots,S+s_1-t_1$. It is then straightforward that $\e\path{\jmath}\f_0$ with $\jmath=(j_1+2,\dots,j_{S+s_1-t_1}+2,t_1+1,t_1,\dots,2,1)$. The lemma follows. 
\end{proof}

\begin{lemma}\label{lem v1}
Let $\f_1\ne \e=\pm\f_{t_1,\dots,t_m}\in V_1$. Then there exists a path $\e\path{\imath}\f_1$ with pattern $\imath$ of the form $(i_1,\dots,i_T,t_1,t_1-1,\dots,x)$ where $x\in \{1,2\}$, $x=1$ if $t_1=1$ and $i_j>t_1+1$, $j=1,\dots,T$. 

If in addition $\e=(e_1,\dots,e_k)$ with $e_{t_1+2}=e_{t_1+1}(=e_{t_1})$ then there exists a path $\e\path{\jmath}\f_1$ with pattern $\jmath$ of the form $(i_1,\dots,i_S,t_1+1,t_1,t_1-1,\dots,x)$ where $x\in \{1,2\}$ and $i_j>t_1+2$, $j=1,\dots,S$.
\end{lemma}
\begin{proof}
We prove the first part by induction on $t_1+\cdots+t_m$. Let $\e\go{t_1}\e'$. It follows from Lemma \ref{lem conn} that $\e'\in V_1$. If $t_1=1$ it clearly follows that a path $\imath$ as required exists. This case includes the basis of induction. Assume now that $t_1>1$. By \eqref{eq: edge} we have that $\e'=\pm\f_{s_1,\dots,s_k}$ with $s_1\ge t_1-1$. If $\e'=\f_1$ then necessarily $\e=(1,-1,-1)$ and the path $\e\path{(2)} \f_1$ satisfies the required condition with $x=2$. Assume now that $\e'\ne \f_1$.
By induction there exists a path $\e'\path{\imath'}\f_1$ with pattern $\imath'$ of the form $(j_1,\dots,j_S,s_1,s_1-1,\dots,x)$ where $x\in \{1,2\}$ and$j_k>s_1+1$ for $k=1,\dots,S$. If we write $\jmath=(j_1,\dots,j_{S+s_1-t_1+1},t_1-1,\dots,x)$ then $j_k>t_1-1$ for $j=1,\dots,S+s_1-t_1+1$.
It is then straightforward that $\e\path{\imath}\f_1$ with $\imath=(j_1+2,\dots,j_{S+s_1-t_1+1}+2,t_1,\dots,x)$.
The first part follows.

 Assume now that $e_{t_1}=e_{t_1+2}$. Applying Lemma \ref{lem conn} let $\e'\in V_0$ be such that $\e\go{t_1+1} \e'$. Note that then\ $\e'=\pm\f_{s_1,\dots,s_k}$ with $s_1\ge t_1$. It therefore follows from the first part of the lemma that there exists a path $\e'\path{\jmath'}\f_1$ with pattern of the form $\jmath'=(j_1,\dots,j_S,s_1,\dots,x)$ and $j_k>s_1+1$ for $k=1,\dots,S$. Writing  $\jmath'=(j_1,\dots,j_{S+s_1-t_1},t_1,\dots,2,1)$ we have $j_k>t_1$ for $k=1,\dots,S+s_1-t_1$. It is then straightforward that $\e\path{\jmath}\f_0$ with $\jmath=(j_1+2,\dots,j_{S+s_1-t_1}+2,t_1+1,t_1,\dots,x)$. The lemma follows. 
\end{proof}

\section{Vanishing in the discrete series class}\label{s_van_ds}

Let 
\[
\Irr_\U=\sqcup_{n=0}^\infty \Irr(U_{2n})
\]
where $\Irr(U_0)=\{\triv_0\}$ consists of the trivial representation of the trivial group.
In light of Proposition \ref{prop dp} we introduce the following class of representations.
Let $\Pi_\disc^\circ\subset \Irr_\U$ be the set of discrete series representations with trivial partial cuspidal support. We first provide an interpretation of the M{\oe}glin-Tadi\'c classification for this particular class of irreducible discrete series representations.

\subsection{Preliminaries on classification of discrete series representations}\label{ss_dsc}

\subsubsection{}\label{sec zseg} Let 
\[
\Cusp_{\GL}= \sqcup_{n=1}^\infty \Cusp(\GL_n(E)),\ \ \ \Irr_{\GL}=\sqcup_{n=0}^\infty \Irr(\GL_n(E))
\] 
and $\nu(g)=\abs{\det g}_E$ for any $g\in \GL_n(E)$ and $n\in \N$. A Zelevinsky segment is a subset of $\Cusp_{\GL}$ of the form
\[
[a,b]_{(\rho)}=\{\nu^i \rho\mid i=a,a+1,\dots,b\}
\]
where $\rho\in\Cusp_{\GL}$, $a,\,b\in \R$ and $b-a+2\in \N$ (so that $[a,a-1]_{(\rho)}$ is empty).

For a Zelevinsky segment $\Delta=[a,b]_{(\rho)}$ we associate the representation $L(\Delta)\in\Irr_{\GL}$, the unique irreducible quotient of $\nu^a\rho\times \nu^{a+1}\rho\times\cdots\times \nu^b\rho$. It is an essentially square integrable representation and all essentially square integrable representations in $\Irr_{\GL}$ are of this form.

For $\rho\in \Cusp_{\GL}$ and $a\in \N$ let 
\[
\delta(\rho,a)=L([\frac{1-a}2,\frac{a-1}2]_{(\rho)}).
\]
For a Zelevinsky segment $\Delta$ we also set
\[
\Delta^\vee=\{\rho^\vee\mid \rho\in \Delta\},\ \ \ \overline\Delta=\{\overline\rho\mid \rho\in \Delta\}\ \ \ {\rm and} \ \ \ \nu^a\Delta=\{\nu^a\rho\mid \rho\in \Delta\}, \ \ \ a\in \R. 
\]

\subsubsection{Cuspidal reducibility points}\label{ss_l_fn}
Let $\rho\in \Cusp_{\GL}$. If $\rho$ is unitary and $\rho^{\vee}\not\cong \overline{\rho}$ then $\nu^x\rho\rtimes \triv_0$ is irreducible for all $x\in \R$. If $\rho$ is conjugate self-dual, that is, such that $\rho^{\vee}\cong \overline{\rho}$ then there exists a unique $x\in \R_{\ge 0}$ such that $\nu^{x}\rho\rtimes \triv_0$ is reducible (and so does $\nu^{-x}\rho\rtimes \triv_0$ by Proposition \ref{prop_basic_ind} (3)). 
In fact, it is proved in \cite[Theorem 3.1]{MR1266747} that this $x$ is either $0$ or $\frac12$. This dichotomy for conjugate self-dual representations in $\Cusp_\GL$ is characterized in the literature in various different ways that we collect here together.
\begin{theorem}\label{thm pairity}
Let $\rho\in\Cusp_\GL$ be conjugate self-dual. The following are equivalent:
\begin{enumerate}
\item\label{pt red} $\nu^{\frac12}\rho\rtimes \triv_0$ is reducible;
\item \label{pt AS}The Asai-Shahidi $L$-function associated with $\rho$ (see e.g. \cite[\S 3]{MR1266747}) has a pole at $s=0$;
\item \label{pt fl}The twisted tensor $L$-function associated to $\rho$ by Flicker in the Appendix to \cite{MR1241802} has a pole at $s=0$;
\item \label{pt dist}$\rho$ is $\GL(F)$-distinguished.
\end{enumerate}
\end{theorem}
\begin{proof}
The equivalence of \eqref{pt red} and \eqref{pt AS} follows from \cite[Theorems 2.7, 2.8 and 2.9]{MR1266747}. The equivalence of  \eqref{pt AS} and  \eqref{pt fl} is immediate from \cite[Theorem 1.6]{MR2146859} where it is proved that the Asai-Shahidi and Flicker's twisted tensor $L$-functions coincide. The equivalence of  \eqref{pt fl} and  \eqref{pt dist} is \cite[Corollary 1.5]{MR2063106}).
\end{proof}

As we later observe, this dichotomy is also related to the parity of the $L$-parameter of $\rho$ (for instance see \cite[Theorem 4.9]{MR3739332}). With this in mind we make the following definition. 
\begin{definition}\label{def even}
Let $\rho\in \Cusp_{\GL}$ be conjugate self-dual. We say that $\rho$ is \emph{even} if the equivalent conditions \eqref{pt red}-\eqref{pt dist} of Theorem \ref{thm pairity} hold and \emph{odd} otherwise. 
\end{definition}
We will later see that the $L$-parameter of an even $\rho$ is conjugate orthogonal and of an odd $\rho$ is conjugate symplectic.

\subsubsection{}\label{sss_Admissible_datum}
Define an admissible datum to be a pair of the form $(\J,({\mbf{a_\rho}},\epsilon_\rho)_{\rho\in \J})$ where
\begin{itemize}
\item $\J$ is a finite set of conjugate self-dual representations in $\Cusp_{\GL}$. 
\item For each $\rho\in \J$ there is a natural number $k_\rho\in \N$, that is necessarily even if $\rho$ is odd, such that ${\mbf{a_\rho}}=(a_1,\dots,a_{k_{\rho}})$  with $a_1>\cdots>a_{k_{\rho}}\ge 0$ and such that $a_1,\dots,a_{k_{\rho}}$ are all in $\Z$ if $\rho$ is odd and all in $\frac{1}{2}+\mathbb Z$ if $\rho$ is even.
\item $\epsilon_\rho$ is a $k_{\rho}$-tuple of signs in $V_0\sqcup V_1$ (see Definition \ref{def:vi} for the notation).
\end{itemize}

\subsubsection{}

For $\pi\in \Pi_\disc^\circ$ and any $\rho\in \Cusp_{\GL}$ conjugate self-dual, let $\Jord_\rho(\pi)$ be the set of integers $a\in \N$ such that
\begin{itemize}
\item $a$ is even if $\rho$ is even and odd if $\rho$ is odd,
\item $\delta(\rho,a)\rtimes \pi$ is irreducible.
\end{itemize}
Let  $k_\rho(\pi)=\abs{\Jord_\rho(\pi)}$. Let 
\[
\J_\pi=\{\rho:k_\rho(\pi)>0\}
\]
and for $\rho\in \J_\pi$ let $\mbf{a_\rho}(\pi)=(a_1,\dots,a_{k_\rho(\pi)})$ where $a_1>\cdots>a_{k_\rho(\pi)}$ and 
\[
\Jord_\rho(\pi)=\{2a_i+1:i=1,\dots,k_\rho(\pi)\}.
\]
Set $\Jord(\pi)=\{(\rho,a)\mid \rho\in \J_\pi, a\in \Jord_{\rho}\}$.

M{\oe}glin defined the function $\epsilon_\pi:\Jord(\pi)\rightarrow \{\pm 1\}$ for $\pi\in \Pi_\disc^\circ$ (see \cite[\S 2]{MR1896238}) which we now recall. 

First suppose that $\rho$ is odd so that $\rho\rtimes \triv_0$ is reducible. It is also known that $\rho\rtimes \triv_0$ is a direct sum of two inequivalent tempered representations of which exactly one is generic. Call the generic component $\tau_{1}$. Now define $\epsilon_{\pi}((\rho,2a_{1}+1))=1$ if and only if there exists a representation $\tau\in \Irr_{\GL}$ such that $\pi\hookrightarrow \tau\times L([1,a_{1}]_{(\rho)})\rtimes \tau_{1}$. 

Suppose now that $\rho$ is even. In this case define $\epsilon_{\pi}((\rho,2a_{k_{\rho}(\pi)}+1))=1$ if and only if there exists a representation $\sigma \in \Irr_\U$ such that $\pi\hookrightarrow L([\frac{1}{2},a_{k_{\rho}(\pi)}]_{(\rho)})\rtimes \sigma$. 

In either of the two cases we now define $\epsilon_\pi$ on $\Jord_\rho(\pi)$ in the following recursive manner. For any $i\in \{1,\dots,k_{\rho}(\pi)-1\}$, we require that $\epsilon_{\pi}((\rho,2a_{i}+1))=\epsilon_{\pi}((\rho,2a_{i+1}+1))$ if and only if there exists a representation $\sigma \in \Irr_\U$ such that $\pi\hookrightarrow L([a_{i+1}+1,a_{i}]_{(\rho)})\rtimes \sigma$. 

Finally, we define $\epsilon_\rho(\pi)=(e_1,\dots,e_{k_\rho(\pi)})$ by setting $e_i=\epsilon_\pi((\rho,2a_i+1))$.

\subsubsection{}\label{sss_mt_dsc}
The M{\oe}glin-Tadi\'c classification implies that $\pi\mapsto (\J_\pi,({\mbf{a_\rho(\pi)}},\epsilon_\rho(\pi))_{\rho\in \J_\pi})$ is a bijection from $\Pi_\disc^\circ$ to the set of admissible data (see \cite[Theorem 6.1]{MR1896238}). 

\subsubsection{}\label{sss_ds_cons}
Next we explain how to extract from an admissible datum corresponding to a representation $\pi\in \Pi_\disc^\circ$, ways to realize $\pi$ as a quotient of a representation induced from an essentially square integrable representation of a standard Levi subgroup of the Siegel Levi. We apply the results in \cite[\S 7 and \S 9]{MR1896238}. We freely use below notation introduced in \S \ref{app_s:signs}.

M{\oe}glin and Tadi\'c realized the discrete series representations as subrepresentations of certain induced representations. For distinction problems it is more convenient to realize representations as quotients. We translate their results via contragredient (recall \eqref{eq cont ind}). Note that $\Pi_\disc^\circ$ is preserved by $\pi\mapsto \pi^\vee$ and $\J_{\pi^{\vee}}=\{\rho^\vee \mid \rho\in \J_\pi\}$, $\pi\in \Pi_\disc^\circ$.

Fix $\pi\in \Pi_\disc^\circ$ and $\rho\in \J_{\pi}$ and let $\mbf{a}={\mbf{a}}_{\rho^\vee}(\pi^{\vee})=(a_1,\dots,a_k)$ and $\epsilon=\epsilon_{\rho^\vee}(\pi^{\vee})\in V_x$ with $x\in\{0,\ 1\}$. For any path $\epsilon\path{\imath} \f_x$ we associate a representation $I_\rho(\mbf{a},\epsilon,\iota)$ as follows. Let $((x_1,y_1),\dots,(x_m,y_m))$ be the history of the path $\imath$ (in particular $m=[k/2]$). If $k=2m$ (i.e., $x=0$) we set
\[
I_\rho(\mbf{a},\epsilon,\iota)=L([-a_{x_1},a_{y_1}]_{(\rho)})\times \cdots\times L([-a_{x_m},a_{y_m}]_{(\rho)})
\]
  and if $k=2m+1$ (i.e., $x=1$) and $\{z\}=\{1,\dots,k\}\setminus \{x_1,\dots,x_m,y_1,\dots,y_m\}$ then 
\[
I_\rho(\mbf{a},\epsilon,\iota)=L([-a_{x_1},a_{y_1}]_{(\rho)})\times \cdots\times L([-a_{x_m},a_{y_m}]_{(\rho)}) \times L([-a_z,-1/2]_{(\rho)}).
\]

Now let $\pi\in \Pi_\disc^\circ$ and for any $\rho\in \J_{\pi}$ choose a path $\epsilon_{\rho^\vee}(\pi^\vee)\path{\imath_\rho}\f_{x(\rho)}$ where $x(\rho)=x(\rho,\pi)$ is either zero or one. Let $\J_{\pi}=\{\rho_1,\dots,\rho_n\}$ be ordered arbitrarily. Then $\pi$ is a quotient of the representation
\[
I_\pi(\imath_{\rho_1},\dots,\imath_{\rho_n})=I_{\rho_1}(\mbf{a_{\rho_1^\vee}(\pi^\vee)},\epsilon_{\rho_1^\vee}(\pi^\vee),\iota_{\rho_1})\times\cdots\times I_{\rho_n}(\mbf{a_{\rho_n^\vee}(\pi^\vee)},\epsilon_{\rho_n^\vee}(\pi^\vee),\iota_{\rho_n})\rtimes \triv_0.
\]

\subsection{Statement and proof of the vanishing result}
\subsubsection{}
One of our main applications of the geometric lemma is the following.
\begin{proposition}\label{prop main}
Let $\Delta_1,\dots,\Delta_m$ be segments satisfying the following property. There exists $k\le m$, such that 
\begin{itemize}
\item for some $\rho\in \Cusp_{\GL}$ conjugate self-dual we have $\Delta_{m+1-i}=[-a_i,b_i]_{(\rho)}$ for all $i=1,\dots,k$ where the difference between any two elements in $\{a_1,\dots,a_k,b_1,\dots,b_k\}$ is an integer, and $a_1>\cdots>a_k>b_k>\cdots>b_1$; 
\item for $i\le m-k$, $\Delta_{i}$ is of the form $[-a,b]_{(\rho')}$ where $\rho'\in \Cusp_{\GL}$ is unitary and either $\rho'\not\cong\rho$ or $b<a<b_1$. 
\end{itemize}
If the representation $\sigma=L(\Delta_{1})\otimes\cdots \otimes L(\Delta_{m})\otimes \triv_0$ (of the relevant Levi subgroup $M$ of $U_{2n}$) admits a relevant orbit (see Definition \ref{def co}) then $k$ is even and $a_{2i}=a_{2i-1}-1$, $i=1,\dots,k/2$.
\end{proposition}
\begin{proof}
We apply the notation of \S \ref{ss gen orb} for the relevant orbit. 
Applying \cite[\S9.5]{MR584084} write   
\[
r_{L,M}(\sigma)=\otimes_{\imath\in\Index} L(\Delta_{\imath})\otimes \triv_0
\] 
and recall that 
\[
\Delta_{(m+1-i,j_{m+1-i})}=[-a_i,x_i]_{(\rho)}
\] 
for some $x_i\le b_i$, $i=1,\dots,k$. 
Since the representation $L(\Delta)$ is generic it is not $\Sp$-distinguished for any Zelevinsky segment $\Delta$ (see \cite[Theorem 3.2.2]{MR1078382}). It therefore follows from Corollary \ref{cor dist adm} that $\set_+(w)$ is empty.
By assumption 
\[
\nu L(\overline{\Delta}_{(m+1-i,j_{m+1-i})}^{\vee})=L([1-x_i,a_i+1]_{(\rho)})\ne L(\Delta_\imath)
\] 
for any $\imath\in \Index$. We may therefore deduce from the second part of Corollary \ref{cor dist adm} that 
\[
(m+1-i,j_{m+1-i})\in \set,\ \ \ i=1,\dots,k.
\]
To complete the proof of this proposition we prove by induction that if $2i-1\le k$ then also $2i\le k$, $\tau(m+1-2i,j_{m+1-2i})=(m+2-2i,j_{m+2-2i})$ and $a_{2i}=a_{2i-1}-1$. 

It follows from Corollary \ref{cor dist adm} that 
\[
L(\Delta_{\tau(m,j_m)})=\nu L(\Delta_{(m,j_m)})=L([1-a_1,1+x_1]_{(\rho)}). 
\]
By the assumptions we must have $\tau(m,j_m)=(m-1,j_{m-1})$ and $a_2=a_1-1$. It also follows from the order constraints on $\tau$ (see \eqref{eq: c form} and \eqref{orbit_cond}) that $(n,j)\in \set$ if and only if $j=j_n$ for $n\in\{m-1,m\}$. 
Now assume that $\tau(m+1-2i,j_{m+1-2i})=(m+2-2i,j_{m+2-2i})$ and $a_{2i}=a_{2i-1}-1$ for $i=1,\dots,s-1$, that $(n,j)\in \set$ if and only if $j=j_n$ for $m+3-2s\le n\le m$ and that $2s-1\le k$. As above, by Corollary \ref{cor dist adm} we have
\[
L(\Delta_{\tau(m+2-2s,j_{m+2-2s})})=\nu L(\Delta_{(m+2-2s,j_{m+2-2s})})=L([1-a_{2s-1},1+x_{2s-1}]_{(\rho)}). 
\]
By induction hypothesis $\tau(m+2-2s,j_{m+2-2s})=(i,j)$ with $i\le m+2-2s$ and by our assumption this implies that $\tau(m+1-2s,j_{m+1-2s})=(m+2-2s,j_{m+2-2s})$, $a_{2s}=a_{2s-1}-1$ and $2s\le k$. The constraints \eqref{eq: c form} and \eqref{orbit_cond} on $\tau$ also imply that $(n,j)\in \set$ if and only if $j=j_n$ for $m+1-2s\le n\le m$. The proposition follows.
\end{proof}

\subsubsection{}We recall that the graph of signs is defined is \S \ref{app_s:signs}. We freely use the notation introduced there.
\begin{theorem}\label{thm:main_res_ds}
Let $\pi\in \Pi_\disc^\circ$. Suppose that there exists $\rho\in \J_{\pi}$, with $\epsilon_{\pi^\vee}(\rho^\vee)=(e_1,\dots,e_k)=\pm\f_{t_1,\dots,t_m}$ and ${\mbf{a_{\pi^\vee}(\rho^\vee)}}=(a_1,\dots,a_k)$, that satisfies at least one of the following three conditions:
\begin{enumerate}
\item \label{case odd}$t_1$ is odd,
\item \label{case jord}$a_{2i-1}>a_{2i}+1$ for some $i\le t_1/2$,
\item \label{case even}$e_{t_1+2}=e_{t_1+1}(=e_{t_1})$.
\end{enumerate}
Then $\pi$ is not $\Sp$-distinguished.
\end{theorem}
\begin{proof}
In order to show that $\pi$ is not $\Sp$-distinguished we realize it as a quotient of an induced representation $I$ that is not $\Sp$-distinguished. 
We separate into several cases, and in each case, writing $I=\ip_{G,M}(\sigma)$, we apply Theorem \ref{thm geoml} and show that $I$ is not $\Sp$-distinguished by showing that $\sigma$ has no relevant orbit.

Write $\J_{\pi}=\{\rho_1,\dots,\rho_n\}$ with $\rho=\rho_n$ and choose an arbitrary path $\epsilon_{\pi^\vee}(\rho_k^\vee)\path{\imath_k} \f_{x_k}$ with $x_k\in \{0,1\}$ for $k=1,\dots,n-1$. Let $\e=\epsilon_{\pi^\vee}(\rho^\vee)$. If $\e\in V_0$ then choose a path $\e\path{\imath_n}\f_0$ to satisfy the requirement of $\imath$ in Lemma \ref{lem v0} if either \eqref{case odd} or \eqref{case jord} above hold and $\jmath$ if \eqref{case even} holds. 
Then $I_\pi(\imath_1,\dots,\imath_n)$ is of the form $\ip_{G,M}(\sigma)$ where $\sigma$ satisfies the assumptions of Proposition \ref{prop main} with $k=t_1$ in Cases \eqref{case odd} or \eqref{case jord} and with $k=t_1+1$ in Case \eqref{case even}. It therefore follows from the proposition that $I_\pi(\imath_1,\dots,\imath_n)$ and therefore also $\pi$ is not $\Sp$-distinguished.

Assume now that $\e\in V_1$. Choose a path $\e\path{\imath_n}\f_1$ to satisfy the requirement of $\imath$ in Lemma \ref{lem v1} if either \eqref{case odd} or \eqref{case jord} above hold and $\jmath$ if \eqref{case even} holds. If the pattern of the path $\imath_n$ ends at $x=2$ then as in the previous case it follows that $I_\pi(\imath_1,\dots,\imath_n)$  is of the form $\ip_{G,M}(\sigma)$ where $\sigma$ satisfies the assumptions of Proposition \ref{prop main} with $k=t_1$ in Cases \eqref{case odd} or \eqref{case jord} and with $k=t_1+1$ in Case \eqref{case even}. It follows from the proposition that $\sigma$ has no contributing orbit and hence $\pi$ is not $\Sp$-distinguished. Assume now that $\imath_n$ ends at $x=1$. Then
\[
I_\pi(\imath_1,\dots,\imath_n)=L(\Delta_1)\times\cdots\times L(\Delta_m)\times L(\Delta)\rtimes \triv_0
\]
where $\sigma=L(\Delta_1)\otimes\cdots\otimes L(\Delta_m) \otimes \triv_0$ satisfies the assumptions of Proposition \ref{prop main} with $k=t_1$ in Cases \eqref{case odd} or \eqref{case jord} and with $k=t_1+1$ in Case \eqref{case even}. Furthermore $L(\Delta)=L([-a,-1/2]_{(\rho)})$ where in the notation of Proposition \ref{prop main} we have $0<a<b_1$ and $-1/2<b_1$. In other words, $[-a_i,b_i]_{(\rho)}$ contains $[-a,-1/2]_{(\rho)}$ and therefore $L(\Delta)\times L(\Delta_{m+1-i})\cong L(\Delta_{m+1-i})\times L(\Delta)$ for $i=1,\dots,k$. 
It follows that $I_\pi(\imath_1,\dots,\imath_n)$ is isomorphic to
\[
 L(\Delta_1)\times\cdots\times L(\Delta_{m-k})\times L(\Delta)\times L(\Delta_{m+1-k})\times\cdots\times L(\Delta_m)\rtimes \triv_0
\]
which is of the form $\ip_{G,M}(\sigma)$ where $\sigma$ satisfies the assumptions of Proposition \ref{prop main} with $k=t_1$ in Cases \eqref{case odd} or \eqref{case jord} and with $k=t_1+1$ in Case \eqref{case even}. Again it follows from the proposition that $\sigma$ has no contributing orbit and hence $\pi$ is not $\Sp$-distinguished.
\end{proof}

\section{Vanishing in the tempered class}\label{s_van_temp}
In this section we obtain a sufficient condition for a tempered representation to be not $\Sp$-distinguished.  
\subsection{Tempered representations of $U_{2n}$}\label{s_prelim_temp}
Given an irreducible tempered representation $\pi$ of $U_{2n}$, it is known due to Harish-Chandra that there exist irreducible discrete series representations $\pi_{\ds}\in \Irr(U_{2m})$ (for some $m\leq n$) and $\delta_{1},\dots,\delta_{t}\in \Irr_{\GL}$ such that $\pi$ is a quotient (in fact a direct summand) of the unitary representation $\delta_{1}\times \cdots\times \delta_{t}\rtimes \pi_{{\rm ds}}$. Furthermore, $\pi_{{\rm ds}}$ and the multi-set $\{\delta_1,\overline{\delta_1}^\vee,\dots,\delta_t,\overline{\delta_t}^\vee\}$ are uniquely determined by $\pi$ (see \cite[Proposition III.4.1]{MR1989693}). 
We remark further that $\delta_1,\dots,\delta_t$ can be reordered arbitrarily and each $\delta_i$ can be replaced by $\overline{\delta_i}^\vee$.

Following \cite[Definition 5.4]{MR3123571} we define the Jordan set of the tempered representation $\pi$ in the following manner. Let $\delta_{i}=\delta(\rho_i,a_i)$ where $\rho_{i}\in \Cusp_{\GL}$ is unitary $i=1,\dots,t$. Then define $\Jord(\pi)$ to be the multi-set
\[
\{(\rho_{1},a_{1}),\dots, (\rho_{t},a_{t}), (\overline{\rho_{1}}^{\vee},a_{1}),\dots,(\overline{\rho_{t}}^{\vee},a_{t})\}\cup \Jord(\pi_{{\rm ds}}).
\]
Note that, unlike in the case of discrete series representations, $\Jord(\pi)$ can have multiplicities.

Define $\Pi_{{\rm temp}}^\circ\subset \Irr_\U$ to be the class of tempered representations with trivial partial cuspidal support.  Thus, $\pi\in \Pi_\temp^\circ$ if and only if $\pi_{{\rm ds}}\in \Pi_\disc^\circ$. 

If $\pi\in \Irr_\U$ is tempered and not in $\Pi_{{\rm temp}}^\circ$ then $\pi$ is not $\Sp$-distinguished by Proposition \ref{prop dp}. We can also exclude distinction for many representations in $\Pi_\temp^\circ$.

\subsection{Vanishing result on tempered representations}

For $\rho\in \cusp_{\GL}$ let $\exp(\rho)\in \R$ be the unique real number $s$ such that $\nu^{-s}\rho$ is unitary.
For a segment $\Delta=[a,b]_{\rho}$ let $\exp(\Delta)=(a+b)/2+\exp(\rho)$. (Thus, $L(\Delta)$ is unitary if and only if $\exp(\Delta)=0$.)

\begin{proposition}\label{prop_van_temp}
Let $\pi\in \Pi_{\temp}^\circ$ be such that there exists $(\rho,a)\in\Jord(\pi)$ which satisfies at least one of the following conditions:
\begin{enumerate}
\item $\rho$ is not conjugate self-dual.
\item $\rho$ is conjugate self-dual and for $(\rho^\vee,b)\in \Jord_{\rho^{\vee}}(\pi_{{\rm ds}}^{\vee})$ we have that $a\not\equiv b\ \mod \ 2$.
\item $\rho$ is conjugate self-dual and for $(\rho^\vee,b)\in \Jord_{\rho^{\vee}}(\pi_{{\rm ds}}^{\vee})$ we have that $b\le a$.
\end{enumerate}
Then $\pi$ is not $\Sp$-distinguished. 
\end{proposition}
\begin{proof}
Without loss of generality we may choose $a$ maximal so that one of the three assumptions hold. Thus, if $(\rho,b)\in \Jord(\pi)$ but $(\rho^\vee,b)\not\in\Jord(\pi_{\ds}^\vee)$ then $b\le a$. By interchanging $\rho$ and $\overline{\rho^{\vee}}$ if necessary, we can further assume that $a\ge b$ whenever $(\overline{\rho^{\vee}},b)\in \Jord(\pi)$ and $(\overline{\rho},b)\not\in\Jord(\pi_{\ds}^\vee)$.

With our assumption, it follows from the discussion in \S \ref{s_prelim_temp} and \S\ref{sss_ds_cons} that $\pi$ can be written as an irreducible quotient of an induced representation of the form
\[
I_{\pi}=L(\Delta_{1})\times \cdots\times  L(\Delta_{k})\rtimes \triv_0
\] 
where 
\[
\Delta_1=[\frac{1-a}2,\frac{a-1}2]_{(\rho)}
\] 
and either $\Delta_i\subseteq \Delta_1$ or $\Delta_i$ is disjoint from $\{\nu^{j+(a-1)/2}\rho:j\in \Z\}$ for $i=2,\dots,k$.

Assume by contradiction that $\pi$ is $\Sp$-distinguished. Then $I_\pi$ is also $\Sp$-distinguished and by Theorem \ref{thm geoml} the representation $\sigma=L(\Delta_{1})\otimes \cdots\otimes  L(\Delta_{k})\otimes \triv_0$ has a relevant orbit.

We apply the notation of \S \ref{ss gen orb} for the relevant orbit.  Let $x=(a-1)/2$.
By \cite[\S9.5]{MR584084} write 
\[
r_{L,M}(L(\Delta_{1})\otimes \cdots \otimes L(\Delta_{k})\otimes \triv_0)=\otimes_{\imath\in\Index} L(\Delta_{\imath})\otimes \triv_0
\] 
and recall that 
\[
\Delta_{(1,j_1)}=[-x,y]_{(\rho)}
\] 
for some $y\in \{-x,-x+1,\dots,x\}$ and that by assumption 
\[
\nu^{-x-1}\rho,\nu^{1+x}\overline{\rho^{\vee}}\notin \cup_{i=1}^{k}\Delta_{i}.
\]
By \cite[Theorem 3.2.2]{MR1078382} and Corollary \ref{cor dist adm} the set $\set_+(w)$ is empty. But we now have
\[
\nu^{-1}\Delta_{(1,j_1)}=[-x-1,y-1]_{(\rho)}\ne \Delta_\imath, \ \ \ \imath\in \Index
\]
and
\[
\nu \overline{\Delta_{(1,j_1)}}^\vee=[1-y,1+x]_{(\overline\rho^\vee)}\ne \Delta_\imath, \ \ \ \imath\in \Index
\]
which contradicts Corollary \ref{cor dist adm}. The proposition follows.

\end{proof}

\section{Representation theory of general linear groups}\label{s_rep_gln}

Before we continue with further applications of the geometric lemma to $\Sp$-distinction, we need to introduce some further notation and recall the Langlands and Zelevinsky classifications of $\Irr_{\GL}$. 

Zelevinsky segments (henceforth, simply segments) were defined in \S \ref{sec zseg}. For a segment $\Delta=[a,b]_{(\rho)}$, we denote by $b(\Delta)=\nu^{a}\rho$ its beginning, by $e(\Delta)=\nu^{b}\rho$ its end and by $\ell(\Delta)=b-a+1$ its length.

\subsection{Classification of $\Irr_\GL$}\label{ss_irr_class}

We refer to \cite{MR584084} for the results stated in this section.

\subsubsection{Segment representations}
Let $\Delta=[a,b]_{(\rho)}$ be a Zelevinsky segment. 
The representation $\nu^{a}\rho \times \nu^{a+1}\rho \times \cdots\times \nu^{b}\rho$ has a unique irreducible subrepresentation which we denote by $Z(\Delta)$. 
It is also the unique irreducible quotient of $\nu^b\rho\times \nu^{b-1}\rho\times\cdots\times \nu^a\rho$. 
By convention, if the segment $\Delta$ is empty then $Z(\Delta)$ is the trivial representation of the trivial group.

The representations $Z(\Delta)$ are the building blocks in the Zelevinsky classification of $\Irr_\GL$.

\subsubsection{}
Two segments $\Delta_{1}$ and $\Delta_{2}$ are \emph{linked} if $\Delta_{1}\cup \Delta_{2}$ is a segment different from $\Delta_1$ and from $\Delta_2$. If $\Delta_{1}$ and $\Delta_{2}$ are linked and $b(\Delta_{2})=\nu^{a}b(\Delta_{1})$ for some $a\in \N$, then we say that $\Delta_{1}$ precedes $\Delta_{2}$ and write $\Delta_{1}\prec \Delta_{2}$.

\subsubsection{}
Let $\mathcal O$ be the set of finite multi-sets of segments. For $\mult\in \cO$ we write $\mult=\{\Delta_1,\dots,\Delta_t\}$ as an unordered $t$-tuple. 
(Here, we may have $\Delta_i=\Delta_j$ for $i\ne j$.) When choosing a specific order, by abuse of notation we write $\mult=(\Delta_1,\dots,\Delta_t)$ as an ordered $t$-tuple.

We say that $\mult$ is ordered in \emph{standard form} if $\Delta_i\not\prec\Delta_j$ for all $1\le i <j\le t$. For example, by requiring that $b(\Delta_i)\ge b(\Delta_{i+1})$ for $i=1,\dots,t-1$ we obtain a standard form.

We also let $\mult^\vee=\{\Delta_1^\vee,\dots,\Delta_t^\vee\}$ and similarly define the multi-sets $\overline{\mult}$ and $\nu^a\mult$, $a\in \C$.
\subsubsection{}
Let $\mathfrak m=\{\Delta_1,\dots,\Delta_t\}\in\cO$ be ordered in standard form. The representation
\begin{equation*}
Z(\Delta_1) \times\cdots \times Z(\Delta_t)
\end{equation*}
is independent of the choice of order of standard form. It has a unique irreducible submodule that we denote by $Z(\mathfrak m)$.

\subsubsection{} The Zelevinsky classification says that the map $\mathfrak m\mapsto Z(\mathfrak m):\cO\rightarrow \Irr_{\GL}$ is a bijection.

\subsubsection{}
The representation
\[\tilde\zeta(\mult)=Z(\Delta_t) \times\cdots \times Z(\Delta_1)
\]
is also independent of the choice of standard order on $\mult$ and $Z(\mult)$ is the unique irreducible quotient of $\tilde\zeta(\mult)$.

\subsubsection{} For a segment $\Delta$ the representation $L(\Delta)$ is defined in \S \ref{sec zseg}. We remark that $\Delta\mapsto L(\Delta)$ is a bijection between the set of segments and the subset of essentially square-integrable representations in $\Irr_{\GL}$.

\subsubsection{}
Let $\mathfrak m=(\Delta_1,\dots,\Delta_t)\in\cO$ be ordered in standard form. The representation
\begin{equation*}
L(\Delta_1)\times\cdots\times L(\Delta_t)
\end{equation*}
is independent of the choice of order of standard form. It has a unique irreducible quotient that we denote by $L(\mathfrak m)$.

\subsubsection{}\label{sss: Langlands_class}
The Langlands classification says that the map $\mathfrak m\mapsto L(\mathfrak m) :\cO\rightarrow \Irr_{\GL}$ is a bijection.

\subsubsection{}\label{mwalgo}
It follows from the two classifications above that for any $\mathfrak m\in \cO$ there exists a unique $\mathfrak m^t\in \cO$ such that $Z(\mathfrak m)=L(\mathfrak m^t)$. The function $\mathfrak m\mapsto \mathfrak m^t$ is an involution on $\cO$. Given a multi-set $\mathfrak m$, an algorithm to compute $\mathfrak m^{t}$ is provided in \cite{MR863522}. 

For $\pi=Z(\mathfrak m)\in \Irr_{\GL}$, let $\pi^t=L(\mathfrak m)$. The map $\pi\mapsto \pi^t$ is the Zelevinsky involution on $\Irr_{\GL}$.

\subsubsection{} For an irreducible cuspidal $\rho\in \Cusp_{\GL}$ define its {\it cuspidal line}
\begin{equation*}
\rho^\Z=\{\nu^{m}\rho\mid m\in \mathbb Z\}.
\end{equation*}
To $\rho^\Z$ we transfer the standard order $\le$ on $\Z$. That is, for $\rho,\,\rho'\in \Cusp_{\GL}$ we write $\rho\le \rho'$ if $\rho'=\nu^n\rho$ for some $n\in \N$.
\subsubsection{}\label{def: cusp supp}
For every $\pi\in \Irr_{\GL}$ there exist $\rho_1,\dots,\rho_k\in \Cusp_{\GL}$, unique up to rearrangement, so that $\pi$ is isomorphic to a subrepresentation of $\rho_1\times \cdots \times \rho_k$. Let the multi-set of cuspidal representations 
\[
\supp(\pi)=\{\rho_i\mid i=1,\dots,k\}
\] 
be the cuspidal support of $\pi$.

\subsubsection{} For $\mult\in\OO$ define the multi-set 
\[
\supp(\mult)=\{\rho\in \Cusp_{\GL}\mid \rho\in\Delta\text{ for some }\Delta\in \mult\}
\] 
to be the cuspidal support of $\mult$.
We then have $\supp(\mult)=\supp(Z(\mult))=\supp(L(\mult))$.

\begin{definition}\label{def: cusp supp_1}
We say that a representation $\pi\in \Irr_\GL$ (resp. multi-set $\mult\in \cO$) is \emph{rigid} if $\Supp(\pi)\subseteq \rho^\Z$ (resp. $\Supp(\mult)\subseteq \rho^\Z$) for some $\rho\in \Cusp_\GL$.
We then also say that $\pi$ (resp. $\mult$) is supported on $\rho^\Z$.
\end{definition}

\subsubsection{Exponent of a representation}\label{sss exp}
For a representation $\pi\in \Irr_{\GL}$ with central character $\omega_{\pi}$ let $\alpha=\exp(\pi)\in \R$ be the exponent of $\pi$. It is the unique real number such that $\nu^{-\alpha}\omega_\pi$ is a unitary character. 

For a segment $\Delta$ we have 
\[
\exp(\Delta)=\exp(Z(\Delta))=\exp(L(\Delta)).
\]

\subsection{Langlands parameters}
Let $W_E$ denote the Weil group of $E$ and 
\[
W'_E= W_E\times {\rm SL}_{2}(\mathbb C)
\] 
the Weil-Deligne group. The Langlands parameter (henceforth, $L$-parameter) of an irreducible representation of $\GL_{n}(E)$ is an $n$-dimensional continuous semi-simple complex representation of the group $W'_{E}$. Let $\Phi(\GL_{n}(E))$ be the set of all equivalence classes of $L$-parameters for the group $\GL_{n}(E)$. The Langlands reciprocity map, is the bijection 
\[
{\rm rec_n}:\Irr(\GL_{n}(E))\rightarrow\Phi(\GL_{n}(E)) 
\]
established in \cite{MR1876802}. We denote by ${\rm rec_\GL}:\sqcup_{n=1}^\infty \Irr(\GL_{n}(E))\rightarrow \sqcup_{n=1}^\infty \Phi(\GL_{n}(E))$ the union of these bijections for all $n$.
The map ${\rm rec_\GL}$ restricts to a bijection from $ \Cusp_\GL$ to the set of irreducible, complex representations of $W_{E}$. 

In \cite{MR584084} Zelevinsky reduced the Langlands reciprocity map to its restriction to $\Cusp_\GL$ as follows. Let $\Delta$ be a segment such that $\ell(\Delta)=x$. Write $\Delta=[{-\frac{(x-1)}{2}},{\frac{(x-1)}{2}}]_{(\rho)}$ where $\rho\in \Cusp_{\GL}$. Then we have that 
\[
{\rm rec_\GL}(L(\Delta))={\rm rec_\GL}(\rho)\otimes {\rm Spc}(x)
\] 
where ${\rm Spc}(x)$ denotes the unique irreducible algebraic $x$-dimensional representation of ${\rm SL}_2(\mathbb C)$. Clearly, the parameter ${\rm rec_\GL}(L(\Delta))$ is an irreducible representation of $W_{E}'$. 

Finally if $\pi=L(\mathfrak m)$ where $\mathfrak m=\{\Delta_{1},\dots,\Delta_{t}\}$, then we have that
\begin{equation}\label{eq_lp_gl}
{\rm rec_\GL}(\pi)=\oplus_{i=1}^{t}{\rm rec_\GL}(L(\Delta_{i})).
\end{equation}

For an $L$-parameter $\phi$, denote by ${}^{c}\phi^{\vee}$ its conjugate-dual parameter (see for instance \cite[\S 2.2]{MR3338302} for the precise definition). We will use several times in this article the following fact (\cite[Lemma~VII.1.6]{MR1876802}):
\begin{theorem}\label{thm:rec_gal}
 We have
  \[
    {\rm rec_\GL}(\overline{\pi}^{\vee}) = {}^{c}({\rm rec_\GL}(\pi))^{\vee}.
  \]
\qed
\end{theorem}

We say that a parameter $\phi\in \Phi(\GL_n(E))$ is conjugate self-dual if ${}^c\phi\cong\phi^\vee$. We refer to \cite[\S 2.2]{MR3338302} for the notion of a conjugate self-dual parameter with parity $\pm1$  and we also apply the terminology of \cite[\S3]{MR3202556} to say that $\phi\in \Phi(\GL_n(E))$ conjugate self-dual is conjugate orthogonal if it has parity $1$ and conjugate symplectic if it has parity $-1$.

A conjugate self-dual cuspidal representation always has parity.
The results of Mok \cite[Lemma 2.2.1 and Theorem 2.5.1]{MR3338302} combined with \cite[Theorems 2.7 and 2.8]{MR1266747} imply the following.
\begin{theorem}\label{thm evod}
Let $\rho\in \Cusp_\GL$ be conjugate self dual. Then $\rho$ is even (see Definition \ref{def even}) if and only if $\rec_\GL(\rho)$ is conjugate orthogonal. \qed
\end{theorem}
\begin{remark}\label{rmk csq}
As a consequence, a conjugate self-dual discrete series representation in $\Irr_\GL$ always has parity determined as follows. Let $\rho\in \Cusp_\GL$ be conjugate self-dual and $a\in \N$. Then $\delta(\rho,a)$ is conjugate self-dual with parity $(-1)^{a-1}\eta_\rho$ where $\eta_\rho$ is the parity of $\rho$.
\end{remark}

\subsection{Ladder representations}\label{ss:ladder}
The class of ladder representations was introduced in \cite{MR3163355}. The Jacquet modules of a ladder representation are calculated explicitly in \cite[Corollary 2.2]{MR2996769}. Moreover, this class is preserved by the Zelevinsky involution and the algorithm provided in \cite{MR863522} to compute the Zelevinsky involution of an irreducible representation takes a much simpler form when the representation is a ladder (see \cite[\S 3]{MR3163355}). Some of these results and structural properties make this class more approachable in comparison to the entire admissible dual for the purpose of distinction problems (for instance see \cite{MR3628792}). We will now recall their definition. 

\subsubsection{Definition of ladder representations}\label{sss:ladder}
\begin{definition} 
An (ordered) multi-set $\mult\in \cO$ is called  a \emph{ladder} if there exist $\rho\in \Cusp_{\GL}$ and integers
\[
a_{1}>\cdots>a_{k}\ \ \  \text{and} \ \ \ b_{1}>\cdots>b_{k} 
\]
such that $\mult=(\Delta_1,\dots,\Delta_k)$ where $\Delta_i=[a_i,b_i]_{\rho}$.
A representation $\pi\in \Irr_{\GL}$ is called a ladder representation if $\pi=L(\mathfrak m)$ (or equivalently, if $\pi=Z(\mathfrak m)$) where $\mathfrak m$ is a ladder.
\end{definition}

\subsubsection{}\label{sss:speh}
In the above notation, the ladder representation $\pi=L(\mathfrak m)$ is called a Speh representation if 
\begin{equation}\label{eq_def_speh}
a_{i}=a_{i+1}+1 \ {\rm and} \ b_{i}=b_{i+1}+1 \ {\rm for}\  i=1,\dots,k-1. 
\end{equation}

We recall the fact that any unitary representation in $\Irr_{\GL}$ can be obtained as a parabolic induction of Speh representations (see \cite[Theorem D]{MR870688}).

\section{Further applications of the geometrical lemma to distinction}\label{s_geo_app}

Let $P=M \ltimes U$ be a standard parabolic subgroup of $U_{2n}$, contained in the Siegel parabolic, with its standard Levi decomposition and fix a double coset in $P\bs U_{2n}/\Sp_{2n}$. We will now collect some simple consequences of the geometric lemma (described in \S \ref{s_geo_lem}) that we use later in the article. Throughout this section we freely apply the notation of \S \ref{ss gen orb} for the chosen orbit. 
\begin{lemma}\label{aux_geo_lem}
Let $\imath=(c,d)\in \Index$ be the minimal index in $\set$ such that $\tau(\imath)\neq \imath$. Then $\tau(\imath)=(a,j_{a})$ for some $a>c$.
\end{lemma}
\begin{proof}
Write $\tau(\imath)=(a,b)$. By the minimality of $\imath$ and (\ref{orbit_cond}), it follows easily that $c<a$. Assume if possible that $b<j_{a}$. By \eqref{eq: c form} we have $(a,j_{a})\in \set$ and by \eqref{orbit_cond} we have $\tau(a,j_{a})\prec\tau(a,b)=\imath$ which is contradicting the minimality of $\imath$. 
\end{proof}

\begin{lemma}\label{aux_geo_lem2}
Let $\rho\in \Cusp_\GL$ and $\mathfrak m=(\Delta_{1},\dots,\Delta_{t})\in \mathcal O$ be an ordered multi-set with $\Delta_i=[a_i,b_i]_{(\rho)}$, $i=1,\dots,t$ and $b_1\le \cdots\le b_t$. Suppose that $w\in {}_{M}W_{M}\cap W[2]$ is relevant to $Z(\Delta_{1})\otimes\cdots\otimes Z(\Delta_{t})\otimes \triv_0$. For the orbit corresponding to $w$ we have $\set=\set_+(w)$. In particular, $\set$ is of the form 
\[
\set=\{(a,j_a):a\in A\}
\]
for some $A\subseteq \{1,\dots,t\}$. 
\end{lemma}
\begin{proof}
Let $L=M\cap w^{-1}Mw$ and applying \cite[Proposition 3.4]{MR584084}, in the above notation write 
\[
r_{L,M}(Z(\Delta_{1})\otimes\cdots\otimes Z(\Delta_{t})\otimes \triv_0)=\otimes_{\iota\in\Index} Z(\Delta_\iota)\otimes \triv_0.
\]
Then $\Delta_{(i,j_i)}=[y_i,b_i]_{(\rho)}$ for some $y_i\ge a_i$ for $i=1,\dots,t$.
Assume by contradiction that there exists an index in $\set$ not fixed by the involution $\tau$ and denote the minimal such index by $\imath=(c,d)$. By Lemma \ref{aux_geo_lem}, there exists $a\in \{c+1,\dots,t\}$ such that $\tau(\imath)=(a,j_{a})$. Thus, by Corollary \ref{cor dist adm} we have  $\nu^{-1}Z(\Delta_{\imath})=Z(\Delta_{(a,j_{a})})$. This implies that $\Delta_\imath=[y_a+1,b_a+1]_{(\rho)}$ which contradicts the hypothesis on $\mult$ and the fact that $c<a$. Thus, $\set=\set_+(w)$. The last part of the lemma follows from (\ref{orbit_cond}).
\end{proof}

\begin{lemma}\label{hered}
Let $\pi_{1}\in \Alg(\GL_{p_{1}}(E))$ and $\pi_{2}\in \Alg(\GL_{2p_{2}}(E))$ be such that $\nu^{-1/2}\pi_{1}$ is $\GL(F)$-distinguished and $\pi_{2}$ is $\Sp$-distinguished ($p_{1},p_{2}\geq 0$). Then $\pi_1\times \pi_2\rtimes 1$ is $\Sp$-distinguished.
\end{lemma}
\begin{proof}
This follows from a combination of an open orbit and a closed orbit argument for lifting distinguished representations via parabolic induction. Indeed, it follows from \cite[Proposition 7.2]{MR3541705} that the representation $\pi_2\rtimes 1$ is $\Sp$-distinguished and now the lemma follows from \cite[Proposition 7.1]{MR3541705}.
\end{proof}

\begin{lemma}\label{lem:first_case}
Let $\rho_1,\dots,\rho_k\in \Cusp_{\GL}$ be unitary representations such that $\rho_{i}\ncong \rho_{j},\overline{\rho_{j}}^{\vee}$ if $i\neq j$. Let 
\[
\tau=L([a_1,b_1]_{(\rho_{1})})\times \cdots\times L([a_k,b_k]_{(\rho_{k})})
\]
where $a_{i},b_{i}\in \mathbb R$ are such that $b_{i}-a_{i}\in \mathbb Z_{\geq 0}$, $i=1,\dots,k$. Then the representation $\tau\rtimes \triv_0\in \Alg(U_{2n})$ is $\Sp$-distinguished if and only if $\nu^{-1/2}\tau$ is $\GL(F)$-distinguished. In particular, if $\tau\rtimes \triv_0$ is $\Sp$-distinguished then $\exp([a_i,b_i]_{(\rho_{i})})=\frac{1}{2}$ for all $i=1,\dots,k$.
\end{lemma}
\begin{proof}
Let $\Delta_{i}=[a_i,b_i]_{(\rho_{i})}$ for $i=1,\dots,k$. If $\nu^{-1/2}\tau$ is $\GL(F)$-distinguished, then the representation $\tau\rtimes \triv_0$ is $\Sp$-distinguished by Lemma \ref{hered}. To prove the `only if' part, suppose that $\tau\rtimes \triv_0$ is $\Sp$-distinguished. Let $\sigma=L(\Delta_{1})\otimes\cdots\otimes L(\Delta_{k})\otimes \triv_0$ and $M$ be the standard Levi subgroup of $U_{2n}$ corresponding to $\sigma$. Now $\Sp$-distinction of $\tau\rtimes \triv_0$ implies that there exists a $w\in {}_{M}W_{M}\cap W[2]$ which is relevant to $\sigma$. Let $L=M\cap w^{-1}Mw$ and applying \cite[\S9.5]{MR584084} write 
\[
r_{L,M}(L(\Delta_{1})\otimes\cdots\otimes L(\Delta_{k})\otimes \triv_0)=\otimes_{\imath\in\Index} L(\Delta_\imath)\otimes \triv_0
\]
(in the notation of \S \ref{ss gen orb}). By the hypothesis, Corollary \ref{cor dist adm} and the fact that no essentially square integrable representation in $\Irr_\GL$ is $\Sp$-distinguished (by \cite[Theorem 3.2.2]{MR1078382}), we obtain that $j_{i}=1$, $i=1,\dots,k$, $\set=\emptyset$ and the involution $\tau$ on $\Index$ is trivial. Thus $\nu^{-1/2}L(\Delta_{i})$ is $\GL(F)$-distinguished for each $i\in \{1,\dots, k\}$. By \cite[Proposition 12]{MR1111204} $\exp(\Delta_{i})=\frac{1}{2}$. The representation $\nu^{-1/2}\tau$ is $\GL(F)$-distinguished by \cite[Proposition 26]{MR1194271}. 
\end{proof}

\section{$L$-parameters for the unitary groups and the stable base change map}\label{s_rep_u_n}
We now collect some preliminaries relevant to this article on the $L$-parameters of quasi-split unitary groups and the stable base change map. We begin by recalling the Langlands quotient theorem for these groups.
\subsection{Langlands quotient theorem}

\subsubsection{Standard modules of $U_{2n}$}
\begin{definition} 
A standard module $\lambda \in \Alg(U_{2n})$ is a representation of the form
\[
\nu^{x_{1}}\tau_{1}\times \cdots\times \nu^{x_{k}}\tau_{k}\rtimes \tau_{\temp}
\]
where $k\geq 0$, $\tau_{i}\ (1\leq i \leq k)$ are tempered representations in $\Irr_{\GL}$, $\tau_{\temp}$ is a tempered representation in $\Irr_\U$ and $x_{1}>\cdots > x_{k}>0$. 
\end{definition}
The representation $\nu^{x_{1}}\tau_{1}\times \cdots\times \nu^{x_{k}}\tau_{k}$ is the $\GL$-part of $\lambda$.  

\subsubsection{} We now state the Langlands quotient theorem for $U_{2n}$ (see \cite{MR0507262} for a proof). 
\begin{theorem}\label{thm_lqt}
A standard module $\lambda \in \Alg(U_{2n})$ has a unique irreducible (Langlands) quotient that we denote by $LQ(\lambda)$. The map $\lambda\mapsto LQ(\lambda)$ from standard modules in $\Alg(U_{2n})$ to $\Irr(U_{2n})$ is a bijection.    
\qed
\end{theorem}

\subsection{Stable base change}\label{ss sbc}
We will now recall the definition and some preliminary facts on the stable base change map from $\Irr(U_{2n})$ to $\Irr(\GL_{2n}(E))$.
\subsubsection{L-parameter for $U_{2n}$}
Let $W_F$ denote the Weil group and $W'_F:= W_F\times {\rm SL}_{2}(\mathbb C)$ the Weil-Deligne group of $F$. The Langlands dual group of $U_{2n}$ is the semi-direct product ${}^{L}U_{2n}=\GL_{2n}(\mathbb C)\rtimes W_{F}$ where the action of  $W_{F}$ factors through the Galois group ${\rm Gal}(E/F)$. An $L$-parameter of $U_{2n}$ is a map $W_{F}'\to {}^{L}U_{2n}$ satisfying several properties (see, for instance, \cite[\S 8]{MR3202556} or \cite[\S 2.2]{MR3338302}). Let $\Phi(U_{2n})$ denote the set of all equivalence classes of $L$-parameters of the group $U_{2n}$. The Langlands reciprocity map for this case was established by Mok (see \cite[Theorem 2.5.1]{MR3338302}). We denote by ${\rm rec_\U}$ the union over all $n\in \N$ of the finite to one surjective maps from $\Irr(U_{2n})$ to $\Phi(U_{2n})$ defined by Mok. 

Given $\pi\in \Irr(U_{2n})$, the fiber of the map ${\rm rec_\U}$ containing $\pi$ is the $L$-packet of $\pi$. 

\subsubsection{The restriction map $\xib'$}
For $\phi\in \Phi(U_{2n})$ the restriction $\phi|_{W_{E}'}$ lies in $\Phi(\GL_{2n}(E))$. Denote this restriction map $(\phi\mapsto \phi|_{W_{E}'}):\sqcup_{n=1}^\infty\Phi(U_{2n})\to \sqcup_{n=1}^\infty\Phi(\GL_{2n}(E))$ by $\xib'$. The map $\xib'$ is injective (see \cite[Theorem 8.1 (ii)]{MR3202556} or \cite[\S 2.2]{MR3338302}).

\subsubsection{} Stable base change $\xib$ is the functorial transfer from $\Irr(U_{2n})$ to $\Irr(\GL_{2n}(E))$ defined by
\[
\xib'(\rec_\U(\pi))=\rec_\GL(\xib(\pi)).
\]

\subsubsection{Image of $\xib$} The next result follows from \cite[Lemma 2.2.1 and Theorem 2.5.1]{MR3338302} (see also \cite[Theorem 8.1]{MR3202556}). 
\begin{proposition}\label{prop:im_bc}
A representation $\pi\in \Irr(\GL_{2n}(E))$ is in the image of the map $\xib$ if and only if $\pi^{\vee}\cong \overline{\pi}$ and $\rec_\GL(\pi)$ is conjugate symplectic. 
\qed
\end{proposition}

In what follows we provide some information about the stable base change fibers ($L$-packets). 
First, we remark that this task is reduced to tempered $L$-packets as follows.

\subsubsection{}\label{sss rtt} The map $\xib$ takes tempered representations to tempered representations. 
If $\pi\in \Irr_\GL$ lies in the image of the base change map then it is the Langlands quotient of a standard module of the form 
\[
\nu^{x_1}\tau_1\times \cdots\times \nu^{x_k}\tau_k \times \pi_\temp \times \nu^{-x_k}\overline{\tau_k}^\vee\times\cdots\times\nu^{-x_1}\overline{\tau_1}^\vee
\]
for unique real numbers $x_1>\cdots>x_k>0$ and tempered representations $\tau_1,\dots,\tau_k$ and $\pi_\temp$ in $\Irr_\GL$ so that $\rec_\GL(\pi_\temp)$ is conjugate symplectic. We then have
\[
\xib^{-1}(\pi)=\{\LQ(\nu^{x_1}\tau_1\times \cdots\times \nu^{x_k}\tau_k \rtimes \tau_\temp)\mid \tau_\temp\in \xib^{-1}(\pi_\temp)\}.
\]

\subsubsection{} \label{sss tpl}
An $L$-packet of tempered representations in $\Irr_\U$ either consists entirely of discrete series representations or contains none (see \cite[Theorem 5.7]{MR2366373}). M{\oe}glin further explicated in  \cite[\S 5]{MR2366373} the $L$-packets consisting of discrete series representations. 

A tempered representation $\pi\in \Irr_\GL$ is called \emph{stable $F$-discrete} if it has the form 
\begin{equation}\label{eq fd}
\pi\cong\delta_1\times\cdots\times\delta_k
\end{equation}
where $\delta_1,\dots,\delta_k\in\Irr_\GL$ are discrete series representations such that $\rec_\GL(\delta_i)$ is conjugate symplectic for $i=1,\dots,k$ and $\delta_i\not\cong\delta_j$ for $i\ne j$.

Recall from Remark \ref{rmk csq} that for $\rho\in \Cusp_\GL$ and $a\in \N$ we have that $\rec_\GL(\delta(\rho,a))$ is conjugate symplectic if and only if $\rho$ is conjugate self-dual of parity $(-1)^a$.

For $\pi\in \Irr_\GL$ tempered and in the image of stable base change we have that $\xib^{-1}(\pi)$ consists of discrete series representations if and only if $\pi$ is stable $F$-discrete.

Let $\pi$ be stable $F$-discrete as in \eqref{eq fd} and write $\delta_i=\delta(\rho_i,a_i)$, $i=1,\dots,k$. Then the $L$-packet
$\xib^{-1}(\pi)$ consists of discrete series representation $\pi'\in \Irr_\U$ such that $\Jord(\pi')=\{(\rho_i,a_i),\mid i=1,\dots,k\}$.
It further follows from \cite[Theorem 7.1]{MR2366373} that the stable base change fiber is of cardinality 
\[
\abs{\xib^{-1}(\pi)}=2^{k-1}.
\]
This observation, that a discrete series $L$-packet is determined by the Jordan set, is a consequence of the fact that the extended cuspidal support of a discrete series $\tau\in\Irr_\U$ as defined in \cite[\S 5.4]{MR2366373} is, in fact, the cuspidal support of $\times_{(\rho,a)\in\Jord(\tau)} \delta(\rho,a)$. This, in turn, follows from the fact that the basic assumption (BA) of \cite{MR1896238} (and hence also \cite[(2-3)]{MR1896238}) is now a theorem (see \cite[Proposition 3.1]{MR2366373} and \cite[\S 12]{MR1896238}). 

\subsubsection{}\label{sss tlp} Let $\pi\in \Irr_\GL$ be tempered and in the image of stable base change. Then there is a unique stable $F$-discrete representation $\pi_0\in \Irr_\GL$ and a tempered representation $\tau\in \Irr_\GL$ such that 
\[
\pi\cong\tau\times \pi_0\times \overline{\tau}^\vee.
\]
Furthermore, $\tau\times \overline{\tau}^\vee$ is uniquely determined by $\pi$. We have 
\[
\xib^{-1}(\pi)=\{\tau'\in \Irr_\U \mid \tau'\hookrightarrow \tau\rtimes \pi_0' \text{ for some } \pi_0'\in \xib^{-1}(\pi_0)\}. 
\]
Write $\tau=\delta_1\times\cdots\times\delta_t$ where $\delta_i\in \Irr_\GL$ is a discrete series representation $i=1,\dots,t$.
In fact, it follows from \cite[Theorem 13.1]{MR1896238} (and the fact that $\Jord(\pi_0')$ is independent of $\pi_0'\in\xib^{-1}(\pi_0)$) that 
\[
\abs{\xib^{-1}(\pi)}=2^m \abs{\xib^{-1}(\pi_0)}
\] 
where $m$ is the number of pairs $(\rho,a)$ such that $\rho\in \Cusp_\GL$ is conjugate self-dual with parity $(-1)^a$, $\delta(\rho,a)\rtimes \pi_0'$ is reducible for one (and hence all) $\pi_0'\in\xib^{-1}(\pi_0)$ and $\delta(\rho,a)\cong\delta_i$ for some $i\in \{1,\dots,t\}$.
Furthermore, in the above notation, write $\delta_i=\delta(\sigma_i,b_i)$, with $\sigma_i\in \Cusp_\GL$ and $b_i\in \N$, $i=1,\dots,t$ and $\pi_0=\delta(\rho_1,a_1)\times\cdots\times\delta(\rho_k,a_k)$ as in  \eqref{eq fd}. Then $\xib^{-1}(\pi)$ consists of the tempered representations $\pi'\in\Irr_\U$ such that $\Jord(\pi')$ is the multi-set
\[
\{(\sigma_i,b_i), (\overline{\sigma_i}^\vee, b_i)\mid i=1,\dots,t\}\cup\{(\rho_i,a_i)\mid i=1,\dots,k\}.
\]

\subsubsection{}\label{sss tpl_1} Next we make more explicit the $L$-packets of discrete series representations that contain a single member.

For $\rho\in \Cusp_\GL$ even and $a\in 2\N$ let $\tau^+(\rho,a)\in \Pi_\disc^\circ$ be the representation associated in \S\ref{sss_mt_dsc} to the admissible data
$(\J,({\mbf{a}},\epsilon))$ where $\J=\{\rho\}$, ${\mbf{a}}=((a-1)/2)$ and $\epsilon=(1)$ (so that $k_\rho=1$). In other words, $\tau^+(\rho,a)$ is the unique irreducible quotient of 
\[
L([\frac{1-a}2,-\frac12]_{(\rho)})\rtimes \triv_0
\] 
and is a strongly positive discrete series representation (see for example \cite[\S 7]{MR1896238}).
\begin{lemma}\label{lem mog}
Let $\delta\in\Irr(\GL_{2n}(E))$ be a discrete series representation in the image of stable base change. Write $\delta=\delta(\rho,a)$ where $\rho\in \Cusp_\GL$ is conjugate self-dual with parity $(-1)^a$ and let $\tau\in \Irr_\U$ be the discrete series representation such that $\{\tau\}=\xib^{-1}(\delta)$.
\begin{enumerate}
\item If $a$ (and hence $\rho$) is odd then $\tau\not\in\Pi_\disc^\circ$.
\item If $a$ (and hence $\rho$) is even then $\tau=\tau^+(\rho,a)$.
\end{enumerate}
\end{lemma}
\begin{proof}
This follows from \cite[\S 5]{MR2366373}. In particular, when $a$ is odd the extended cuspidal support of $\tau$ in the sense of loc. cit. contains $\rho$ with multiplicity one while when $a$ is even, the extended cuspidal support  of $\tau^+(\rho,a)$ is the cuspidal support of $\delta$. The lemma therefore follows from \cite[Theorem 5.7]{MR2366373}.
\end{proof}

\subsubsection{} It is observed in \cite[Proposition 12]{DP} that if an $L$-packet in $\Irr_\U$ has an associated Arthur parameter and its stable base change is $\Sp$-distinguished then the $L$-packet consists of a single representation. We notice here that this observation holds for all unitary representations in $\Irr_\GL$. 

\begin{proposition}\label{prop DP}
Let $\pi\in \Irr_\GL$ be unitary, conjugate self-dual and $\Sp$-distinguished. Then $\pi$ is in the image of stable base change and its base change fiber is a singleton, i.e., $\abs{\xib^{-1}(\pi)}=1$.  
\end{proposition}
\begin{proof}
Write $\pi=L(\mult)$. In light of \S \ref{sss rtt} it is enough to show that no segment in $\mult$ has exponent zero. 
For a segment $\Delta$ and $a\in \N$ let $\mult(\Delta,a)=\{\nu^{\frac{1-a}2+i}\Delta\mid i=0,1,\dots,a-1\}$. It follows from \cite[Corollary 2]{MR2414223} that $\pi=\pi_1\times\cdots\times\pi_k$ where each $\pi_i=L(\mult_i)$ is such that $\mult_i$ is either of the form $\mult(\Delta,2a)$ or of the form $\nu^\alpha\mult(\Delta,2a)+\nu^{-\alpha}\mult(\Delta,2a)$ where $\exp(\Delta)=0$, $a\in \N$ and $0<\alpha<\frac12$. It easily follows that no segment in $\mult_i$ has exponent zero and since $\mult=\mult_1+\cdots+\mult_k$ the proposition follows. (Here, we view multi-sets as functions of finite support, hence addition is defined.)
\end{proof}

\subsubsection{}
Note that Proposition \ref{prop:im_bc} implies that for any rigid $\pi$ in the image of the map $\xib$, there exists a unique conjugate self-dual representation $\rho\in \Cusp_{\GL}$ such that either $\Supp(\pi)\subseteq \rho^{\mathbb Z}$ or $\Supp(\pi)\subseteq (\nu^{\frac{1}{2}}\rho)^{\mathbb Z}$. We make the following definition.

\begin{definition}\label{def:im_bc}
We say that a conjugate self-dual rigid representation $\pi\in \Irr_{\GL}$ is {\it centered} at $\rho\in \Cusp_{\GL}$, if $\overline{\rho}\cong \rho^{\vee}$ and $\Supp(\pi)\subset \rho^{\mathbb Z}\sqcup  (\nu^{\frac{1}{2}}\rho)^{\mathbb Z}$. 
\end{definition}

\section{Distinction and the stable base change map}\label{s_stab_dist}

Dijols and Prasad conjectured in \cite[Conjecture 2]{DP} that an $L$-packet of Arthur type in $\Irr_U$ contains an $\Sp$-distinguished member if and only if the stable base change of the $L$-packet is $\Sp$-distinguished.
In this section we study a family of representations that indicate that the Arthur type assumption is necessary. Namely, we show that for some $\Sp$-distinguished ladder representations in $\Irr_\GL$ the unique element of its base change fiber is not $\Sp$-distinguished.

\subsection{Reducibility, distinction and $L$-packets associated to ladders}
Let $\pi=L(\mult)$ be a ladder representation. We recall the equivalent conditions for $\pi$ to be $\Sp$-distinguished and for $\pi$ to be in the image of stable base change. 

\subsubsection{}\label{sss spd} It follows from \cite[Theorem 10.3]{MR3628792} that the representation $\pi$ is $\Sp$-distinguished if and only if the ladder $\mathfrak m$ is of the form $(\Delta_{1},\Delta_{2},\dots,\Delta_{2s-1},\Delta_{2s})$ where $\Delta_{2i-1}=\nu\Delta_{2i}$, $i=1,\dots,s$.

\subsubsection{}\label{sss lbcf} Assume that $\pi$ is conjugate self dual and write $\mult=(\Delta_1,\dots,\Delta_t)$ so that $\Delta_{t+1-i}=\overline{\Delta_i}^\vee$, $i=1,\dots,t$. Then according to \S \ref{ss sbc} (and in particular Lemma \ref{lem mog}) 
$\pi$ is in the image of stable base change if and only if either $t$ is even or $\rec_\GL(L(\Delta_{[(t+1)/2]}))$ is conjugate symplectic. When this is the case, $\pi=\xib(\tau)$ for a unique $\tau\in \Irr_\U$ that is determined as follows.
We have
\begin{equation}\label{eq tau}
\tau=\begin{cases} \LQ(L(\Delta_1) \times \cdots \times L(\Delta_{t/2}) \rtimes \triv_0) & t\text{ is even} \\  \LQ(L(\Delta_1) \times \cdots \times L(\Delta_{(t-1)/2}) \rtimes \sigma) & t\text{ is odd} \end{cases}
\end{equation}
where for $t$ odd we have $\{\sigma\}=\xib^{-1}(L(\Delta_{[(t+1)/2]}))$. Recall that by Lemma \ref{lem mog} if $\ell(\Delta_{[(t+1)/2]})$ is odd then $\sigma$ has a non-trivial partial cuspidal support while if 
$\ell(\Delta_{[(t+1)/2]})$ is even then $\sigma=\tau^+(\rho,a)$ where $\Delta_{[(t+1)/2]}=[(1-a)/2,(a-1)/2]_{(\rho)}$.

 \subsection{The case of an even ladder} In this section we study the relation between $\Sp$-distinction of a ladder representation and of its stable base change fiber in the case that the ladder has even number of segments.
  
\subsubsection{An auxiliary result on irreducibility} Define 
\[
\mathcal S=\{\rho\in \Cusp_{\GL}\mid \rho \rtimes \triv_0\ {\rm is \ reducible}\}.
\]
We state below a special case of \cite[Theorems 1.1 and 1.2]{1703.09475} that we are going to use. Here and otherwise, for a multi-set $\mathfrak n=\{\Delta_1,\dots,\Delta_t\}$ of segments let $\mathfrak n_{>0}$ be the multi-set defined by
\[
\mathfrak n_{>0} =\{\Delta_i\mid 1\le i\le t,\ \exp(\Delta_i)>0\}. 
\]

\begin{theorem}\label{thm:l_t}
Let $\pi=L(\mathfrak m)$ be a ladder representation. Then $\pi\rtimes \triv_0$ is irreducible if and only if $\Supp(\pi)\cap \mathcal S=\emptyset$ and $L(\mathfrak m_{>0})\times L((\overline{\mathfrak m}^{\vee})_{>0})$ is irreducible.
\qed
\end{theorem}

\subsubsection{} In \eqref{eq tau} (and in the notation of \S \ref{sss lbcf}) in the case that $t$ is even, we express $\tau$, the unique member of the stable base change fiber of the ladder representation $\pi$ as the Langlands quotient of a representation induced from a ladder on the Siegel parabolic. The following corollary tells us when this induced representation is irreducible. 
\begin{corollary}\label{corr:irrd}
Let $\mathfrak m=(\Delta_{1},\dots,\Delta_{t})$ be a ladder and let $\pi=L(\mathfrak m)$. Moreover, suppose that $\pi^{\vee}=\overline{\pi}$. Let $k=[\frac{t+1}{2}]$, $1\leq k'\leq k$, $\mathfrak n=(\Delta_{k'},\dots,\Delta_{k})$ and $\pi'=L(\mathfrak n)$. Then the following are equivalent 
\begin{enumerate}
\item $\pi \rtimes \triv_0$ is irreducible
\item $\Delta_k\cap \mathcal S$ is empty 
\item $\pi'\rtimes \triv_0$ is irreducible. 
\end{enumerate}  
\end{corollary}
\begin{proof}
 By assumption $\mult=\overline{\mathfrak m}^{\vee}$. Since $\mult$ is a ladder this means that $\Delta_i^\vee=\overline{\Delta_{t+1-i}}$, $i=1,\dots,t$.
 It therefore follows from \cite[Theorem 3.1]{MR1266747} (see \S \ref{ss_l_fn} above) that the following conditions are equivalent
 \begin{itemize}
\item $\Supp(\pi)\cap \mathcal S$ is empty
 \item $\Delta_k\cap \mathcal S$ is empty 
\item $\Supp(\pi')\cap \mathcal S$ is empty.
\end{itemize}  
It follows from  \cite[Proposition 6.20 and Lemma 6.21]{MR3573961} that $L(\mathfrak m_{>0})\times L(\mathfrak m_{>0})$ is irreducible. Thus, by Theorem \ref{thm:l_t} we have that $\pi \rtimes \triv_0$ is irreducible if and only if $\Supp(\pi)\cap \mathcal S$ is empty.

Note further that $\mathfrak{n}$ is a ladder and $(\overline{\mathfrak n}^{\vee})_{>0}$ is the empty set. It therefore follows from Theorem \ref{thm:l_t} that $\pi'\rtimes \triv_0$ is irreducible if and only if $\Supp(\pi')\cap \mathcal S$ is empty. The corollary follows.
\end{proof}

\subsubsection{}
In the case of a ladder with even number of segments, $\Sp$-distinction is preserved by the stable base change map as the following result shows.
\begin{proposition}\label{prop:bc_dist_1}
Let $\pi\in\Irr_\GL$ be a ladder representation that is conjugate self-dual. Suppose that $\pi=L(\mathfrak m)$ where $|\mathfrak m|$ is an even integer. Let $\tau\in \Irr_\U$ be the unique representation such that $\xib(\tau)=\pi$ (see \eqref{eq tau}). If $\tau$ is $\Sp$-distinguished then $\pi$ is $\Sp$-distinguished.
\end{proposition}
\begin{proof}
Write $\pi=L(\mathfrak m)$ where $\mathfrak m=(\Delta_{1},\dots,\Delta_{2s})$ such that $\Delta_{i}^{\vee}=\overline{\Delta_{2s-i+1}}$, $i=1,\dots,2s$. 
By \eqref{eq tau} $\tau$ is the unique irreducible quotient of the representation
\[
L(\Delta_{1})\times\cdots\times L(\Delta_{s})\rtimes \triv_0.
\]
This induced representation is $\Sp$-distinguished since $\tau$ is. Thus it follows from Theorem \ref{thm geoml} that $\sigma=L(\Delta_{1}) \otimes \cdots \otimes L(\Delta_{s})\otimes \triv_0$ admits a relevant orbit. We use the notation of \S \ref{ss gen orb} for this orbit except that (since the symbol $\tau$ is already taken) we denote by $\eta$ the involution on $\Index$ associated with the contributing orbit. Applying \cite[\S9.5]{MR584084} write   
\[
r_{L,M}(\sigma)=\otimes_{\imath\in\Index} L(\Delta_{\imath})\otimes \triv_0.
\] 
Let $\Delta_{i}=[a_{i},b_{i}]_{(\rho)}$, $i=1,\dots,2s$ for some $\rho\in \Cusp_{\GL}$ conjugate self-dual. By assumption, we have $(a_i+b_i)/2=\exp(\Delta_{i})>0$, $i=1,\dots,s$. Therefore, $0<a_s+b_s\le a_s+b_{s-1}-1$ (the second inequality follows from the definition of a ladder). That is, $1-b_{s-1}<a_{s}$, and therefore
\[
\nu^{1-b_{i}}\rho\notin \cup_{j=1}^{s}\Delta_{j}
\] 
for any $i=1,\dots,s-1$. Recalling that 
\begin{equation*}
\Delta_{(i,1)}=[x_{i},b_{i}]_{(\rho)}
\end{equation*} 
for some $x_i\le b_i$, $i=1,\dots,s$, by Corollary \ref{cor dist adm} and \eqref{eq: c form} it follows that
\begin{equation}\label{eq:even}
\{(i,j)\in \Index \mid i=1,\dots,s-1, j=1,\dots, j_{i}\}\subset \set.
\end{equation}  
\begin{claim}\label{claim:even}
If $2i-1<s$ then $j_{2i-1}=j_{2i}=1$ and $\eta(2i-1,1)=(2i,1)$.
\end{claim}
We prove the claim by induction on $i$. Suppose that the claim is true for $1,2,\dots,i-1$. For the induction step, let $\eta(2i-1,1)=(i',j')$. It follows from the induction hypothesis that $i'\geq 2i-1$. Note that since the representation $L(\Delta)$ is generic it is not $\Sp$-distinguished for any segment $\Delta$ (see \cite[Theorem 3.2.2]{MR1078382}). It therefore follows from Corollary \ref{cor dist adm} and \eqref{orbit_cond} that in fact $i'>2i-1$ and we have $\Delta_{(i',j')}=\nu^{-1}\Delta_{(2i-1,1)}$. In particular, $\nu^{b_{2i-1}-1}\rho\in\Delta_{(i',j')}\subseteq \Delta_{i'}$. Since, $(\Delta_{1},\dots,\Delta_{s})$ is a ladder this implies that $(i',j')=(2i,1)$. By the induction hypothesis, \eqref{orbit_cond} and \cite[Theorem 3.2.2]{MR1078382}, it follows that $j_{2i-1}=1$. By \eqref{eq: c form}, \eqref{orbit_cond} and the induction hypothesis it follows that $j_{2i}=1$. The claim follows.

If $s$ is even then by \eqref{eq:even}, Claim \ref{claim:even} and Corollary \ref{cor dist adm} the ladder $(\Delta_{1},\dots,\Delta_{s})$, and hence $\mathfrak m$, is of the form described in \S \ref{sss spd} and thus the proposition follows. So suppose now that $s$ is odd. As above the ladder $(\Delta_{1},\dots,\Delta_{s-1})$ is of the form described in \S \ref{sss spd}. By Claim \ref{claim:even}, \eqref{eq: c form} and \eqref{orbit_cond} we get that $j_{s}\leq 2$ and $\eta(s,i)=(s,i)$, $i=1,\dots,j_{s}$. By Corollary \ref{cor dist adm} and \cite[Theorem 3.2.2]{MR1078382} we further get that $j_{s}=1$, $(s,1)\in \Index \setminus \set$ and $\nu^{-1/2}L(\Delta_{s})$ is $\GL(F)$-distinguished. By \cite[Proposition 12]{MR1111204} $\exp(\Delta_{s})=\frac{1}{2}$. Coupled with the fact that $\Delta_{s}^{\vee}\cong \overline{\Delta_{s+1}}$, we get that $\Delta_{s}=\nu\Delta_{s+1}$, and the multi-set $\mathfrak m$ is again of the form described in \S \ref{sss spd}. This proves the proposition.

\end{proof}

\subsubsection{}
The converse of  Proposition \ref{prop:bc_dist_1} is not true and we have the following result in that direction.
\begin{proposition}\label{prop:bc_dist}
Let $\pi\in\Irr_\GL$ be a ladder representation that is conjugate self-dual and $\Sp$-distinguished. Write $\pi=L(\mathfrak m)$ and $\mult=(\Delta_1,\Delta_2,\dots,\Delta_{2s-1},\Delta_{2s})$ as in \S \ref{sss spd}. 
Let $\tau'=L(\Delta_1,\dots,\Delta_s)\rtimes \triv_0$ and let $\tau\in \Irr_\U$ be the unique representation such that $\xib(\tau)=\pi$ (see \eqref{eq tau}).
We have
\begin{enumerate}
\item $\tau$ is the unique irreducible quotient of $\tau'$ and $\tau=\tau'$ if and only if $\Delta_s\cap \mathcal{S}$ is empty. 
\item If $s$ is even then $\tau'$ is $\Sp$-distinguished.
\item If $s$ is odd and $\Delta_s\cap \mathcal{S}$ is empty then $\tau$ is not $\Sp$-distinguished.
\end{enumerate}
\end{proposition}
\begin{proof}
The first part is immediate from \S \ref{ss sbc} and Corollary \ref{corr:irrd}. For the second part, it follows from \S \ref{sss spd} that if $s$ is even then $L(\Delta_1,\dots,\Delta_s)$ is $\Sp$-distinguished and therefore from Lemma \ref{hered} that $\tau'$ is $\Sp$-distinguished. 

Suppose now that $s=2k+1$ for some $k\in \mathbb Z_{\geq 0}$ and $\Delta_s\cap \mathcal{S}$ is empty, and assume by contradiction that $\tau$ is $\Sp$-distinguished. By Corollary \ref{corr:irrd} (with $k'=s$) we have that $L(\Delta_{2k+1})\rtimes\triv_0$ is irreducible and by Proposition \ref{prop_basic_ind} we have that
\[
L(\Delta_{2k+1})\rtimes\triv_0 \cong  L(\overline{\Delta_{2k+1}})^{\vee}\rtimes\triv_0.
\]
Since $\pi$ is conjugate self-dual we have $\overline{\Delta_{2k+1}}^\vee=\Delta_{2k+2}$ and therefore
\[
L(\Delta_{2k+1})\rtimes\triv_0 \cong L(\Delta_{2k+2})\rtimes\triv_0.
\]
Thus $\tau$ is the unique irreducible quotient of 
\[
L(\Delta_{1}) \times \cdots \times L(\Delta_{2k})\times L(\Delta_{2k+2})\rtimes\triv_0. 
\]
Since $L(\Delta_{1},\dots,\Delta_{2k},\Delta_{2k+2})$ is the unique irreducible quotient of the representation $L(\Delta_{1}) \times \cdots \times L(\Delta_{2k})\times L(\Delta_{2k+2})$, we have that $\tau$ is the unique irreducible quotient of 
\[
 L(\Delta_{1},\dots,\Delta_{2k},\Delta_{2k+2})\rtimes\triv_0. 
\]

Since by assumption we have that $\Delta_{2k+1}=\nu\Delta_{2k+2}$ as well as $\Delta_{2k+1}=\overline{\Delta_{2k+2}}^{\vee}$ we deduce that $\exp(\Delta_{2k+1})=\frac{1}{2}$ and $\exp(\Delta_{2k+2})=-\frac{1}{2}$. Let $\rho\in \Cusp_{\GL}$ be such that $\pi$ is centered at $\rho$ (see Definition \ref{def:im_bc}) and $a\in \frac{1}{2}\mathbb Z$ such that $\Delta_{2k+1}=[-a,1+a]_{(\rho)}$. Then $\Delta_{2k+2}=[-(1+a),a]_{(\rho)}$. Note that $a\ge -1/2$ (otherwise $\Delta_{2k+1}$ would be empty).

Let us first treat the case when $k>0$. Let $\mathfrak m'=(\Delta_{1}',\dots,\Delta_{l}')$ denote the multi-set $\{\Delta_{1},\dots,\Delta_{2k},\Delta_{2k+2}\}^{t}$ ordered such that $e(\Delta_{1}') \leq \cdots\leq e(\Delta_{l}')$. We then have
\[
\tilde{\zeta}(\mathfrak m') \cong Z(\Delta_{1}')\times\cdots\times Z(\Delta_{l}').
\]
Since  $L(\Delta_{1},\dots,\Delta_{2k},\Delta_{2k+2})=Z(\mult')$ is the unique irreducible quotient of $\tilde{\zeta}(\mathfrak m')$ we deduce that
 $\tau$ is an irreducible quotient of $\tilde{\zeta}(\mathfrak m') \rtimes \triv_0$. 

By the M\oe{}glin-Waldspurger algorithm for ladder representations (provided in \cite[\S 3]{MR3163355}) it follows that $\Delta_{1}'=\{\nu^{-(1+a)}\rho\}$ is a segment of length one. Since the representation $\tilde{\zeta}(\mathfrak m')\rtimes \triv_0$ has $\tau$ as an irreducible quotient, it is $\Sp$-distinguished. 

We now apply the geometric lemma for $\tilde{\zeta}(\mathfrak m')\rtimes \triv_0$ viewed as induced from 
\[
\sigma=Z(\Delta_{1}')\otimes\cdots\otimes Z(\Delta_{l}')\otimes \triv_0.
\]
It follows from Theorem \ref{thm geoml} that $\sigma$ admits a relevant orbit. We use the notation of \S \ref{ss gen orb} for this orbit except that (since the symbol $\tau$ is already taken) we denote by $\eta$ the involution on $\Index$ associated with the contributing orbit.

Applying \cite[Proposition 3.4]{MR584084} write $r_{L,M}(\sigma)=\otimes_{\imath\in\Index} Z(\Delta_{\imath}')\otimes \triv_0$ and note that $j_1=1$ and $\Delta_{(1,1)}'=\Delta_1'=\{\nu^{-(1+a)}\rho\}$. If $(1,1)\in\set$, then by Lemma \ref{aux_geo_lem2} we have that $\eta(1,1)=(1,1)$ and by the first part of Corollary \ref{cor dist adm} we deduce that $Z(\Delta_{(1,1)}')=\nu^{-(1+a)}\rho\in \Cusp_\GL$ is $\Sp$-distinguished. This contradicts \cite[Theorem 3.2.2]{MR1078382}. 

Thus, $(1,1)\in\Index \setminus\set$. Since $\exp(Z(\Delta_{(1,1)}'))=-(1+a)\le -1/2$ the representation $\nu^{-1/2}Z(\Delta_{(1,1)}')$ is not $\GL(F)$-distinguished. It therefore follows from Corollary \ref{cor dist adm} that $\eta(1,1)\ne (1,1)$ and that $Z(\Delta_{\eta(1,1)})=\nu^{2+a}\rho$, and from \eqref{orbit_cond} that $\eta(1,1)=(\alpha,1)$ for some $\alpha\in \{2,\dots, l\}$. Let $e(\Delta_{2k})=\nu^{b}\rho$. Since $(\Delta_{1},\dots,\Delta_{2k+1})$ is a ladder we have that $b\geq 2+a$. Thus from the M\oe{}glin-Waldspurger algorithm it follows that $\Delta_{\alpha}'\in \{\Delta_{1},\dots,\Delta_{2k}\}^{t}$. By \cite[Theorem 10.3]{MR3628792}, the representation $L(\Delta_{1},\dots,\Delta_{2k})$ is $\Sp$-distinguished. By \cite[Proposition 7.5]{MR3628792} we get that $\ell(\Delta_{\alpha}')$ is even and in particular $j_\alpha\ge 2$. Applying \eqref{eq: c form}, Lemma \ref{aux_geo_lem2} and \cite[Proposition 7.5]{MR3628792}, we get that $(\alpha,2)\in \Index \setminus\set$. Note that $b(\Delta_{(\alpha,2)}')=\nu^{a+3}\rho$. Since $\nu^{-(2+a)}\rho\notin \Supp (\mathfrak m')$ this contradicts Corollary \ref{cor dist adm}. Thus we get that when $k>0$, $\tilde{\zeta}(\mathfrak m')\rtimes \triv_0$ and hence $\tau$ cannot be $\Sp$-distinguished.

Assume now that $k=0$. In this case $\tau$ is a quotient of $L(\Delta_{2})\rtimes \triv_0$ which is then $\Sp$-distinguished. As observed above, we have $\exp(\Delta_{2})=-\frac{1}{2}$. On the other hand, by Lemma \ref{lem:first_case} we get that $\exp(\Delta_{2})=\frac{1}{2}$ which gives us a contradiction. Thus $\tau$ cannot be $\Sp$-distinguished. This completes the proof of the proposition.

\end{proof}

\subsection{The case of an odd ladder}
Next we consider a ladder representation $\pi$ of the form $L(\Delta_{1},\dots,\Delta_{2k+1})$ in the image of the map $\xib$. 
In many cases we show below that the unique $\tau\in\Irr_\U$ such that $\xib(\tau)=\pi$ is not $\Sp$-distinguished.

Recall from \S \ref{sss spd} that $\pi$ is not $\Sp$-distinguished, and from \S \ref{ss sbc} that $\Delta_{2k+2-i}=\overline{\Delta_i}^\vee$, $i=1,\dots,2k+1$ and $L(\Delta_{k+1})=\delta(\rho,a)$ for some $\rho\in \Cusp_\GL$ and $a\in \N$ so that $\rho$ is conjugate self-dual of parity $(-1)^a$.

\begin{proposition}\label{prop:odd2} 
Let $\pi=L(\Delta_{1},\dots,\Delta_{2k+1})\in \Irr_{\GL}$ be a ladder representation in the image of $\xib$. Write $L(\Delta_{k+1})=\delta(\rho,a)$ as above and let $\tau\in \Irr_\U$ be such that $\{\tau\}=\xib^{-1}(\pi)$. If either $a$ is odd, $k=0$ or $b(\Delta_{k})\ne \nu^{3/2}\rho$ (this inequality holds, in particularly, if $b(\Delta_k)=\nu \,b(\Delta_{k+1})$) then $\tau$ is not $\Sp$-distinguished.
\end{proposition}
\begin{proof}
If $a$ is odd then it follows from \eqref{eq tau} that $\tau$ has non-trivial partial cuspidal support. Thus, $\tau$ is not $\Sp$-distinguished by Proposition \ref{prop dp}. 

Assume that $a$ is even.
It follows from \eqref{eq tau} that $\tau$ is the unique irreducible quotient of $L(\Delta_{1},\dots, \Delta_{k})\rtimes \tau^+(\rho,a)$. Let $b=(a-1)/2$ and recall further that $\tau^+(\rho,a)$ is the unique irreducible quotient of $ L([-b,-\frac{1}{2}]_{(\rho)})\rtimes \triv_0$ (see \S \ref{sss tpl_1}). This implies that $\tau$ is an irreducible quotient of the representation 
\begin{equation}\label{eq:odd2}
L(\Delta_{1},\dots, \Delta_{k})\times  L([-b,-1/2]_{(\rho)}) \rtimes \triv_0.
\end{equation}
To complete the proof of the proposition it is enough to prove that this induced representation is not $\Sp$-distinguished. If $k=0$ this is immediate from Lemma \ref{lem:first_case}.
Assume that $k>0$ and assume by contradiction that the representation in (\ref{eq:odd2}) is $\Sp$-distinguished. 

It follows from Theorem \ref{thm geoml} that $\sigma=L(\Delta_{1},\dots, \Delta_{k})\otimes  L([-b,-1/2]_{(\rho)}) \otimes \triv_0$ admits a relevant orbit. We use the notation of \S \ref{ss gen orb} for this orbit except that (since the symbol $\tau$ is already taken) we denote by $\eta$ the involution on $\Index$ associated with the contributing orbit. Let $\sigma'=\otimes_{(i,j)\in\Index} \pi_{i,j}'\otimes \triv_0$ be an irreducible subquotient of $r_{L,M}(\sigma)$ that admits the corresponding non-trivial  invariant functional (i.e., $\sigma'$ is $(L_x,\delta_x)$-distinguished).  

\begin{claim}\label{claim1}
If $(1,1)\notin \set$, then $\eta(1,1)\neq (1,1)$.
\end{claim}
Observe that $\exp(\Delta_{i})\geq 1$, $i=1,\dots,k$. By \cite[Theorem 2.1]{MR2996769}, we get that $\pi_{1,1}$ is a (ladder) representation of the form $L(\Delta_{1}',\dots,\Delta_{l}')$ where $\{d_1,\dots,d_l\}$ is a subset of $\{1,\dots,k\}$ such that $\Delta_i'\subseteq \Delta_{d_i}$ and $e(\Delta_i')=e(\Delta_{d_i})$, $i=1,\dots,l$. In particular, $\exp(\Delta_{i}')\geq 1$, $i=1,\dots,l$ and therefore also $\exp(\pi_{1,1})\geq 1$. Thus $\nu^{-1/2}\pi_{1,1}$ cannot be $\GL(F)$-distinguished (by \cite[Proposition 12]{MR1111204}) and Claim \ref{claim1} follows from Corollary \ref{cor dist adm}.  

\begin{claim}\label{claim2} 
We have $(2,j)\in \set$ for all $j=1,\dots,j_2$.
\end{claim}
By \eqref{eq: c form} it is enough to show that $(2,1)\in \set$.
Assume by contradiction that $(2,1)\in \Index\setminus \set$. It follows from \cite[Proposition 9.5]{MR584084} that $\pi_{2,1}$ is of the form $L([x,-\frac{1}{2}]_{(\rho)})$ where $-b\leq x\leq -\frac{1}{2}$ and in particular $\exp(\pi_{2,j})< 0$. As in Claim \ref{claim1}, it therefore follows from Corollary \ref{cor dist adm} and \cite[Proposition 12]{MR1111204} that $\eta(2,1)\ne (2,1)$. 

By (\ref{orbit_cond}) and Claim \ref{claim1} we have $\eta(2,1)=(1,1)$. It therefore follows from Corollary \ref{cor dist adm} that  $\pi_{1,1}=L([\frac{3}{2},1-x]_{(\rho)})$. By assumption, this means that $b(\Delta_k)\le \nu^{1/2}\rho$ and in particular $j_1>1$. 

In the notation of the proof of Claim \ref{claim1} we now have $l=1$ and $\exp(e(\Delta_{d_1}))=1-x\le 1+b$. Since $\exp(e(\Delta_i))>1+b$ for $i=1,\dots,k-1$ and $\exp(e(\Delta_k))\ge 1+b$ it follows that $d_1=k$ and $x=-b$. This forces $j_{2}=1$ and  $\pi_{2,1}=L([-b,-\frac{1}{2}]_{(\rho)})$. 

Moreover, from (\ref{orbit_cond}) we get that $j_{1}=2$, $(1,2)\in \set$ and $\eta(1,2)=(1,2)$. From \cite[Theorem 2.1]{MR2996769} we get that $\pi_{1,2}=L(\Delta_{1},\dots,\Delta_{k-1}, [y,\frac{1}{2}]_{(\rho)})$ where $b(\Delta_k)=\nu^y\rho$. It further follows from Corollary \ref{cor dist adm} that $\pi_{1,2}$ is $\Sp$-distinguished. By \S \ref{sss spd} we get that $e(\Delta_{k-1})=\nu^{3/2}\rho$. This contradicts the facts that $(\Delta_{k-1},\Delta_k,\Delta_{k+1})$ is a ladder, $e(\Delta_{k+1})=\nu^b\rho$ and $b\ge 1/2$. This proves Claim \ref{claim2}.

Note that Claims \ref{claim1} and \ref{claim2} together with (\ref{orbit_cond}) imply that $\Index = \set$. Since no essentially discrete series representation in $\Irr_{\GL}$ can be $\Sp$-distinguished (\cite[Theorem 3.2.2]{MR1078382}), it further follows that $j_{2}=1$, $(2,1)\in \set$, and $\eta(2,1)= (1,1)$. Thus, by Corollary \ref{cor dist adm} we have $\pi_{1,1}=L([1-b,\frac{1}{2}]_{(\rho)})$. This means again that $l=1$ and $e(\Delta_{d_1})=\nu^{1/2}\rho$ which is a contradiction (since $e(\Delta_i)\ge \nu^{1+b}\rho$, $i=1,\dots,k$). The proposition is now proved.
\end{proof}

\subsubsection{}
We summarize the results obtained in this section for the class of Speh representations in the image of the map $\xib$. By \S \ref{sss spd} a Speh representation $\pi=L(\mathfrak m)$ is $\Sp$-distinguished if and only if $|\mathfrak m|$ is an even integer. The following is a consequence of Propositions \ref{prop:bc_dist} and \ref{prop:odd2}.
\begin{theorem}\label{thm:speh_bc}
Let $\pi=L(\mathfrak m)\in \Irr_{\GL}$ be a Speh representation contained in the image of the stable base change map $\xib$ and let $\tau\in \Irr_\U$ be such that $\{\tau\}=\xib^{-1}(\pi)$. Then we have the following:
\begin{enumerate}
\item Suppose that $\pi$ is not $\Sp$-distinguished. Then $\tau$ is not $\Sp$-distinguished.
\item Suppose that $\pi$ is $\Sp$-distinguished and write $\mult=(\nu^{2k-1}\Delta,\nu^{2k-2}\Delta,\dots,\Delta)$ for a segment $\Delta$ and $k\in \N$. Suppose further that $\nu^k\Delta\cap \mathcal{S}$ is empty. Then $\tau$ is $\Sp$-distinguished if and only if $k$ is even.  
\end{enumerate}\qed
\end{theorem}

\begin{remark}\label{rem_dp_conj}
While part (1) of the theorem supports \cite[Conjecture 2]{DP}, part (2) does not contradict it. In the notation of part (2) of the theorem, if $k$ is odd and $\nu^k\Delta\cap \mathcal{S}$ is empty then $\tau$ does not have an Arthur parameter. We thank Dipendra Prasad for explaining to us this point. 
\end{remark}

\section{Distinction for standard modules with an irreducible $\GL$-part}\label{s_std_mod}
In this section we characterize distinction of standard modules with a generic $\GL$-part. 
\subsection{}\label{ss zel}
Let $\pi=L(\Delta_1,\dots,\Delta_t)\in\Irr_\GL$. By \cite[Theorem 9.7]{MR584084} the representation $\pi$ is generic if and only if $\pi$ is a standard module. That is, 
\[
\pi\cong L(\Delta_1)\times\cdots\times L(\Delta_t)
\]
(for some and therefore any order on the segments).

\subsection{}\label{ss matringe}  Let $\pi=L(\Delta_1,\dots,\Delta_t)\in\Irr_\GL$ be generic. By \cite[Theorem 5.2]{MR2755483} the representation $\pi$ is $\GL(F)$-distinguished if and only if there exists an involution $w\in S_t$ such that 
\begin{itemize}
\item $\Delta_{w(i)}=\overline{\Delta_i}^\vee$, $i=1,\dots,t$ and
\item $L(\Delta_i)$ is $\GL(F)$-distinguished if $w(i)=i$.
\end{itemize}
In fact, we will only apply the `if' part of this statement, which is a consequence of \cite[Theorem 4.2]{MR2592729} and \cite[Proposition 26]{MR1194271}.

\subsection{}The main result of the section is as follows.
\begin{theorem}\label{thm:gen_lq}
Let $\Delta_1,\dots,\Delta_t$ be segments such that $\exp(\Delta_1)\ge \cdots \ge \exp(\Delta_t)>0$ and $\pi=L(\Delta_1,\dots,\Delta_t)\in\Irr_\GL$ is generic. Then, the standard module $\pi\rtimes \triv_0$ is $\Sp$-distinguished if and only if $\nu^{-1/2}\pi$ is $\GL(F)$-distinguished. 
\end{theorem}
\begin{proof}
If $\nu^{-1/2}\pi$ is a $\GL(F)$-distinguished representation, then $\pi\rtimes \triv_0$ is $\Sp$-distinguished by Lemma \ref{hered}. Suppose that $\pi\rtimes \triv_0$ is $\Sp$-distinguished. As remarked in \S \ref{ss zel}, we may rearrange the segments $\Delta_1,\dots,\Delta_t$ in any order. Writing $\Delta_{i}=[a_{i},b_{i}]_{(\rho_{i})}$ where $\rho_{i}\in \Cusp_{\GL}$ is unitary we may assume that the following conditions are satisfied:
\begin{equation}\label{eq ord}
a_{1}\leq \cdots \leq a_{t} \ \ \  \text{and if} \ \ \ a_{i}=a_{i+1}\ \ \ \text{then}\ \ \ b_{i}\geq b_{i+1}. 
\end{equation}
It follows from Theorem \ref{thm geoml} that $\sigma=L(\Delta_{1})\otimes \cdots \otimes L(\Delta_{t})\otimes \triv_0$ admits a relevant orbit. We use the notation of \S \ref{ss gen orb} for this orbit. 
Applying \cite[\S9.5]{MR584084} write   
\[
r_{L,M}(\sigma)=\otimes_{\imath\in\Index} L(\Delta_{\imath})\otimes \triv_0
\] 
and recall that 
\[
\Delta_{(i,1)}=[x_i,b_i]_{(\rho_i)} \ \ \ \text{and}  \ \ \ \Delta_{(i,j_i)}=[a_i,y_i]_{(\rho_i)}
\] 
for some $a_i\le x_i\le y_i\le b_i$, $i=1,\dots,t$. 

\begin{claim}\label{claim_gen1}
The set $\set$ is empty.
\end{claim}
Assume by contradiction that $\set$ is not empty and let $i\in \{1,\dots,t\}$ be minimal such that $(i,j)\in \set$ for some $j$. By (\ref{eq: c form}) we have $(i,j_{i})\in \set$. Let $\tau(i,j_{i})=(i',j')$. Since no essentially square integrable representation is $\Sp$-distinguished (\cite[Theorem 3.2.2]{MR1078382}), it follows from Corollary \ref{cor dist adm} that $(i',j')\neq (i,j_{i})$. By minimality of $i$ and \eqref{orbit_cond} we have $i'>i$ and it further follows from  Corollary \ref{cor dist adm} that
\[
\nu^{-1}L(\Delta_{(i,j_{i})})=L(\Delta_{(i',j')}).
\]
In particular, $a_{i'}\le a_i-1$ which contradicts \eqref{eq ord}. This proves Claim \ref{claim_gen1}.

\begin{claim}\label{claim_gen2}
We have $j_{i}=1$ for every $i\in \{1,\dots, t\}$.
\end{claim}
To prove the claim we show, by induction on $k$, that for all $1\leq i \leq k$, we have:
\begin{enumerate}
\item $j_{i}=1$ 
\item if $\tau(i,1)=(i',j')$ then $j_{i'}=1$. 
\end{enumerate} 

Suppose that the induction hypothesis holds for $k-1$ and let $\tau(k,1)=(k',j')$. If $k'< k$, then the induction hypothesis implies the statement for $k$. Assume now that $k'\geq k$. It follows from (\ref{orbit_cond}), Claim \ref{claim_gen1} and the induction hypothesis that $j'=1$ and further from Corollary \ref{cor dist adm} that $\Delta_{(k',1)}=\nu\overline{\Delta_{(k,1)}}^\vee$.
Since $\exp(\Delta_{i})>0$ we get that $\exp(\Delta_{(i,1)})>0$ for every $i=1,\dots, t$. Since $\exp(\Delta_{(k,1)})=1-\exp(\Delta_{(k',1)})$, we have $0< \exp(\Delta_{(k,1)}), \ \exp(\Delta_{(k',1)})<1$.

If $\exp(\Delta_{(k,1)})= \exp(\Delta_{(k',1)})=\frac{1}{2}$ then we must have $j_k=j_{k'}=1$ since $\exp(\Delta_i)>0$ for all $i$. This proves the induction step in this case. 
Otherwise, i.e., if $\exp(\Delta_{(k,1)})\ne \exp(\Delta_{(k',1)})$, then the same argument shows that $j_k,\,j_{k'}\le 2$, if one of them equals $2$ the other must equal $1$, and $\ell(\Delta_i)\le \ell(\Delta_{(i,1)})+1$ for $i\in \{k,k'\}$. 

By assumption, since $\exp(\Delta_{(k,1)})\ne \frac12$, the representation $\nu^{-1/2}L(\Delta_{(k,1)})$ cannot be $\GL(F)$-distinguished and therefore, by Corollary \ref{cor dist adm} we have $k'> k$, $\Delta_{(k',1)}=[1-b_{k},b_{k'}]_{(\rho_{k'})}$
and $\overline\rho_{k}^{\vee}\cong \rho_{k'}$. 

If $j_{k'}=2$, then $\Delta_{(k',2)}=\{\nu^{-b_{k}}\rho_{k'}\}$. Let $\tau(k',2)=(l,l')$. By (\ref{orbit_cond}) we have $l>k$ and by Claim \ref{claim_gen1} and Corollary \ref{cor dist adm}, $\Delta_{(l,l')}=\{\nu^{b_{k}+1}\rho_{k}\}$. But now \eqref{eq ord} implies that $\Delta_k$ and $\Delta_{l}$  are linked which contradicts the fact that $\tau$ is generic (see \cite[Theorem 9.7]{MR584084}). Thus we conclude that $j_{k'}=1$. 

If $j_{k}=2$ then similarly $\Delta_{(k,2)}=\{\nu^{-b_{k'}}\rho_{k}\}$ while (since $j_{k'}=1$) 
\[
\Delta_{k'}=\Delta_{(k',1)}=[1-b_k,b_{k'}]_{(\rho_{k'})}.
\]
Let $\tau(k,2)=(l,l')$. By (\ref{orbit_cond}) we have that $l>k'$ and further by the induction hypothesis, that $l'=1$. By Claim \ref{claim_gen1} and Corollary \ref{cor dist adm} $\Delta_{(l,1)}=\{\nu^{b_{k'}+1}\rho_{k'}\}$. Again, \eqref{eq ord} implies that $\Delta_{l}$ and $\Delta_{k'}$ are linked which contradicts the fact that $\pi$ is generic. Thus we conclude that $j_{k}=1$ which finishes the proof of the induction. This completes the proof of Claim \ref{claim_gen2}.

The theorem is now immediate from Claims \ref{claim_gen1} and \ref{claim_gen2}, Corollary \ref{cor dist adm} and \S \ref{ss matringe}.
\end{proof}

\begin{remark}\label{rem:gen_lq}
In the notation of Theorem \ref{thm:gen_lq}, if $\nu^{-\frac12}\pi$ is $\GL(F)$-distinguished and furthermore $\exp(\Delta_i)\in\frac12\Z$, $i=1,\dots,t$ then it follows from \S \ref{ss matringe}, in its notation, that $w$ is trivial and therefore $\nu^{-\frac12}\pi$ is tempered. 

When $\pi\in \Irr(\GL_{2n}(E))$ for some $n\in \N$ is such that $\nu^{-\frac12}\pi$ is tempered and $\GL(F)$-distinguished, Morimoto proved in \cite{Mor} the stronger result that $\LQ(\pi\rtimes\triv_0)$ is $\Sp$-distinguished.
We expect this to be true more generally, for the setting of Theorem \ref{thm:gen_lq}. However, this will require methods different from the ones employed in this work.
\end{remark}

\end{document}